\documentclass[a4paper,12pt,oneside,reqno]{amsart}
\usepackage{hyperref}
\usepackage[headinclude,DIV13]{typearea}
\areaset{15.1cm}{25.0cm}
\usepackage{amsfonts,amssymb,amsmath,amsthm,bbm}
\usepackage{cancel} 
\usepackage[latin1] {inputenc}
\usepackage{graphicx, psfrag}

\newtheorem{theorem}{Theorem}[section]
\newtheorem{lemma}[theorem]{Lemma}
\newtheorem{proposition}[theorem]{Proposition}
\newtheorem{corollary}[theorem]{Corollary}

\theoremstyle{definition}
\newtheorem{definition}[theorem]{Definition}

\theoremstyle{remark}
\newtheorem*{remark}{Remark}

\def\paragraph#1{\noindent \textbf{#1}}

\numberwithin{equation}{section}

\def\d{\mathrm{d}}

\def\<{\langle}
\def\>{\rangle}
\def\a{\alpha}

\def\e{\epsilon}

\def\g{\gamma}

\def\D{\Delta}

\def\S{\Sigma}

\def\R{{\Bbb R}}  
\def\N{{\Bbb N}}  
\def\E{{\Bbb E}}  

\let\cal=\mathcal

\def\GG{{\cal G}}

\def\OO{{\Theta}}

\def \e {{\varepsilon}}

\def \D {{\Delta}}

\def \g {{\gamma}}

\def \d {{\delta}}
\def \a {{\alpha}}

\def \ba {\begin{array}}
	\def \ea {\end{array}}

\newcommand{\be}{\begin{equation}}
\newcommand{\ee}{\end{equation}}

\newcommand{\bea}{\begin{eqnarray}}
\newcommand{\eea}{\end{eqnarray}}
\def\TH(#1){\label{#1}}\def\thv(#1){\ref{#1}}
\def\Eq(#1){\label{#1}}\def\eqv(#1){(\ref{#1})}

\def\sfrac#1#2{{\textstyle{#1\over #2}}}

\def \1{\mathbbm{1}}

\usepackage{cite}
\def\BibTeX{{\rm B\kern-.05em{\sc i\kern-.025em b}\kern-.08em
		T\kern-.1667em\lower.7ex\hbox{E}\kern-.125emX}}
\usepackage[latin1]{inputenc}
\usepackage{amsmath}
\usepackage{txfonts}
\usepackage{amssymb}
\usepackage{graphicx}
\usepackage{color}
\usepackage{longtable}
\allowdisplaybreaks
\graphicspath{{pictures/},{all-with-all-pictures/}}

\title[The recovery of a recessive allele in a Mendelian diploid model]{The recovery of a recessive allele \\in a Mendelian diploid model}

\author[A. Bovier]{Anton Bovier}
\address{A. Bovier\\Institut f\"ur Angewandte Mathematik\\
Rheinische Friedrich-Wilhelms-Universit\"at\\ Endenicher Allee 60\\ 53115 Bonn, Germany}
\email{bovier@uni-bonn.de }
\author[L. Coquille]{Loren Coquille}
\address{L. Coquille\\Univ. Grenoble Alpes, CNRS, Institut Fourier, F-38000 Grenoble, France}
\email{loren.coquille@univ-grenoble-alpes.fr}
\author[R. Neukirch]{Rebecca Neukirch}
\address{R. Neukirch\\Institut f\"ur Angewandte Mathematik\\
Rheinische Friedrich-Wilhelms-Universit\"at\\ Endenicher Allee 60\\ 53115 Bonn, Germany}
\email{rebecca.neukirch@iam.uni-bonn.de}


\thanks{We acknowledge financial support from the German Research Foundation (DFG) 
	through the \emph{Hausdorff Center for  Mathematics}, the Cluster of Excellence \emph{ImmunoSensation}, and
	the Priority Programme SPP1590 \emph{Probabilistic Structures in Evolution}. L.C. has been partially supported by the LabEx PERSYVAL-Lab (ANR-11-LABX-0025-01) through the Exploratory Project \textit{CanDyPop},
	 as well as by the Swiss National Science Foundation through  the grant  No.  P300P2\_161031, and by the Chair "Modélisation
	 Mathématique et Biodiversité" of Veolia Environnement - École Polytechnique - Muséum National d'Histoire Naturelle - Fondation X.\\
	We would like to thank Pierre Collet and Vincent Beffara for their help on the theory of dynamical systems and fruitful discussions.
}

\begin{document}
	
	\begin{abstract}
		      
		We study the large population limit of a stochastic individual-based model 
		 which describes the time evolution of a diploid hermaphroditic population reproducing 
		 according to Mendelian rules.  In \cite{BovNeu} it is proved that sexual reproduction allows 
		 unfit  alleles  to  survive  in  individuals  with  mixed
		genotype much longer than they would in populations reproducing asexually. In the present
		 paper we prove that this indeed opens the possibility that
		 individuals  with  a  pure  genotype can reinvade  in  the  population  
		 after the appearance of further mutations.
		We thus expose a rigorous description of a mechanism by which a recessive allele can re-emerge in a  population. This can be seen as a statement of genetic robustness exhibited by diploid populations performing sexual reproduction.
	\end{abstract}

	\maketitle


\section{Introduction}

In  \emph{population genetics}, the study of Mendelian diploid models of fixed population size began 
more than a century ago  (see e.g. \cite{yule06,fisher18,wright31,haldane24a,haldane24b, 
crowkimura,nagylaki92,ewens04,buerger2000}), while their counterparts of variable population size 
models were studied in the context of adaptive dynamics from 1999 onwards \cite{KG99}. The 
approach of adaptive dynamics is to introduce competition kernels to regulate the population size 
instead of maintaining it constant, see \cite{HS90, ML92, MN92}.
	
	Stochastic individual-based versions of these models appeared  in the 1990s, see \cite{DL96,C_CEAD,Cha06,FouMel2004,C_ME,CM11}. They assume single events of reproduction, mutation, 
	natural death, and death by competition happen at random times to each individual in the 
	population.
	An important and interesting feature of these models is that different 
	limiting processes on different time-scales appear as the carrying capacity tends to infinity while mutation rates and mutation step-size  tend to zero (see \cite{Cha06,DL96,MG96,CM11,B14}). 
	One of the major results in this context is the convergence of a properly rescaled process to the 
	so called \emph{Trait Substitution Sequence} (TSS) process, which describes the evolution of 
	a monomorphic population as a jump process between monomorphic equilibria. 
	More generally, Champagnat and Méléard \cite{CM11} obtained the convergence to a 
	\emph{Polymorphic Evolution Sequence} (PES), where jumps occur between equilibria that may 
	include populations that have multiple co-existing phenotypes. The appearance of 
	co-existing phenotypes is, however, exceptional and happens only at so-called 
	\emph{evolutionary singularities}. 
	From a biological point of view, this is 
	somewhat unsatisfactory, as it apparently fails to explain the biodiversity seen in real biological 
	systems. 
	
	Most of the models considered  in this context  assume haploid populations with a-sexual 
	reproduction. Exceptions are the  paper \cite{CMM13} by Collet, Méléard and Metz from 2013 
	and  a series of papers by Coron and co-authors \cite{coron, coron2, coron3} following it. 
	In \cite{CMM13},  the 
	\emph{Trait Substitution Sequence}  is derived in a Mendelian diploid model under the 
	 assumption that the fitter mutant allele and the resident allele are co-dominant. 
	 
	 The main reason why both in haploid models and in the model considered in 
	 \cite{CMM13} the evolution along monomorphic populations is typical is that the time scales for 
	 the fixation of a new trait and the extinction of the resident trait are the same 
	 (both of order $\ln K$) (unless some very special fine-tuning of parameters occurs that allows 
	 for co-existence). This precludes (at least in the rare mutation scenarios considered) that 
	 an initially less fit trait survives long enough for several new mutations to occur, 
	 creating a situation where this trait may become fit again and recover.
	 
	 In a follow-up paper to \cite{CMM13}, two of the present authors \cite{BovNeu}, it was shown that, if instead 
	 one assumes that the resident allele is recessive, the time to extinction of this allele is dramatically increased. 
	  This will be discussed in detail in Section \ref{sec-previous-works} and  paves the way for the appearance of a richer limiting process.

The general framework  in \cite{CMM13} and \cite{BovNeu} is the following. Each individual is characterised by a reproduction and death rate which depend on a phenotypic trait  
determined by its genotype, which here is determined by two alleles (e.g. $A$ and $a$)
 on one single locus. The evolution of the trait distribution of the three genotypes $aa, aA$ and $AA$ is 
 studied under the action of (1) heredity, which transmits traits to new offsprings according to 
 Mendelian rules, (2) mutation, which produces variations in the trait values in the population onto 
 which selection is acting, and (3) competition for resources between individuals.

The paper \cite{BovNeu} proves that sexual reproduction allows unfit  alleles  to  survive  in  individuals  
with  mixed
genotype much longer than they would in populations reproducing asexually. 
This opens the possibility that while this allele is still alive in the population, the appearance of
new mutants alters the fitness landscape in such a way that is favourable for this allele and 
allows it to   reinvade  in  the  population, leading to a new equilibrium with co-existing phenotypes. 
The goal of this paper is to rigorously prove that such a scenario indeed occurs under fairly natural 
assumptions.\\
Recently, Billiard and Smadi \cite{BilSma15} considered related questions for haploid individuals (performing
clonal reproduction). The authors show that a deleterious allele can reinvade after a new mutation, but the range of parameters allowing this behaviour is though very small.

\subsection{The stochastic model} \label{sec-stoch-model} 
The individual-based microscopic Mendelian diploid model is a non-linear  birth-and-death process.
We consider a model for a population of a finite number of hermaphroditic individuals which reproduce 
sexually. Each individual $i$ is characterised by two alleles, $u_1^iu_2^i$, taken from some allele 
space $\mathcal{U}\subset\R$.  These two 
alleles define the genotype of the individual $i$. We suppress parental effects, which means that we identify individuals with genotype $u_1u_2$ and $u_2u_1$. 
 Each 
individual has a Mendelian reproduction rate with possible mutations and a natural death rate. Moreover, there is an 
additional death rate  due to ecological 
competition with the other individuals in the population. 
Let

\begin{center}
	\begin{longtable}{ll}
		$f_{u_1u_2}\in\R_+$ 	&\hspace{-0.2cm}be the per capita birth rate (fertility) of an individual with genotype\\ &\hspace{-0.2cm}$u_1u_2$,\\
		$D_{u_1u_2}\in\R_+$ 	&\hspace{-0.2cm}be the per capita natural death rate of an individual with genotype $u_1u_2$,\\
		$K\in\N$ 				&\hspace{-0.2cm}be the carrying capacity, a parameter which scales the population size,\\
		$\frac{c_{u_1u_2,v_1v_2}}{K}\in\R_+$ & \hspace{-0.2cm}\parbox[t]{1.0\textwidth}{be the competition effect felt by an individual with genotype $u_1u_2$ \\
			from an individual of genotype $v_1v_2$,}\\
		$R_{u_1u_2}(v_1v_2)\!\in\!\{0,\!1\} $&\hspace{-0.2cm}be the reproductive compatibility of the genotype $v_1v_2$ with $u_1u_2$,\\
		$\mu_K\in\R_+$ &\hspace{-0.2cm}be the mutation probability per birth event. Here it is independent of \\
			&\hspace{-0.2cm}the genotype,\\
		$m(u,dh)	$				&\hspace{-0.2cm}\parbox[t]{1.0\textwidth}{be the mutation law of a mutant allelic trait $u+h\in\mathcal{U}$, born from an\\individual with allelic trait $u$.} 
	\end{longtable}
\end{center}
\vspace{-10pt}
 Scaling the competition function $c$ down by a factor $1/K$ amounts to scaling the  population size to order $K$. We are interested in asymptotic  results when $K$ is  large.
 We assume rare mutations, i.e. $\mu_K\ll1$. If a mutation occurs at a birth event, only one allele changes from
 $u$ to $u+ h$ where $h$ is a random variable with law $m(u,dh)$.
			
			At any time $t$, there is  a finite number, $N_t$, of individuals, each with  genotype in $\mathcal{U}^2$. 
			We denote by $u_1^1(t)u_2^1(t),...,u_1^{N_t}(t)u_2^{N_t}(t)$ the genotypes of the population at time $t$. 
			The population, $\nu_t$,  at time $t$ is represented by the rescaled sum of Dirac measures on $\mathcal{U}^2$, 
			\begin{align}
			\nu_t=\frac{1}{K}\sum_{i=1}^{N_t}\delta_{u_1^i(t)u_2^i(t)}.
			\end{align}
			Formally, $\nu_t$ takes values in the set of re-scaled point measures 
			\begin{align}
			\mathcal{M}^K=\left\{\frac{1}{K}\sum_{i=1}^{n}
			\delta_{u_1^iu_2^i}\;\Big|\;n\geq0, u_1^1u_2^1,...,u_1^nu_2^n\in\mathcal{U}^2\right\},
			\end{align}
			on $\mathcal{U}^2$, 
			equipped with the vague topology.

For a  formal 	construction of the corresponding measure valued Markov process, 
see \cite{BovNeu}.
			We just insist on the reproduction rate of an individual of genotype $u_1u_2$ with 
			an individual of genotype $v_1v_2$, which takes the form
			$f_{u_1u_2}\frac{f_{v_1v_2}R_{u_1u_2}(v_1v_2)}{K\langle\nu R_{u_1u_2},f\rangle}$, where $\nu R_{u_1u_2}$ is the population restricted to the pool of potential partners of an individual of genotype $u_1u_2$.
			At the individual level, this form of reproduction rate means that each individual of genotype $u_1u_2$ reproduces at a rate which depends on its genotype through its fertility $f_{u_1u_2}$; more precisely, it bears an exponentially distributed clock with parameter $f_{u_1u_2}$, and when this rings, it chooses a partner at random in the pool determined by the reproductive compatibility $R$, with a weight proportional to the fertility of the partner. The reproductive compatibility can for example be thought as a way of coding if two individuals are or not of the same species. Within a species, the reproduction rate of a pair of (compatible) individuals is given by the product of their fertilities.  \\

			We make the following Assumptions \textbf{(A)}:
			\begin{itemize}
				\item[\textbf{(A1)}]The functions $f,D$ and $c$ are measurable and bounded, which means that there exist constants $\bar f,\bar D,\bar c<\infty$ such that 
				for all $u_1u_2,v_1v_2\in\mathcal{U}^2$, 
				\begin{align}
				0\leq f_{u_1u_2}\leq\bar f, \quad 0\leq D_{u_1u_2}\leq\bar D \quad \text{and}\quad 0\leq c_{u_1u_2,v_1v_2}\leq\bar c.\vspace{0pt}
				\end{align}	
				\item[\textbf{(A2)}] There exists a constant $\underline{c}>0$ such that
				for all $u_1u_2\in\mathcal{U}^2$, 
				$c_{u_1u_2,u_1u_2}\geq\underline{c}$ and
				$f_{u_1u_2}-D_{u_1u_2}>0$, 
				\item[\textbf{(A3)}] There exists a function, $\bar m:\R\rightarrow\R_+$, such that 
				$\int\bar m(h)dh<\infty$ and $m(u,h)\leq\bar m(h)$ for any $u\in\mathcal{U}$ and $h\in\R$.
				\end{itemize}
			
				For fixed $K$, under the Assumptions \textbf{(A1)+(A3)} and assuming that $
				\E(\langle\nu_0,\1\rangle)<\infty$, Fournier and Méléard \cite{FouMel2004} have shown 
				existence and uniqueness in law of the process.
				For $K\rightarrow\infty$, under mild restrictive assumptions, they prove the 
				convergence  of the process in the space 
				$\mathbb{D}(\R_+,\mathcal{M}^K)$ of càdlàg functions from $\R^+$ to
				 $\mathcal M_K$, to a deterministic process, which is the solution to a non-linear 
				 integro-differential equation. Assumption \textbf{(A2)} ensures that the population 
				 does not tend to infinity in finite time or becomes extinct too fast.

\subsection{Previous works}\label{sec-previous-works}

Consider the process starting with a monomorphic $aa$ population, with one additional mutant individual of genotype $aA$. Assume  that the phenotype difference between the mutant and the resident population is small. The phenotype difference is assumed to be a slightly smaller death rate compared to the resident population, namely
\begin{equation}\label{def-Delta}
D_{aa}=D,\quad D_{aA}=D-\Delta,
\end{equation}
for some small enough $\Delta>0$.
The mutation probability for an individual with genotype $u_1u_2$ is given by  $\mu_K$. Hence, the time 
until the next mutation in the whole population is of order $\frac{1}{K\mu_K}$. 
Now assume that the demographic parameters introduced in Section \ref{sec-stoch-model} depend continuously on the phenotype. In particular, they are the same for individuals bearing the same phenotype.

In \cite{CMM13} it is proved that if the two alleles $a$ and $A$ are co-dominant and if the allele $A$ is slightly fitter than the allele $a$, namely
\begin{equation}
D_{aa}=D,\quad D_{aA}=D-\Delta,\quad D_{AA}=D-2\Delta,
\end{equation}
 then in the limit  of large population  and  rare mutations ($\ln K\ll \frac1{\mu_K K}\ll e^{VK}$ for some $V>0$), the suitably time-rescaled process converges to the TSS model of adaptive dynamics, essentially as shown in \cite{Cha06} in the haploid case. In particular, the genotypes containing the unfit allele $a$ decay exponentially fast after the invasion of $AA$ (see Figure \ref{det_collet_vs_bovneu}). 

If in place of co-dominance we assume, as in \cite{BovNeu}, that the  fittest phenotype $A$ is dominant, namely
\begin{equation}
D_{aa}=D,\quad D_{aA}=D-\Delta,\quad D_{AA}=D-\Delta,
\end{equation}
 then this has a dramatic effect on the 
evolution of the population and, in particular, leads to a much prolonged survival of the unfit phenotype $aa$.  
Indeed, it was known for some time (see e.g. \cite{nagylaki92}) that in this case the unique stable 
fixed point $ (0,0,\bar n_{AA})$ corresponding to a monomorphic $AA$ population is degenerate, i.e. its 
Jacobian matrix has  zero-eigenvalue. This implies that in the deterministic system, the $aa$ and $aA$ 
populations decay in time  only polynomially fast to zero, namely like $1/t^2$ and $1/t$, respectively. 
This is in contrast to the exponential decay in the co-dominant scenario  (see Figure \ref{det_collet_vs_bovneu}). In \cite{BovNeu} it was shown that the deterministic system remains a 
good approximation of the stochastic system as long as the size of the $aA$ population remains
much larger than $K^{1/2}$ and therefore that the $a$ allele survives for a time of order at least 
$K^{1/2-\a}$, for any $\a>0$\footnote{The article \cite{BovNeu} only state that survival occurs up to time $K^{1/4-\a}$. However, taking into account that it is really only the survival of the $aA$ population that needs to be ensured, one can easily improve this to $K^{1/2-\a} $.}. 
Note that this statement is a non trivial fact, since it is not a consequence of the law of large numbers, because the time window diverges as $K$ grows. In summary, the unfit recessive $a$ allele survives in the population much longer due to the slow decay of the
$aA$ population. 

\begin{figure}[h!]
	\begin{center}
		\includegraphics[width=0.47\textwidth]{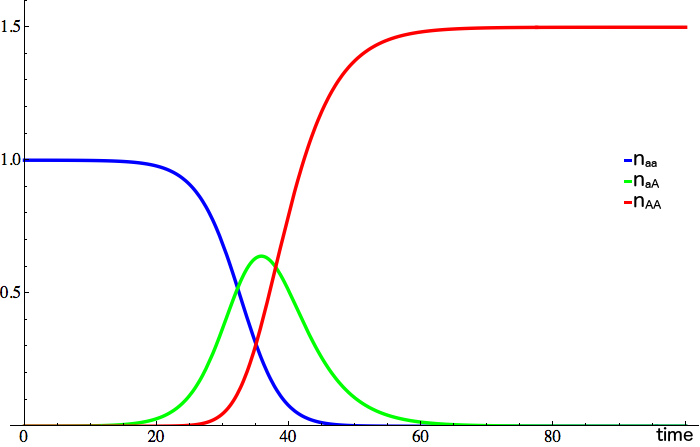}
		\includegraphics[width=0.46\textwidth]{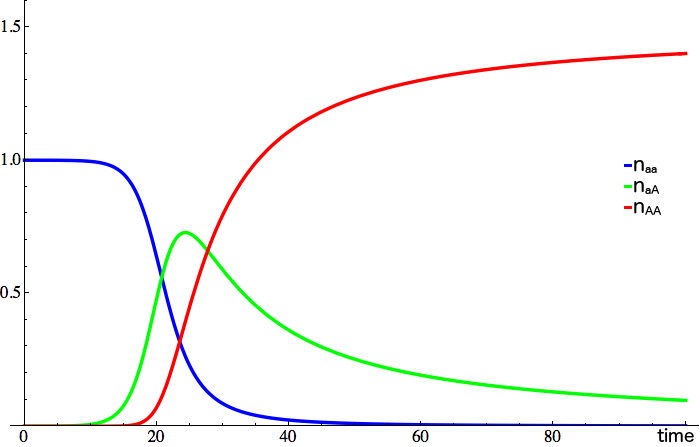}
		\caption{Evolution of the model from a resident $aa$ population at equilibrium with a small amount of mutant $aA$, and when the alleles $a$ and $A$ are co-dominant (left) or when the mutant phenotype $A$ is dominant (right).}
		\label{det_collet_vs_bovneu}
	\end{center}
\end{figure}

It is argued in \cite{BovNeu} that if we choose the mutation time scale in such a way that there remain enough $a$ alleles in the 
population when a new mutation occurs, i.e.
\begin{equation}\label{time_scale}
\ln K\ll \frac1{\mu_K K}\ll K^{1/2-\alpha}\quad\text{as }K\to\infty, \text{for some }\alpha>0,
\end{equation} 
 and if the new mutant can coexist with the unfit $aa$ individuals, then the 
$aa$ population can potentially recover.
This is the starting point of the present paper. 

\subsection{Goal of the paper}

The goal of this paper is to show that under reasonable hypotheses, the prolonged survival of
the $a$ allele after the invasion of the $A$ allele can indeed lead to a recovery of the $aa$-type. 
To do this, we assume that there will occur a new mutant allele (on the same gene), $B$, that on the one hand has 
a higher fitness than the $AA$-phenotype but that (for simplicity) has no competition with the
$aa$-type. The possible genotypes after this mutation are  $aa,aA,AA,aB,AB$, and $BB$, so that 
even for the deterministic system we have now to deal with a $6$-dimensional dynamical system 
whose analysis if far from simple.

Under the assumption of dominance of the fittest phenotype, and  mutation rate satisfying \eqref{time_scale},
we consider the model described in Section \ref{sec-stoch-model} starting at the time of the second mutation,  that is (with probability converging to 1 as $K\to\infty$) the $AA$ population being close to its equilibrium and the $aA$ population having decreased to a size of order $\e_K=O(K\mu_K)$, while the $aa$ population is of the order of the square of the $aA$ population. We assume that there just occurred a mutation to a fitter (and more dominant) allele $B$: we thus start with a quantity $\frac1K$ of genotype $AB$. We will start with a population where $AA$ is close to its equilibrium, the 
populations of $aA$ and $aa$ are already small (of order $\e_K$ and $\e_K^2$), and by mutation a single 
individual of genotype $AB$ appears.

 By using well known techniques \cite{Cha06,CM11, CMM13}, we know that the  $AB$ population behaves as a super-critical branching process and reaches a level $\e_0>0$ with positive probability in a time of order $\ln K$, without perturbing the 3-system $(aa,aA,AA)$.  

We see in numerical solutions to the deterministic system that a reduced fertility together with a reduced competition between $a$ and $B$ phenotypes constitutes a sufficient condition for the recovery of the $aa$ population.
For simplicity and in order to prove rigorous results, we suppose that there can be no  reproduction between individuals of phenotypes $a$ and $B$,
nor competition between them, and we reduce the number of remaining parameters as much as possible (see Section \ref{sec-model}). 
We study the deterministic system which corresponds to the large population limit of the stochastic counterpart, and we show that (for an initial quantity $\e$ of $aA$, $\e^2$ of $aa$ and $\e^3$ of $AB$) the system converges to a fixed point denoted by $p_{aB}$ consisting of the two coexisting populations $aa$ and $BB$. 
{If no further assumptions are made, we show  that the number of individuals bearing an $a$ allele decreases to level {$\e^{1+\D/(1-\D)}$} (where $\Delta$ is defined in \eqref{def-Delta}) before $aa$ grows and stabilises at order 1.}
\begin{remark}
Note that our study considers an initial quantity of $aA$ individuals $n_{aA}(0)=\e>0$, which should be thought as $\e =\e_K=O(K\mu_K)$ for the following discussion about the associated stochastic system. The quantity $\e_K$ satisfies $(\ln K)^{-1}\ll\e_K\ll O(K^{-\frac12+\alpha})$ if the mutation time satisfies \eqref{time_scale}.
\end{remark}

Let us discuss the (conjectural) implications of our results on the stochastic system:\\
{If $\Delta<\frac{\alpha}{1-2\alpha}$, our control on the $a$ allele is in principle sufficient  in order  for the \emph{stochastic} system to exhibit the recovery of $aa$ with positive probability in the large population limit.   
 Indeed, if the mutation time is of order  $K^{\frac12-\alpha}$, then the initial amount of $aA$ and $aa$ genotypes is close to the typical fluctuations of those populations, that is respectively of orders $O(K^{-\frac12+\alpha})$ and $O(K^{-1+2\alpha})$. 
Following the heuristics of \cite{BovNeu} (although the six-dimensional stochastic process is surely much more tedious to study), the deterministic system should constitute a good approximation of the process {if} the typical fluctuations of  populations containing an $a$ allele do not bring them to extinction. 
 If $\Delta<\frac{\alpha}{1-2\alpha}$ this ensures that the population containing an $a$ allele  is not falling below order $K^{-1/2}$ at any time. }\\

In order to go deeper and control the speed of recovery of the $aa$ population, we look for a parameter regime which ensures that the $aa$ population  always grows after the invasion of $B$.
Ensuring this lower bound on $aa$  is not trivial at all,
and the solution we found is to introduce an additional parameter $\eta$, which lowers the competition between the $aA$ and $BB$ populations, compared to the one between $AA$ and $BB$. Note that the competition does not depend only on the phenotype, and can be interpreted as a refinement of a phenotypic competition for resources: the strength (or
ability to get resources) of an individual not only depends  on its phenotype but also on
the dominance of its genotype. A biological interpretation for this kind of competition could be that it is coded in the alleles which food an individual with a given genotype prefers. 
We show that for $\eta$ larger than some positive value (of order $\Delta$),  the $aa$ population always grows after the invasion of $B$. The time of convergence to the coexistence fixed point is thus lowered, see Figure \ref{pic-loglog-eta-var}.
Moreover, we point out the existence of a bifurcation:  for $\eta$ larger than some threshold, the co-existence fixed point $p_{aB}$ becomes unstable and the system converges to another fixed point where all populations coexist.


Our contribution  is a rigorous description of a mechanism by which a recessive allele can re-emerge in a  population. This can be seen as a statement of genetic robustness exhibited by diploid populations performing sexual reproduction.

The structure of the paper is the following. In Section \ref{sec-model} we describe our assumptions on the parameters of the model, and compute the large population limit; in Section \ref{sec-qual} we present our results on the evolution of the deterministic system towards the co-existence fixed point $p_{aB}$, and we give a heuristic of the proof. Section \ref{sec-proof} is dedicated to the proof of these results.
The closing Section \ref{discussion} contain a heuristic considerations and numerical simulations of the model with relaxed assumption on the parameters.

\begin{figure}[h!]
	\includegraphics[width=.7\textwidth]{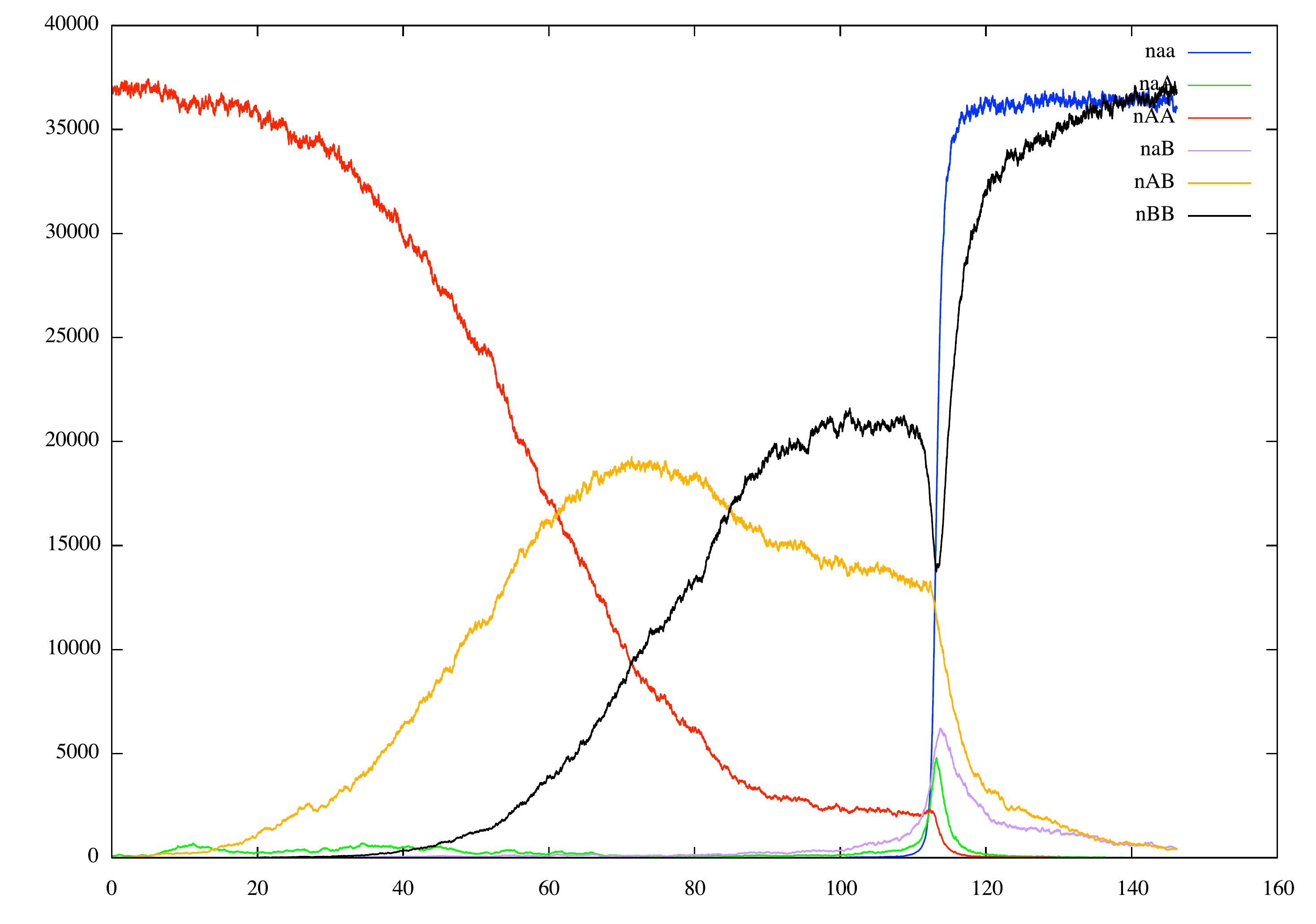}
	\caption{Simulation of the stochastic system for $f=6,D=0.7$, $\Delta=0.1$, $c=1, \eta=0.02$, $\e=0.014$ and $K=7000$.}
	\label{pic-stoch}
\end{figure}

\paragraph{Notation.} We write $x=\OO(y)$ whenever $x=O(y)$ and $y=O(x)$ as $\e\to0.$

\section{Model setup}\label{sec-model}

Let $\mathcal G=\{aa,aA,AA,aB,AB,BB\}$ be the genotype space.
Let $n_{i}(t)$  be the number of individuals with genotype $i\in\mathcal G$ in the population at time $t$ and 
set $n^K_{i}(t)\equiv  \frac{1}{K}n_{i}(t)$.
\begin{definition}
	The \emph{equilibrium size} of a monomorphic $uu$ population, $u\in\{a,A,B\}$, is the fixed point of  a
	1-dimensional Lotka-Volterra 
	equation and is given by 
	\begin{align}
	\label{equi}
	\bar n_u=\frac{f_{uu}-D_{uu}}{c_{uu,uu}}.
	\end{align} 
\end{definition}

\begin{definition}
	For  $u,v\in\{a,A,B\}$, we call
	\begin{align}
	S_{uv,uu}=f_{uv}-D_{uv}-c_{uv,uu}\bar n_{u},
	\end{align}
 the \emph{invasion fitness} of a mutant $uv$ in a resident $uu$ population.
\end{definition}

We take  the phenotypic viewpoint and assume that the $B$ allele is the most dominant one. That means the ascending order of dominance (in the Mendelian sense) is given by $a<A<B$, i.e.
\begin{enumerate}
	\item phenotype $a$ consists of the genotype $aa$,
	\item phenotype $A$ consists of the genotypes $aA,AA$,
	\item phenotype $B$ consists of the genotypes $aB,AB,BB$.\\
\end{enumerate}	

 For simplicity,  we assume that the fertilities are the same for all genotypes, and that natural death rates are the same 
within the three different phenotypes. Moreover, we assume that there can be no reproduction between $a$ and $B$ phenotypes.\\
To sumarize, we make the following Assumptions \textbf{(B)} on the rates: 

\begin{itemize}
	
	\item[\textbf{(B1)}] \emph{Fertilities.} For all  $ i\in\mathcal G$, and some $f>0$
	 \begin{align}f_{i}\equiv  f.\end{align}
	 
	\item[\textbf{(B2)}] \emph{Natural death rates.} The difference in fitness of the three phenotypes is realised by choosing 
	 a slightly higher natural death-rate of the $a$-phenotype and a slightly lower death-rate for the $B$-phenotype. For some $0<\D<D$,
	\begin{align}
	D_{aa}&= D+\D,\\
	D_{AA}\equiv  D_{aA}&=D,\\
	D_{aB}\equiv  D_{AB} \equiv  D_{BB}&=D-\D.
	\end{align}
	
	\item[\textbf{(B3)}] \emph{Competition rates.} We require that phenotypes $a$ and $B$ do not compete with each other. 
	Moreover, we introduce a parameter $\eta\geq0$ which lowers the competition between $BB$ and $aA$.
	For some $0\leq\eta<c$, 
	\begin{equation*}\left(c_{i,j}\right)_{\{i,j\}\in\mathcal G\times\mathcal G}=:
	\begin{tabular}{c|c|c|c}
	& $aa$ & $aA \quad AA$&$aB\quad AB\quad BB$\\
	\hline
	$aa$ & $c$ &$c\quad c$&$0 \quad 0 \quad 0$\\
	\hline
	$aA$ & $c$ &$c\quad c$&$c \quad c \quad c-\eta$\\
	$AA$ & $c$ &$c\quad c$&$c \quad c \quad c$\\
	\hline
	$aB$ & $0$ &$c\quad c$&$c \quad c \quad c$\\
	$AB$ & $0$ &$c\quad c$&$c\quad c \quad c$\\
	$BB$ & $0$ &$c-\eta \quad c$&$c \quad c\quad c$\\
	\hline
	\end{tabular}
	\end{equation*}
	\begin{remark} If $\eta>0$, the competition does not depend only on the phenotype, and can be interpreted as a refinement of a phenotypic competition for resources: the strength (or
		ability to get resources) of an individual not only depends  on its phenotype but also on its genotype. Genetically, it makes sense to assume that (positive) competition rates are decreasing in the "genetic distance" between two individuals. This is the case for the above competition matrix if we assume that the mutations $a\to A\to B$ can only occur successively. Indeed, the Hamming distance between genotypes is then the following graph distance, where $d(aA,BB)>d(AA,BB)$.
		\begin{center}
			\includegraphics[width=6cm]{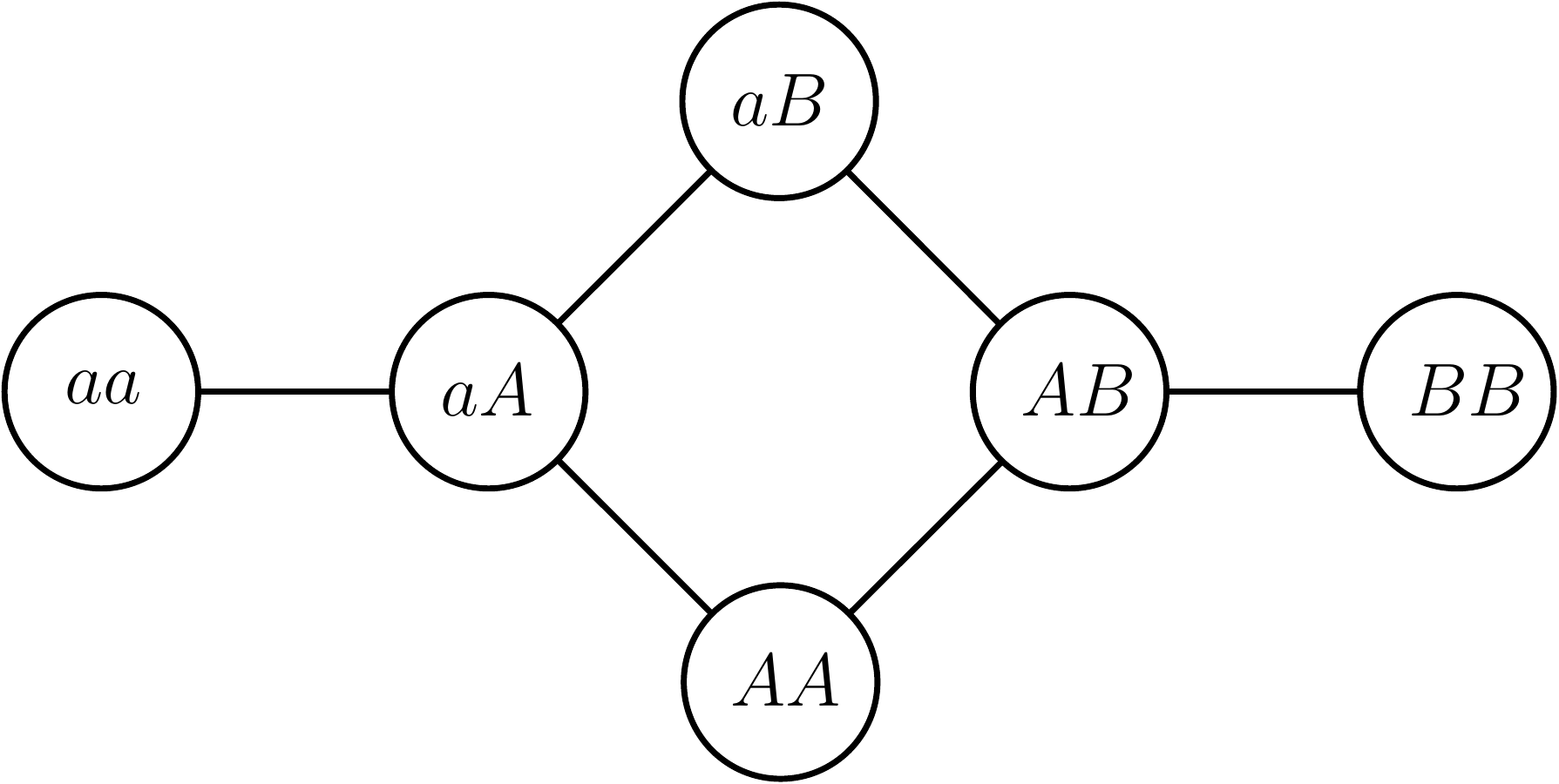}
		\end{center}
	\end{remark}
	
		\item[\textbf{(B4)}] \emph{Reproductive compatibility.} We require that phenotypes $a$ and $B$ do not reproduce with each other. 
	\begin{equation*}\left(R_{i}(j)\right)_{\{i,j\}\in\mathcal G\times\mathcal G}\equiv  
	\begin{tabular}{c|c|c|c}
	& $aa$ & $aA \quad AA$&$aB\quad AB\quad BB$\\
	\hline
	$aa$ & $1$ &$1\quad 1$&$0 \quad 0 \quad 0$\\
	\hline
	$aA$ & $1$ &$1\quad 1$&$1 \quad 1 \quad 1$\\
	$AA$ & $1$ &$1\quad 1$&$1 \quad 1 \quad 1$\\
	\hline
	$aB$ & $0$ &$1\quad 1$&$1 \quad 1 \quad 1$\\
	$AB$ & $0$ &$1\quad 1$&$1\quad 1 \quad 1$\\
	$BB$ & $0$ &$1 \quad 1$&$1 \quad 1\quad 1$\\
	\hline
	\end{tabular}
	\end{equation*}
	
\end{itemize}

Observe that, under Assumptions  \textbf{(B)},
\begin{align}
\label{infit}
S_{AB,AA}&=f-(D-\D)-c\bar n_{AA}=f-D+\Delta-c\frac{f-D}{c}=\D,\\
S_{aa,BB}&=f-D-\D,\\
c\bar n_B&=f-D+\Delta
\end{align}
Therefore,  the mutant $AB$ has a  positive invasion fitness in the population $AA$, as well as $aa$ in the $BB$ population (due to the absence of competition between them).

\subsection{Birth rates}
Since we assume that there is no recombination between phenotypes $a$ and $B$. Thus, 
\begin{enumerate}
	\item the pool of possible partners for the phenotype $a$ consists of phenotypes $a$ and $A$; 
	the total population of this pool is denoted by 
	\begin{equation}
	\Sigma_3:=n_{aa}+n_{aA}+n_{AA},
	\end{equation}
	\item the pool of possible partners for the phenotype $A$ consists of the three phenotypes $a$, $A$, and $B$;
	the total population of this pool is denoted by 
	\begin{equation}
	\Sigma_6:=n_{aa}+n_{aA}+n_{AA}+n_{aB}+n_{AB}+n_{BB},
	\end{equation}
	\item the pool of possible partners for the phenotype $B$ consists of phenotypes $A$ and $B$;
	the total population of this pool is denoted by 
	\begin{equation}
	\Sigma_5:=n_{aA}+n_{AA}+n_{aB}+n_{AB}+n_{BB}.
	\end{equation}
\end{enumerate} 
Computing the reproduction rates with the Mendelian rules as described in \cite{BovNeu} 
leads to the following (time-dependent) birth-rates $b_i=b_i(n(t))$: 

\begin{align}\label{birthratesaa}
b_{aa}=&f\frac{n_{aa} \left(n_{aa}+\frac12n_{aA}\right)}{\S_3}+f\frac{\frac12n_{aB} \left(\frac12n_{aA}+\frac12n_{aB}\right)}{\S_5}
+f\frac{\frac12n_{aA} \left(n_{aa}+\frac12n_{aA}+\frac12n_{aB}\right)}{\S_6},\\[0.5cm]\label{birthratesaA}
b_{aA}=&f\frac{n_{aa} \left(\sfrac12n_{aA}+n_{AA}\right)}{\S_3}+f\frac{\frac12n_{aA} \left(\frac12n_{aB}+\frac12n_{AB}\right)+\frac12n_{aB} \left(n_{AA}+n_{AB}\right)}{\S_5}\nonumber\\[0.2cm]
&+f\frac{\left(\frac12n_{aA}+n_{AA}\right) \left(n_{aa}+n_{aA}+\frac12n_{aB}\right)+\frac14n_{aA} n_{AB}}{\S_6},\\[0.5cm]\label{birthratesAA}
b_{AA}=&f\frac{\frac12n_{AB} \left(\frac12n_{aA}+n_{AA}+\frac12n_{AB}\right)}{\S_5}+f\frac{\left(\frac12n_{aA}+n_{AA}\right) \left(\frac12n_{aA}+n_{AA}+\frac12n_{AB}\right)}{\S_6},\\[0.5cm]\label{birthratesaB}
b_{aB}=&f\frac{\left(\frac12n_{aA}+n_{aB}\right) \left(\frac12n_{aB}+\frac12n_{AB}+n_{BB}\right)}{\S_5}+f\frac{\frac12n_{aA} \left(\frac12n_{aB}+\frac12n_{AB}+n_{BB}\right)}{\S_6},\\[0.5cm]\label{birthratesAB}
b_{AB}=&f\frac{\left(\frac12n_{aA}+n_{AA}+n_{AB}\right) \left(\frac12n_{aB}+\frac12n_{AB}+n_{BB}\right)}{\S_5}+f\frac{\left(\frac12n_{aA}+n_{AA}\right) \left(\frac12n_{aB}+\frac12n_{AB}+ n_{BB}\right)}{\S_6},\\[0.5cm]\label{birthratesBB}
b_{BB}=&f\frac{\frac14\left(n_{aB}+n_{AB}+2 n_{BB}\right)^2}{\S_5}.
\end{align}

\subsection{Death rates}
The death rates are the sum of the natural and competition death rates:
\begin{align}\label{death-ratesaa}
d_{aa}&=n_{aa}(D+\D+c\S_3),\\[0cm]\label{death-ratesaA}
d_{aA}&=n_{aA}(D+c(n_{aa}+n_{aA}+n_{AA}+n_{aB}+n_{AB})+(c-\eta)n_{BB}),\\[0cm]\label{death-ratesAA}
d_{AA}&=n_{AA}(D+c\S_6),\\[0cm]\label{death-ratesaB}
d_{aB}&=n_{aB}(D-\D+c\S_5),\\[0cm]\label{death-ratesAB}
d_{AB}&=n_{AB}(D-\D+c\S_5),\\[0cm]\label{death-ratesBB}
d_{BB}&=n_{BB}(D-\D+(c-\eta)n_{aA}+c(n_{AA}+n_{aB}+n_{AB}+n_{BB})).
\end{align} 

\subsection{Large population limit}

By \cite{FouMel2004} or \cite{CMM13}, for large populations, the behaviour of the stochastic
process is close to the solution of a deterministic equation.

\begin{proposition}[Theorem 2.1 in \cite{EthKur1986}]\emph{}\\ 
	Let $T>0$ and $C\subset\R_+^6$ be a compact set.
	Assume that the initial condition $n^K(0)=\frac{1}{K}(n_{aa}(0),n_{aA}(0),n_{AA}(0),n_{aB}(0),n_{AB}(0),n_{BB}(0))$ converges almost surely to a deterministic vector $x^0=(x^0_1,x^0_2,x^0_3,x^0_4,x^0_5,x^0_6)\in C$, as $K\to\infty$.\\
	Let $\tilde n(t,x^0)$ denote the solution to
	\begin{align}\label{dyn-syst}
	\dot n(t)&=b(n(t))-d(n(t))\equiv F(n(t)),\\
	\hbox{\rm i.e.} \quad \dot n_i(t)&=b_i(n(t))-\left(D_i+\sum_{j\in\GG}c_{i,j}n_j(t)\right)n_i(t),\quad \text{for all } i\in\GG,
	\end{align}
	with initial condition $x^0$,
	where $(b_i)_{i\in\GG}$ and $(d_i)_{i\in\GG}$ are given in \eqref{birthratesaa}-\eqref{birthratesBB} and \eqref{death-ratesaa}-\eqref{death-ratesBB}.
	Then, for all $T>0$,
	\begin{align}
	\label{inva}
	\lim_{K\rightarrow\infty}\sup_{t\in[0,T]}|n^K_{i}(t)-\tilde n_{i}(t,x^0))|=0,\quad \text{a.s.},
	\end{align}
	for all $i\in\GG$.
\end{proposition}

\subsection{Initial condition}

Fix $\e>0$  sufficiently small. 
 For the results below, 
we will consider the dynamical system \eqref{dyn-syst} starting with the initial condition: 
\begin{align}
\label{init-condAA}
\bar n_A\geq n_{AA}(0)&\geq \bar n_A-\OO(\e),\\\label{init-condaA}
n_{aA}(0)&=\e,\\\label{init-condaa}
\e ^2\geq n_{aa}(0)&=\OO(\e^2),\\\label{init-condAB}
n_{AB}(0)&=\e^3,\\ \label{init-condBB}
n_{BB}(0)&=0,\\\label{init-condaB}
n_{aB}(0)&=0,
\end{align}
which corresponds to the long time behaviour of the dynamical system considered in \cite{BovNeu} plus a quantity $\e^3$ of the new mutant $AB$.

\begin{remark}
 In all the figures below, the choice of parameters is the following:
\begin{center}
	\begin{tabular}{ccccc}
		$f=6$, & $D=0.7$, & $\Delta=0.1$, & $c=1$, & $\e=0.01$,
	\end{tabular}
\end{center} 
 and the parameter $\eta$ is specified on each picture.
\end{remark}

\section{Results}\label{sec-qual}
We are working with a 6-dimensional dynamical system, and computing all the fixed points analytically  
is impossible for a general choice of the parameters. We can, however, compute those which are 
relevant for our study. We will call $p_A$ (resp. $p_B$) the fixed points corresponding to the 
monomorphic $AA$ (resp. $BB$) population at equilibrium, and $p_{aB}$ the fixed point 
corresponding to the coexisting $aa$ and $BB$ populations. Setting the relevant populations to 0 and 
solving $\dot n(t)=0$, we get:
\begin{align}
	p_A&=(0,0,\bar n_A,0,0,0),\\
	p_B&=(0,0,0,0,0,\bar n_B),\\
	p_{aB}&=(\bar n_a,0,0,0,0,\bar n_B),
\end{align}
where 
$\bar n_a=\frac{f-D-\Delta}c, \bar n_A=\frac{f-D}c,$ and $ \bar n_B=\frac{f-D+\Delta}{c}.$
Note that the $BB$ equilibrium population is the same in $p_B$ and $p_{aB}$. This is due to the non-interaction between phenotypes $a$ and $B$.

\begin{figure}[t]
	\begin{center}\raisebox{1.5cm}{\includegraphics[width=.495\textwidth]{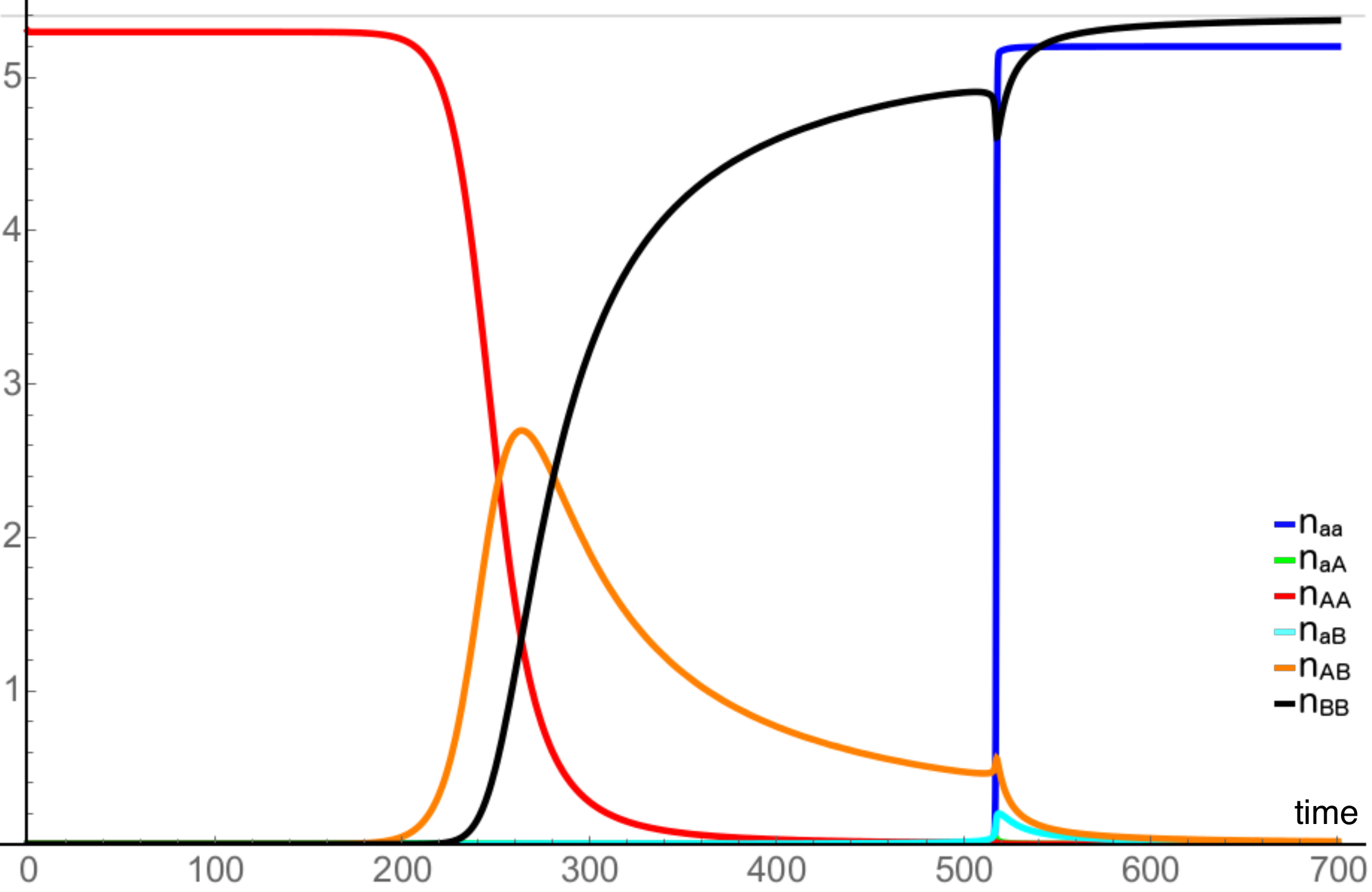}}
		\includegraphics[width=0.45\textwidth]{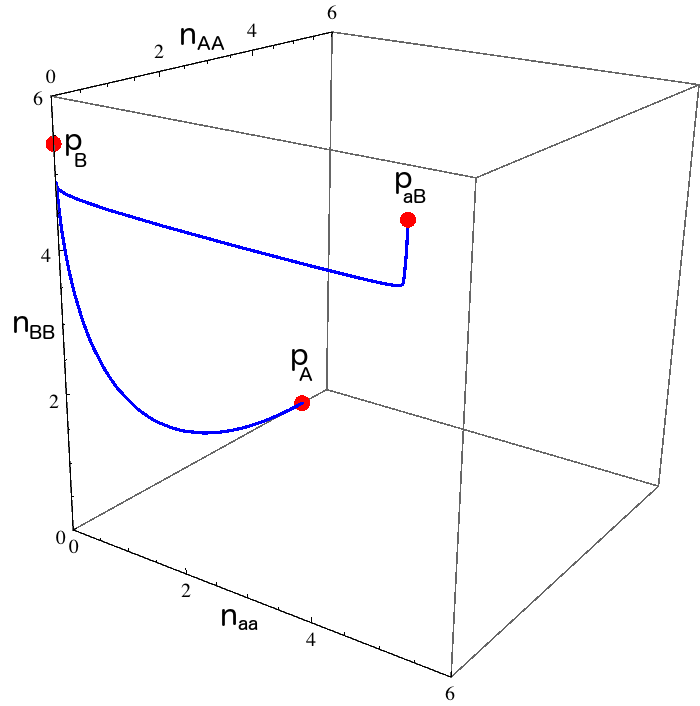}
		\caption{General qualitative behaviour of $\{n_i(t),i\in\mathcal G\}$ and projection of the dynamical system on the coordinates $aa, AA$ and $BB$. The re-invasion of the $aa$ population happens sooner and sooner as $\eta$ grows ($\eta=0.02$ for both pictures). }
		\label{pic-general}
	\end{center}
\end{figure}

Our general result is that starting with initial conditions \eqref{init-condAA}-\eqref{init-condaB}, that is close to $p_A$ (with small coordinates in directions $aa$, $aA$ and $AB$), and under minimal assumptions on the parameters, the system gets very close to  $p_B$ before finally converging to $p_{aB}$, see Figure \ref{pic-general}.

\begin{theorem}\label{main-thm}
	Consider the dynamical system \eqref{dyn-syst} started with initial conditions \eqref{init-condAA}-\eqref{init-condaB}. Suppose the following Assumptions \textbf{(C)} on the parameters hold:
	\begin{enumerate}
		\item[\textbf{(C1)}] 	 $\Delta$ sufficiently small, 
		\item[\textbf{(C2)}] 	 $f$ sufficiently large, 
		\item[\textbf{(C3)}]   $0\leq\eta<c/2$.
	\end{enumerate}
	Then the system converges to the fixed point $p_{aB}$.
	More precisely, for any fixed $\delta>0$, as $\e\to0$, 	it reaches a $\delta$-neighbourhood of $p_{aB}$ in a time of order $\OO(\e ^{-1/(1+\eta\bar n_B-\D)})$. \\
	Moreover:  
	\begin{enumerate}
		\item \label{eta0decay} for $\eta=0$,  the amount of  allele $a$ in the population decays to $\OO(\e^{1+\D/(1+\D)})$ before reaching $\OO(1)$,
	
		\item \label{etalarge} for $\eta>\frac{4\D}{\bar n_B}$, the amount of $a$ allele in the population is  bounded below by $\OO(\e)$ for all $t>0$.
	
	\end{enumerate}
\end{theorem}

\begin{remark}
		For $\eta$ large, we prove that the fixed point $p_{aB}$ is unstable.
		We observe numerically that the system is attracted to a fixed point where all the 6 populations coexist, but we do not prove this.
\end{remark}

Let us now briefly discuss the linear stability of the relevant fixed points and give a heuristics of the proof of Theorem \ref{main-thm}.

\subsection{Linear stability analysis}\label{sec-stability}

The Jacobian matrix  $J_F:=(\partial F_i/\partial n_j)_{ij}$ of the map $F$ defined in \eqref{dyn-syst} can be explicitly computed at $p_A$ and $p_{aB}$ and the situation is as follows:
\begin{itemize}
	\item The eigenvalues of $J_F(p_A)$ are $0, \Delta>0$ and $-(f-D), -(f+\Delta), -(f-\Delta)$ (double) which are all strictly negative under Assumptions (C). The fixed point $p_A$ is thus unstable.
	\item The eigenvalues of $J_F(p_{aB})$ are $0$ (double), and $ -(2f-D), -(f-D+\Delta), 
	-(f-D-\Delta), -((f-D)(5f-4D)+f\Delta)/(4(f-D)+\eta\bar n_B)$ which are strictly negative under Assumptions (C). The linear analysis thus does not imply the stability of $p_{aB}$ but the Phase 4 of the proof does (see Section \ref{subsec-phase5}) . 
	\end{itemize}
	
	It turns out that $J_F(p_B)$ is singular but as the invasion fitness of $aa$ is positive, i.e. $S_{aa,BB}>0$ (see \eqref{infit}), this implies that a small perturbation in the first coordinate will be amplified, and thus implies the instability of the fixed point $p_B$.

\subsection{Heuristics of the proof}\label{subsec-heuristics}

Recall we start the dynamical system \eqref{dyn-syst} with initial conditions \eqref{init-condAA}-\eqref{init-condaB}. A numerical solution of the system is provided on Figure \ref{detsys}. 
\begin{remark}
	Assumption C1 of Theorem \ref{main-thm} is needed throughout the proof in order to be able to use the results of \cite{BovNeu} which rely on the Center Manifold Theorem (a line of fixed points becomes an invariant line under small enough perturbation).
\end{remark}

	\begin{enumerate}
				\item[\textbf{Phase 1.}] Time period:  until $n_{AB}=\e_0$.\\
				  The mutant population, consisting of all individuals of phenotype $B$, first grows up to $\e_0$ exponentially fast {with rate $\Delta$} without perturbing	the behaviour of the 3-system $(aa,aA,AA)$. The rate of growth corresponds to the invasion fitness of $AB$ in the resident population $AA$, see \eqref{infit}. Following \cite{BovNeu}, $AA$ stays close to $\bar n_A$, while $aA$ and $aa$ continue to decay like $1/t$ and $1/t^2$ respectively.
				The duration $T_1$ of this phase is such that $\OO(\e^3)e^{t\D}=\OO(1)\Leftrightarrow T_1=\OO(|\log\e|)$.\\
				
				\item[\textbf{Phase 2.}] Time period:  until $n_{aA}=\OO(n_{AA})$.\\
				The evolution  is a perturbation of an effective 3-system $(AA,AB,BB)$ which behaves exactly the same as in \cite{BovNeu}, since the parameters satisfy the same hypotheses (slightly lower death rate for phenotype $B$  than for phenotype $A$, and constant competition parameters). 
				A comparison result (following Theorem \ref{3system} below) shows that this 3-system is almost unperturbed until $n_{aA}=\OO(n_{AA})$. 	If that happens in a time $T_2$ diverging with $\e$ (which we ensure throughout the calculation), we thus know that $BB$ approaches $\bar n_B$, while 	$n_{AB}\propto1/t$ and
				$n_{AA}\propto1/t^2$.\\
				The important fact in this phase is that the amount of  allele $a$ in the population decays 
				for $\eta$ small while it increases for large enough $\eta$. Indeed, let us derive some bounds on $\Sigma_{aA,aB}=n_{aA}+n_{aB}$.
				The population $\Sigma_{aA,aB}$ reproduces by taking the dominant allele in a population of order $\OO(1)$ and the allele $a$ in itself. Thus its birth rate satisfies $b_{\Sigma_{aA,aB}}
				\approx f\Sigma_{aA,aB}$. We can compute its death rate exactly and use that $n_{BB}\approx\Sigma_5\approx\bar n_B$: 
				\begin{align}\nonumber
				d_{\Sigma_{aA,aB}}&=\Sigma_{aA,aB}(D-\D+c\S_5)-\eta n_{aA}n_{BB}+\D n_{aA}\\
				&\approx f\Sigma_{aA,aB}-n_{aA}(\eta\bar n_{B}-\D),\\ \nonumber
				\dot{\Sigma}_{aA,aB}&\approx n_{aA}(\eta\bar n_{B}-\D)\\
				&=\OO({\Sigma}_{aA,aB}\cdot n_{AB})(\eta\bar n_{B}-\D). \label{sigmadot}
				\end{align}
				The last equality comes from the fact that $aA$ newborns have mainly their $a$ allele coming from ${\Sigma}_{aA,aB}$ and their $A$ allele coming from $AB$. 
			   Using the $1/t$ decay of $AB$ we get:
			   \begin{align}
			   	\dot{\Sigma}_{aA,aB}\approx \frac{\OO({\Sigma}_{aA,aB})}{\OO(1)+\OO(1)t}(\eta\bar n_{B}-\D)
				\end{align} 
			    As ${\Sigma}_{aA,aB}(T_1)=\OO(\e)$ we deduce that
			   ${\Sigma}_{aA,aB}(t)={\OO(\e)}(\OO(1)+\OO(1)t)^{\OO(\eta\bar n_{B}-\D)}$,
			   and thus
			   $n_{aA}=\OO(n_{AB}\cdot{\Sigma}_{aA,aB} )={{\OO(\e)}(\OO(1)+\OO(1)t)^{\OO(\eta\bar n_{B}-\D)}}/{(\OO(1)+\OO(1)t)}$.
			   By solving $n_{aA}=\OO(n_{AA})=\OO(n_{AB}^2)$ we get the order of magnitude of $T_2=\OO\left(\e^{-1/(1+\eta\bar n_{B}-\D)}\right)$.
			   Note that for $\eta=0$,  ${\Sigma}_{aA,aB}(T_2)=\OO(\e ^{1+{\D}/{(1-\D)}})$. 
			   Moreover, \eqref{sigmadot} implies that for $\eta>\Delta/\bar n_B$, we have $\dot{\Sigma}_{aA,aB}>0$, which proves points \ref{eta0decay} and \ref{etalarge} of Theorem \ref{main-thm}.\\

				\item[\textbf{Phase 3.}] Time period:  until ${aa}$ reaches equilibrium.\\
				The fact that $n_{aA}=\OO(n_{AA})$ has a crucial effect on the birth rate of $aa$ 
				(see \eqref{birthratesaa}) since the term
				 $(n_{aa}+\frac12n_{aA})/({n_{aa}+n_{aA}+n_{AA}})$ becomes of order 
				 $\OO(1)$. As long as $AA$ stays smaller than $\OO(\e)$, we get a lower bound
				  on $n_{aa}$ which  grows exponentially fast  since $f$ is chosen large enough 
				  (Assumption C2):
				\begin{align}
				b_{aa}&\geq fn_{aa}\OO(1),\\
				d_{aa}&\leq n_{aa}(D+\D+\OO(\e)),\\
				\dot{n}_{aa}&\geq n_{aa}(f\OO(1)-D-\D-\OO(\e)).
				\end{align} 
				As $aa$ grows, it makes ${\Sigma}_{aA,aB}$ grow, and thus $AA$ and $AB$ as 
				well. We have to show  that this could not prevent $aa$ from reaching 
				equilibrium. We do not give a detailed argument here, but essentially, the 
				presence of the macroscopic $BB$ population prevents all the non-$aa$ 
				populations to grow too much. 
				Note that if $\eta$ is too large, then $aA$ could get a positive fitness and grow to
				 a macroscopic level. That is why we have to impose Assumption C3, which will 
				 become clearer heuristically in the next phase.
				We recall that $aa$ does not compete with $BB$ and thus it grows exponentially 
				fast  with rate ${f-(D+\Delta)}$ until an $\e_0$-neighbourhood of the  fixed point 
				where $aa$ and $BB$ coexist. The rate of growth corresponds to the invasion 
				fitness of $aa$ in the resident population $BB$, see \eqref{infit}. Note that, due to 
				Assumption C2, this rate is much larger than the invasion rate of $BB$ into $AA$. 
				That is why the fourth phase looks very steep on Figure \ref{detsys}, see the 
				stretched version on Figure \ref{zoom-in}. This phase lasts a time 
				$T_3=\OO(|\log\e |)$.\\
				
				\item[\textbf{Phase 4.}] 
				The Jacobian matrix of the field \eqref{dyn-syst} at the fixed point $p_{aB}$ has two zero, and 4 negative eigenvalues.
				 $p_{aB}$ is thus a non-hyperbolic equilibrium point of the system and linearisation fails to determine its stability properties. Instead, we use the result of  center manifold theory (\cite{H77, P01}) that asserts that the qualitative behaviour of the dynamical system in a neighbourhood of the non-hyperbolic critical point $p_{aB}$ is determined by its behaviour on the center manifold near $p_{aB}$. Using the Center Manifold Theorem, we show that asymptotically as $f\to\infty$, the field is attractive for $\eta<c\cdot r_{max} $ where $r_{max}\simeq0.593644$ is the maximum of the rational function \eqref{r}.
				 Thus $p_{aB}$ is a stable fixed point which is approached with speed $\frac 1{t}$ as long as $\eta<c\cdot r_{max}$.
				For higher values of $\eta$, numerical solutions show that the system converges to a fixed point where the 6 populations co-exist, but we do not prove this.
	\end{enumerate}

\begin{figure}[t]
	\includegraphics[width=.9\textwidth]{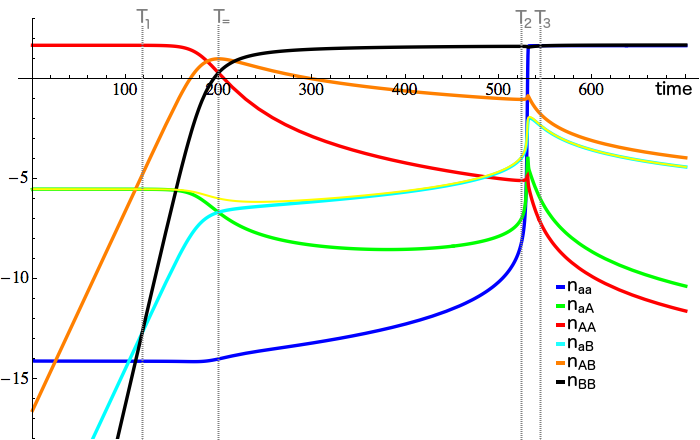}
	\caption{Numerical solution of the deterministic system for $\eta=0.02$, logplot. See Figure \ref{zoom-in} for a zoom around the time interval $[T_2,T_3]$.}
	\label{detsys}
\end{figure}

\section{Proof}\label{sec-proof}

%
%

\begin{definition}\label{def-stop-times}Let $x,y,z\in\{aa,aA,AA,aB,AB,BB\}$ and $h\in\R$. We define
	\begin{align}
	T^{x=y}&=\inf\{t>0:n_x(t)=n_y(t)\},\\
	T^{x=\delta y}&=\inf\{t>0:n_x(t)=\delta n_y(t)\},\\
	T_{h}^{x}&=\inf\{t>0: n_{x}(t)>h\},\\
	T_{h}^{x+y}&=\inf\{t>0: n_{x}(t)+n_{y}(t)>h\},\\
    T_{h}^{x+y+z}&=\inf\{t>0: n_{x}(t)+n_{y}(t)+n_{z}(t)>h\}.
	\end{align}
	Moreover, let
	\begin{equation}
	\label{eps-delta}
		\D>\delta>\e_0>\e>0.
	\end{equation}	
\end{definition}

The value $\e_0$ is the small order 1 level in the Phase 1, see the proof heuristics (Section \ref{subsec-heuristics}). We consider $\Delta$ fixed and sufficiently small, and will  first send $\e\to0$  and then $\e_0\to0$.
A summary of the proof structure can be found in the Appendix \ref{appendix} together with the implications between intermediary results.

\subsection{Preliminaries}
We first state an elementary fact that  will be useful throughout the following analysis.
\begin{lemma}\label{lem-level}
	Let $T>0$ and $f:[0,T]\to\R$ be a differentiable function such that $f(0)\leq0$.
	If,  for all $t\in[0,T]$,  $\left\{f(t)=0\Rightarrow \dot f(t)< 0\right\}$  then $f(t)<0$, for all $t\in(0,T]$.
\end{lemma}
In the sequel, a part of our majorations and comparison results will rely on this lemma, which ensures that the trajectories of
two dynamical systems do not cross as long as the derivative of the upper trajectory is larger
than the derivative of the lower one, where these trajectories meet. We will often encounter the  situation where $c,T>0$ and $n,g:[0,T]\to\R$ are two differentiable functions such that $c n(0)\leq g(0)$.

Then, Lemma \ref{lem-level} implies that if,  for all $t\in[0,T]$,  $\left\{c n(t)=g(t)\Rightarrow c\dot n(t)<\dot g(t)\right\}$  then $c n(t)\leq g(t)$, for all $t\in[0,T]$.

\begin{proposition}\label{aB<AB}  We have the following ordering relations:
	\begin{enumerate}
		\item  If $n_{AA}(0)\leq \bar n_A$ then $n_{AA}(t)\leq \bar n_{A}$, for all $t>0$.
		\item  If $n_{BB}(0)< \bar n_B$ then $n_{BB}(t)\leq \bar n_{B}$, for all $t>0$.
		\item  If $n_{aB}(0)< n_{AB}(0)$ then $n_{aB}(t)\leq n_{AB}(t)$, for all $t>0$.
	\end{enumerate}
\end{proposition}
\begin{proof}
	\begin{enumerate}
		\item[(1)] We use Lemma \ref{lem-level} and show that $\dot n_{AA}<0$ at $n_{AA}=\bar n_A$. For this we construct a majorising process on $n_{AA}$ by comparing the birth and the death rates. Since $\S_6\geq\S_5$ we get:		\begin{align}
		b_{AA}&\leq\sfrac f{\S_5}n_{AA}(n_{aA}+n_{AA}+n_{AB})+\sfrac f{4\S_5}\left(n_{aA}^2+n_{AB}^2\right)\nonumber\\
		&\leq fn_{AA}+\sfrac f{4\S_5}\left(n_{aA}^2+n_{AB}^2\right),\\
		d_{AA}&= n_{AA}(D+c\S_6),\\
		\dot n_{AA}&\leq n_{AA}(f-D-c\S_6)+\sfrac f{4\S_5}\left(n_{aA}^2+n_{AB}^2\right).
		\end{align}
		At $n_{AA}=\bar n_A$ we have
		\begin{align}
		\dot n_{AA}&\leq n_{aA}\left(\sfrac f{4\S_5}n_{aA}-c\bar n_A\right)+ n_{AB}\left(\sfrac f{4\S_5}n_{AB}-c\bar n_A\right)\nonumber\\
		&\leq n_{aA}\left(\sfrac f{4\S_5}n_{aA}-f+D\right)+ n_{AB}\left(\sfrac f{4\S_5}n_{AB}-f+D\right)<0,
		\end{align}
		since $\sfrac f{4\S_5}n_{aA}$ and $\sfrac f{4\S_5}n_{AB}$ are smaller than $\sfrac f4$.
		\item[(2)] We proceed similarly using Lemma \ref{lem-level}. We show that $\dot n_{BB}<0$ at $n_{BB}=\bar n_B$ by constructing a majorising process on $n_{BB}$:	
		\begin{align}
		b_{BB}&=\sfrac f{\S_5}n_{BB}(n_{aB}+n_{AB}+n_{BB})+\sfrac f{4\S_5}\left(n_{aB}^2+n_{AB}^2\right)\nonumber\\
		&\leq fn_{BB}+\sfrac f{4\S_5}\left(n_{aB}^2+n_{AB}^2\right),\\
		d_{BB}&= n_{BB}(D-\D+c\S_5-\eta n_{aA}),\\
		\dot n_{BB}&\leq n_{BB}(f-D+\D-c\S_5+\eta n_{aA})+\sfrac f{4\S_5}\left(n_{aB}^2+n_{AB}^2\right).
		\end{align}
		At $n_{BB}=\bar n_B$ we have
		\begin{align}
		\hspace{2cm}\dot n_{BB}&\leq n_{aB}\left(\sfrac f{4\S_5}n_{aB}-c\bar n_B\right)+ n_{AB}\left(\sfrac f{4\S_5}n_{AB}-c\bar n_B\right)-\bar n_B(c-\eta)n_{aA}\nonumber\\
		&\leq n_{aB}\left(\sfrac f{4\S_5}n_{aB}-f+D-\D\right)+ n_{AB}\left(\sfrac f{4\S_5}n_{AB}-f+D-\D\right)-\bar n_B(c-\eta)n_{aA}\nonumber\\&<0,
		\end{align}
		since $\sfrac f{4\S_5}n_{aB}$ and $\sfrac f{4\S_5}n_{AB}$ are smaller than $\sfrac f4$ and $\eta\leq c$.
		\item[(3)] Intuitively, this inequality comes from the fact that  phenotype $a$ individuals 
		cannot reproduce with phenotype $B$. Indeed, 
		if we consider the couples that could give rise to an $AB$ (resp. $aB$) individual, they are of 
		the form $(Ag_1,Bg_2)$ (resp. $(ag_1,Bg_2)$),
		with $g_1,g_2\in\{a,A,B\}$ and the combination $(AA,Bg_2)$ is possible, whereas 
		$(aa,Bg_2)$ is impossible. Here is the rigorous derivation of the result: 
		We compare the birth- and the death-rates of $n_{AB}$ and $n_{aB}$, introduced in 
		\eqref{birthratesAB},\eqref{death-ratesAB}, and \eqref{birthratesaB},\eqref{death-ratesaB} 
		respectively:
		\begin{align}
		\frac{d_{aB}}{n_{aB}}&=D-\D+c\S_5=\frac{d_{AB}}{n_{AB}},\\[0.3cm]
		b_{aB}&=fn_{aB}\frac{\sfrac12n_{aB}+\sfrac12n_{AB}+n_{BB}}{\S_5}+I_{aB},\\
		b_{AB}&=fn_{AB}\frac{\sfrac12n_{aB}+\sfrac12n_{AB}+n_{BB}}{\S_5}+I_{AB}.
		\end{align}
		We see that the death-rates of the two populations are the same, whereas the birth-rates differ only in a factor which comes from the reproduction of the other populations. If we take a closer look to these factors $I_{aB}, I_{AB}$ under the assumption that $n_{aB}=n_{AB}$ we see that
		\begin{align}
		I_{AB}&=f\left(\sfrac12n_{aA}+n_{AA}\right)\left(\frac{\sfrac12n_{aB}+\sfrac12n_{AB}+n_{BB}}{\S_5}+\frac{\sfrac12n_{aB}+\sfrac12n_{AB}+n_{BB}}{\S_6}\right) \nonumber\\
		&=I_{aB}+fn_{AA}\left(\frac{\sfrac12n_{aB}+\sfrac12n_{AB}+n_{BB}}{\S_5}+\frac{\sfrac12n_{aB}+\sfrac12n_{AB}+n_{BB}}{\S_6}\right).
		\end{align}
		Thus $I_{AB}>I_{aB}$. Hence, $\dot n_{AB}>\dot n_{aB}$ and $n_{AB}(t)$ stays above $n_{aB}(t)$ for all $t>0$.
	\end{enumerate}
\end{proof}

\begin{figure}[t]
	\begin{center}
		\includegraphics[width=0.8\textwidth]{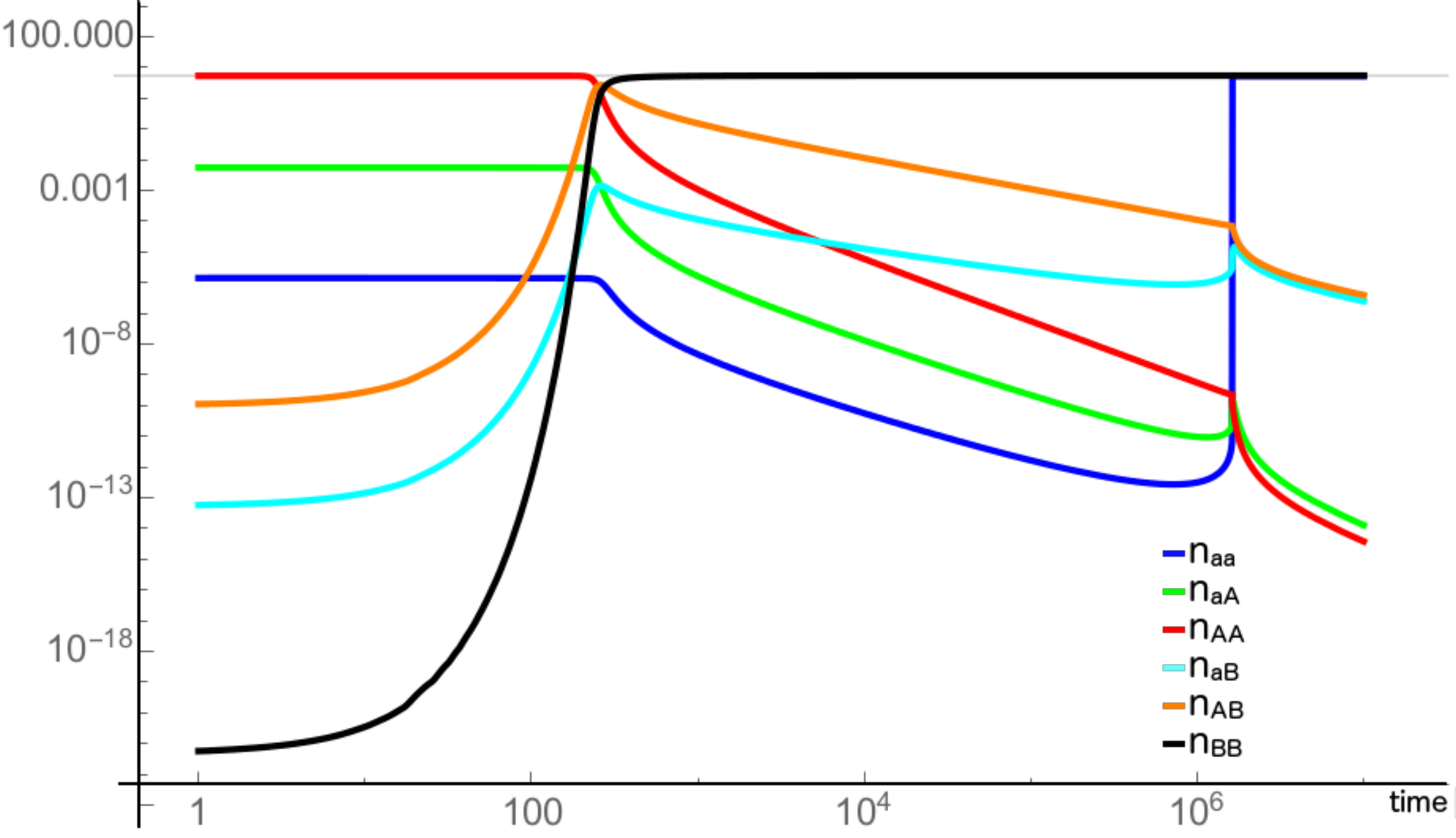}\\
		\includegraphics[width=0.8\textwidth]{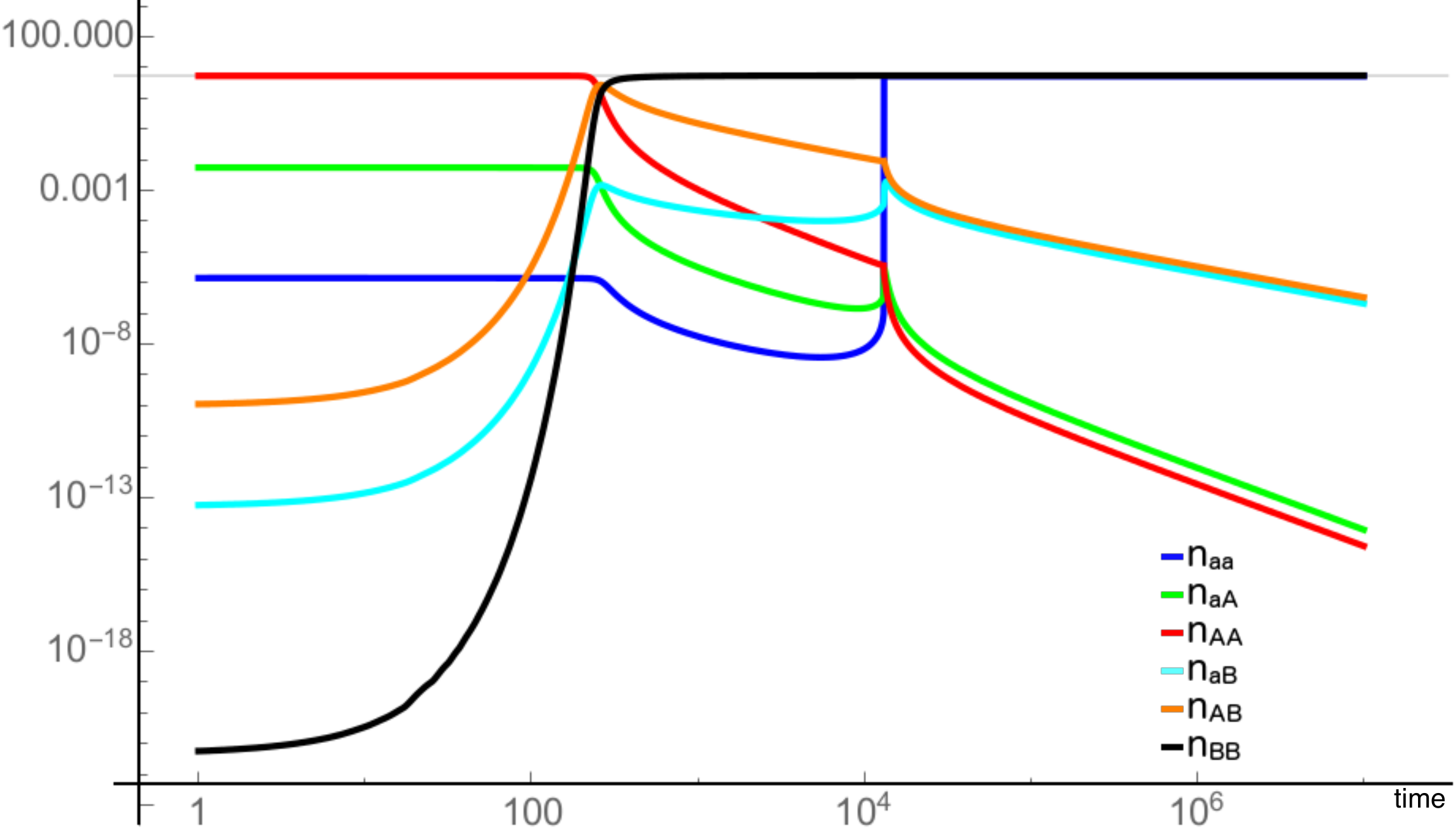}	\\			
		\includegraphics[width=0.8\textwidth]{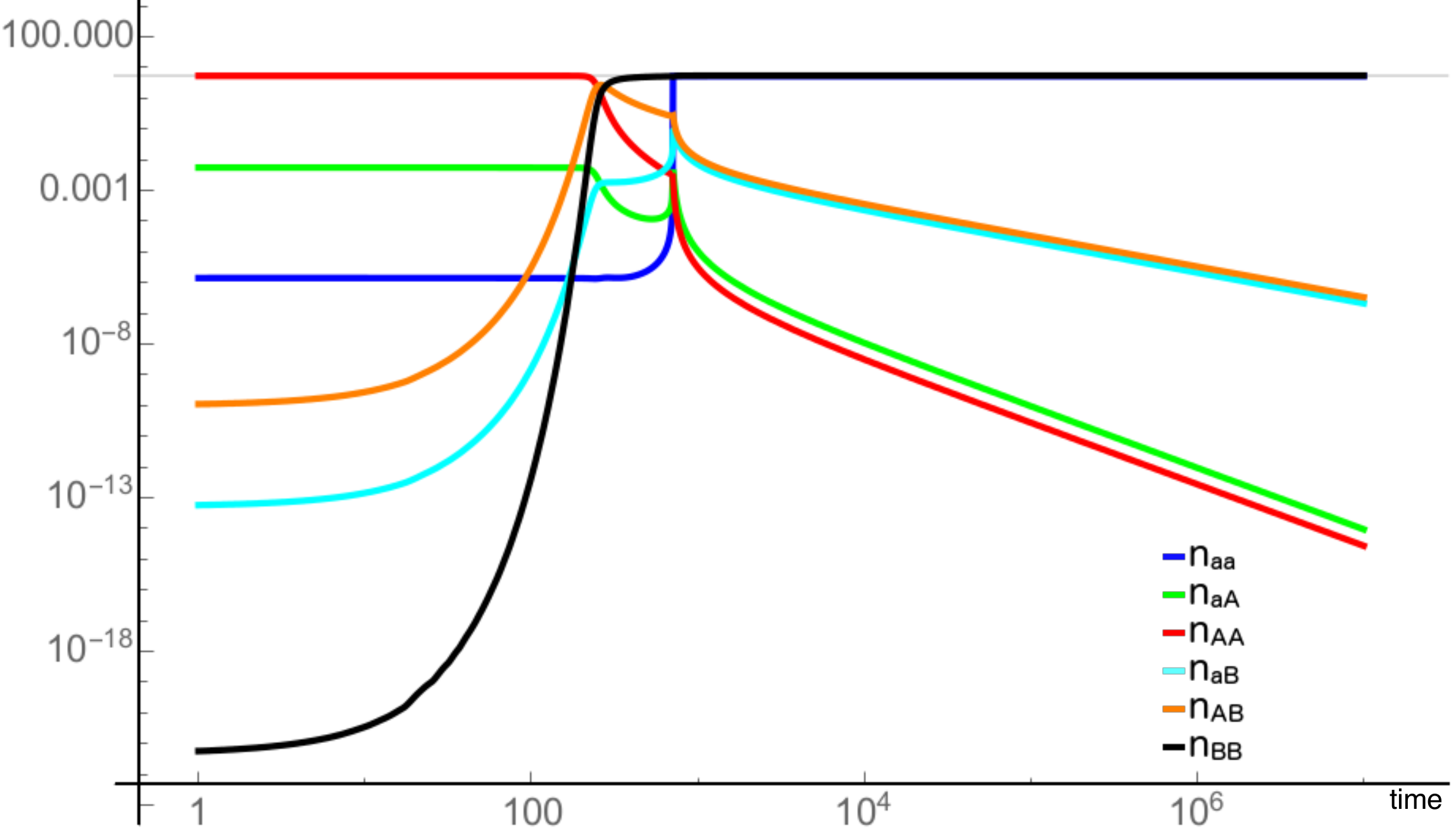}
		\caption{Log-plots 
			of $\{n_i(t),i\in\mathcal G\}$ for $\eta=0$ (top), $\eta=0.003$ (center) and $\eta=0.014$ (bottom).}
		\label{pic-loglog-eta-var}
	\end{center}
\end{figure}

\subsection{Phase 1: Perturbation of the 3-system $(aa,aA,AA)$ until $AB$ reaches $\OO(1)$  }\label{subsec-phase1}\emph{}\\

We start with initial conditions given by \eqref{init-condAA}-\eqref{init-condaB}. 
We will show that the mutant population, consisting of all individuals of phenotype $B$, grows up to some $\e_0>\e$ without perturbing the behaviour of the 3-system $(aa,aA,AA)$ in this time. Let
\begin{equation}
T_1:=T_{\e_0}^{{aa+aA+}aB+AB+BB}
\end{equation}

{
\begin{proposition}\label{prop-bounds}
With the initial conditions \eqref{init-condAA}-\eqref{init-condaB},
\begin{align}
T_1=T_{\e_0}^{AB}=\OO\left(\log(({\e_0}/{\e^3})^{\frac1{\Delta-\OO(\e_0)}}\right).
\end{align}
\end{proposition}
\begin{proof}
	Until $T_1$ the perturbation of the dynamics of the 3-system $(aa,aA,AA)$ is at most of order $\e_0$, by continuity of the solutions of an ODE with respect to initial conditions. We thus have $\bar n_A-\OO(\e_0)\leq n_{AA}(t)\leq\bar n_A+\OO(\e_0)$, as well as $n_{aa},n_{aA}\leq\OO(\e_0)$. With this rough bounds we will find finer bounds.
Heuristically,  the newborns of genotype $aA$ are still  in majority produced by recombination
		of $AA$ and $aA$, because the mutant population is not large enough to contribute.
		The newborns of genotype $aB$ are in majority produced by reproduction of the $aA$ population with the $B$ population. 
		Finally,  the newborns of genotype $aa$ are in majority produced by recombination of $aA$ and $aA$, because the only mutant population which could perturb their dynamics is $aB$, which is of smaller order.
\begin{enumerate}
\item We show that $n_{BB}\leq n_{AB}^2$. Indeed, we use Lemma \ref{lem-level}, note that $0=n_{BB}(0)\leq n_{AB}(0)^2$, and prove that when $n_{BB}=n_{AB}^2$ we have 
\begin{align}
\label{lhsBB}
\frac d{dt}(n_{BB}-n_{AB}^2)=\dot n_{BB}-2\dot n_{AB}n_{AB}=b_{BB}-2n_{AB}b_{AB}-d_{BB}+2n_{AB}d_{AB}<0.
\end{align}
Indeed, using Proposition \ref{aB<AB} we can bound \eqref{lhsBB} by
\begin{align}
&\sfrac f{\S_5}(n_{AB}+n_{BB})^2-\sfrac{2f}{\S_6}n_{AB}^2n_{AA}-n_{BB}(D-\D+cn_{AA})+2n_{AB}^2(D-\D+cn_{AA}+\OO(\e_0))\nonumber\\
&\leq n_{AB}^2\left(\sfrac f{\bar n_A}-f-\D+\OO(\e_0)\right)<0.
\end{align}
Thus $T_1=T_{\e_0}^{aa+aA+aB+AB}$.

\item We show that $n_{aB}\leq n_{aA}n_{AB}$. Indeed, we use Lemma \ref{lem-level}, note that $0=n_{aB}(0)\leq n_{aA}(0)n_{AB}(0)$, and prove that when $n_{aB}=n_{aA}n_{AB}$ we have 
\begin{align}
\label{lhsaB}
&\frac d{dt}(n_{aB}- n_{aA}n_{AB})=
\dot n_{aB}-\dot n_{aA}n_{AB}-\dot n_{AB}n_{aA}\\
&=b_{aB}-n_{AB}b_{aA}-n_{aA}b_{AB}-d_{aB}+n_{AB}d_{aA}+n_{aA}d_{AB}<0.
\end{align}
Indeed, using (1) we can bound \eqref{lhsaB} by
\begin{align}
&\sfrac f{2\S_5}n_{aA}n_{AB}-\sfrac{2f}{\S_6}n_{aA}n_{AA}n_{AB}+n_{aA}n_{AB}(D+cn_{AA}+\OO(\e_0))\nonumber\\
&\leq n_{aA}n_{AB}\left(\sfrac f{2\bar n_A}-f+\OO(\e_0)\right)<0.
\end{align}
Thus $T_1=T_{\e_0}^{aa+aA+AB}$.

\item We show that $n_{aa}\leq n_{aA}^2$. Indeed, we use Lemma \ref{lem-level}, note that $n_{aa}(0)\leq n_{aA}(0)^2$, and prove that when $n_{aa}=n_{aA}^2$ we have 
\begin{align}
\label{lhsaa}
\frac d{dt}(n_{aa}-n_{aA}^2)=
\dot n_{aa}-2\dot n_{aA}n_{aA}=b_{aa}-2n_{aA}b_{aA}-d_{aa}+2n_{aA}d_{aA}<0.
\end{align}
Indeed, using (1) and (2), we can bound \eqref{lhsaa} by
\begin{align}
&\sfrac f{4\S_6}n_{aA}^2-\sfrac{2f}{S_6}n_{aA}^2n_{AA}-n_{aA}^2(D+\D+cn_{AA})+2n_{aA}^2(D+cn_{AA}+\OO(\e_0))\nonumber\\
&\leq n_{aA}^2\left(\sfrac f{4\bar n_A}-f-\D+\OO(\e_0)\right)<0.
\end{align}
Thus $T_1=T_{\e_0}^{aA+AB}$.

\item Now, we show that $T_1=T_{\e_0}^{AB}$. First, we estimate the maximal time $n_{AB}$ would need to reach the level $\e_0$.
 For this upper bound on the time $T_{\e_0}^{AB}$, we have to construct  a minorising process for $n_{AB}$. Indeed, let us compare the birth and death rates using the results (1)-(3): 
		\begin{align}
		b_{AB}&\geq\sfrac12n_{AB}\frac{2fn_{AA}}{n_{AA}+\OO(\e_0)}=n_{AB}(f-\OO(\e_0)),\\
		d_{AB}&\leq n_{AB}(D-\D+c\bar n_A+\OO(\e_0))=n_{AB}(f-\D+\OO(\e_0)).
		\end{align}
		
		Hence, we get for the minorising process
		\begin{align}
		\dot n_{AB}&\geq n_{AB}(\D-\OO(\e_0)),\\
		n_{AB}(t)&\geq\e^3e^{(\D-\OO(\e_0))t},
		\end{align}
		and the time $T_{\e_0}^{AB}$ is at most of order $\OO\left(\log(({\e_0}/{\e^3})^{\frac1{\Delta-\OO(\e_0)}}\right)$.\\
		In a second step we show that $n_{aA}(t)\leq\e e^{\OO(\e_0)t}$ and that the minimal time $n_{aA}$ would need to reach $\e_0$ would be bigger than $\OO\left(\log(({\e_0}/{\e^3})^{\frac1{\Delta-\OO(\e_0)}}\right)$.
		For this we construct a majorising process on $n_{aA}$. The birth rate and the death rate can be bounded by using (1)-(3):
\begin{align}
b_{aA}&\leq \sfrac f{\S_6}n_{AA}n_{aA} +n_{aA}\OO(\e_0),\\
d_{aA}&\geq n_{aA}(f-\OO(\e_0)),
\end{align}
and we get that 
\begin{align}
\dot n_{aA}&\leq \OO(\e_0)n_{aA},\\
n_{aA}(t)&\leq\e e^{\OO(\e_0)t}.
\end{align}
Thus the minimal time $n_{aA}$ would need to reach $\e_0$ is $\OO\left(\log(\e_0/\e)^{\sfrac 1{\OO(\e_0)}}\right)$ which is bigger than the maximal time $n_{AB}$ would need to reach the level $\e_0$. Thus $T_1=T_{\e_0}^{AB}$.

\item It is left to show that $T_1\geq\OO\left(\log(({\e_0}/{\e^3})^{\frac1{\Delta-\OO(\e_0)}}\right)$. 
 For this lower bound on the time $T_1$, we have to construct a majorising process for $n_{AB}$:
		\begin{align}
		b_{AB}&\leq\sfrac f{\S_5}n_{AB}(n_{aA}+n_{AA}+n_{AB})+\sfrac f{\S_5}(n_{aB}+2n_{BB})(n_{aA}+n_{AA}+n_{AB})\nonumber\\
		&\leq fn_{AB}+ f(n_{aB}+2n_{BB})
		\intertext{Using (1),(2) and (4) we can further bound this by:}
		b_{AB}&\leq fn_{AB}+f(n_{aA}n_{AB}+2n_{AB}^2)\nonumber\\
		&\leq n_{AB}(f+\OO(\e_0)),\\
		d_{AB}&\geq n_{AB}(D-\D+c\bar n_A-\OO(\e_0))=n_{AB}(f-\D-\OO(\e_0)).
		\end{align}
		Hence, we get for the majorising process
		\begin{align}
		\dot n_{AB}\leq n_{AB}(\D+\OO(\e_0)),\\
		n_{AB}(t)\leq\e^3e^{(\D+\OO(\e_0))t},
		\end{align}
		and the time $T_1$ is at least of order $\OO\left(\log(({\e_0}/{\e^3})^{\frac1{\Delta-\OO(\e_0)}}\right)$.
\end{enumerate}
\end{proof}

\begin{corollary}\label{phase1}
With the initial conditions \eqref{init-condAA}-\eqref{init-condaB}, for all $t\in[0,T_1]$, the following 
statements hold:
\begin{enumerate}
\item\label{ABexp} $n_{AB}$ grows exponentially with rate $\D$. 
It reaches the level $\e_0$ in a time at most of order $\OO\left(\log\left(({\e_0}/{\e^3})^{\frac1{\Delta-\OO(\e_0)}}\right)\right)$.
\item\label{BB=AB^2} $ n_{BB} \leq n_{AB}^2 $,
\item $n_{aB}\leq\OO(\e^{1-\OO(\e_0)}\e_0)$, $n_{aA}\leq\OO(\e^{1-\OO(\e_0)})$, $n_{aa}\leq \OO(\e^{2-\OO(\e_0)})$ and\\ $\bar n_A-\OO(\e_0)\leq n_{AA}\leq\bar n_A+\OO(\e_0)$.
\end{enumerate}
\end{corollary}

\begin{proof}
\begin{enumerate}
\item[(1)]  is a direct consequence of Proposition \ref{prop-bounds}, see proof of points (4) and (5).
\item[(2)]  is a direct consequence of Proposition \ref{prop-bounds}, see proof of points (1).
\item[(3)]  follows from the proof of Proposition \ref{prop-bounds} together with  $\bar n_A-\OO(\e)\leq\S_5\leq\bar n_A+2\D\e_0$, which we prove now.\\
Using the results of Proposition \ref{prop-bounds} we construct a minorising and a majorising processes on $\S_5$:
\begin{align}
b_{\S_5}&\leq f\S_5+\OO(n_{aa}),\\
b_{\S_5}&\geq f\S_5-\OO(n_{aA}^2),\\
d_{\S_5}&\geq \S_5(D+c\S_5)-(\D+2\eta n_{aA})(n_{aB}+n_{AB}+n_{BB}),\\
d_{\S_5}&\leq \S_5(D+c\S_5+cn_{aa}),\\
\dot\S_5&\leq\S_5(f-D-c\S_5)+(\D+2\eta n_{aA})(n_{aB}+n_{AB}+n_{BB}),\\
\dot\S_5&\geq\S_5(f-D-c\S_5-cn_{aa}).
\end{align}
We use Lemma \ref{lem-level}: at time 0 the bounds are fulfilled; moreover,
at the upper bound we have $\dot\S_5<0$ and at the lower bound $\dot\S_5>0$, which ensures the claimed bounds.
%
\end{enumerate}
\end{proof}
}

Note that Corollary \ref{phase1} 
 implies that 
 \be
 T_1=T_{\e_0}^{aB+AB+BB}=
T_{\e_0}^{AB}=\OO\left(\log\left(\left({\e_0}/{\e^3}\right)^{\frac1{\Delta-\OO(\e_0)}}\right)\right).
\ee
\subsection{Phase 2: Perturbation of the 3-system $(AA,AB,BB)$ until $n_{aA}=\OO(n_{AA})$}\label{subsec-phase2}\emph{}\\
The initial conditions at the beginning of the second phase are:
\begin{align}
&n_{aa}(T_1)\leq\OO\left(\e^{2-\OO(\e_0)}\right)\\
&n_{aA}(T_1)\leq\OO\left(\e^{1-\OO(\e_0)}\right)\\
\bar n_A-\OO(\e_0)\leq&n_{AA}(T_1)\leq\bar n_A\\
&n_{aB}(T_1)\leq\OO\left(\e^{1-\OO(\e_0)}\e_0\right)\\
&n_{AB}(T_1)=\e_0\\
&n_{BB}(T_1)\leq\OO\left(\e_0^{2}\right)
\end{align}
Let $\delta>0$ (to be chosen sufficiently small in the sequel). Let
\begin{equation}
T_2:=T^{aA=\delta AA}\land T^{aB=\delta AB}\land T^{aa= aA\wedge aB}.
\end{equation}
As the process stays uniformly bounded in time, observe that until $T_2$ we automatically have 
\begin{equation}
\label{delta_bound}
n_{aa},n_{aA},n_{aB}\leq\OO(\delta).
\end{equation}
We will show that for $t\in[T_1, T_2]$ the system behaves as a main 3-system $(AA, AB, BB)$ plus perturbations of order $\delta$. The 3-system $(AA, AB, BB)$ behaves exactly the same as in \cite{BovNeu} since the parameters satisfy the same hypotheses (slightly lower death rate for phenotype $B$ than for phenotype $A$ individuals, and constant competition parameters).

Moreover, the crucial role of the parameter $\eta$ is that the population containing an allele $a$ only continues to grow in this phase when $\eta$ is large enough. 
This is due to the smaller competition that $aA$ feels from $BB$, the $aA$ population is thus higher and induces the growth of $aB$.

We start by considering how the growth of $aa, aA$- and $aB$ populations can perturb the 3-system $(AA,AB,BB)$.

\begin{lemma}
	\label{3system}
	Let $n_.^{up}(t)$ be the population of the unperturbed 3- system $(AA,AB,BB)$, that is the solution to \eqref{dyn-syst} with $n_{aa}=n_{aA}=n_{aB}=0$.
	The 3-system $(AA, AB, BB)$ satisfies
	\begin{align}
	\dot{n}_{BB}&\geq\dot{n}_{BB}^{up}-(n_{aA}+n_{aB})\left(\sfrac {f\left(\sfrac12n_{AB}+n_{BB}\right)^2}{(n_{AA}+n_{AB}+n_{BB})^2}+cn_{BB}\right),\\
	{\dot{n}_{BB}}&{\leq\dot{n}_{BB}^{up}+(n_{aA}+n_{aB})\left(\frac{f\left(\sfrac14n_{aB}+\sfrac12n_{AB}+n_{BB}\right)}{\S_5}+cn_{BB}\right)},\\
	\dot{n}_{AB}&\geq\dot{n}_{AB}^{up}-(n_{aa}+n_{aA}+n_{aB})\left(\frac {f(n_{AB}+n_{AA})\left(\sfrac12n_{AB}+n_{BB}\right)}{(n_{AA}+n_{AB}+n_{BB})^2}+cn_{AB}\right),\\
	\dot{n}_{AB}&\leq\dot{n}_{AB}^{up}+\sfrac f{\S_5}n_{aA}\left(\sfrac12n_{aB}+\sfrac12n_{AB}+n_{BB}\right)+\sfrac f{\S_5}n_{aB}\left(\sfrac12n_{AB}+n_{AA}\right),\\
	\dot{n}_{AA}&\geq\dot{n}_{AA}^{up}-(n_{aa}+n_{aA}+n_{aB})\left(\sfrac {f\left(\sfrac12n_{aA}+\sfrac12n_{AB}+n_{AA}\right)^2}{(n_{AA}+n_{AB}+n_{BB})2}+cn_{AA}\right),\\
	\dot{n}_{AA}&\leq\dot{n}_{AA}^{up}+\sfrac f{2\S_5}n_{aA}\left(\sfrac12n_{aA}+n_{AB}+n_{AA}\right).
	\end{align}
\end{lemma}
\begin{proof}
	We consider the rates of $AA,AB$ and $BB$ under the perturbation of $aa,aA$ and $aB$:
	\begin{align}\nonumber
	b_{BB}=&\frac f{\S_5}\left(\sfrac12n_{AB}+n_{BB}\right)^2+\frac {fn_{aB}\left(\sfrac14n_{aB}+\sfrac12n_{AB}+n_{BB}\right)}{\S_5}\\
	=&b_{BB}^{up}-\frac {f\left(\left(\sfrac12n_{AB}+n_{BB}\right)^2(n_{aA}+n_{aB})\right)}{\S_5(n_{AA}+n_{AB}+n_{BB})}+\frac {fn_{aB}\left(\sfrac14n_{aB}+\sfrac12n_{AB}+n_{BB}\right)}{\S_5},\\
	d_{BB}=&d_{BB}^{up}+cn_{BB}(n_{aB}+n_{aA})-\eta n_{aA}n_{BB}.
	\end{align}
	Thus,
	\begin{align}
	\dot{n}_{BB}\leq&\dot{n}_{BB}^{up}+ \sfrac f{\Sigma_5}n_{aB}\left(\sfrac14n_{aB}+\sfrac12n_{AB}+n_{BB}\right)+\eta n_{aA}n_{BB},\\
	\dot{n}_{BB}\geq&\dot{n}_{BB}^{up}-\frac {f(n_{aA}+n_{aB})\left(\sfrac12n_{AB}+n_{BB}\right)^2}{\S_5(n_{AA}+n_{AB}+n_{BB})}-cn_{BB}(n_{aA}+n_{aB}).
	\end{align}
		For the $AB$ population we get:
	\begin{align}
	b_{AB}=&\frac {2f\left(\sfrac12n_{AB}+n_{BB}\right)\left(\sfrac12n_{AB}+n_{AA}\right)}{\S_5}-\frac {fn_{aa}n_{AA}\left(\sfrac12n_{aB}+\sfrac12n_{AB}+n_{BB}\right)}{\S_5\S_6}+\frac {f(n_{AA}+n_{AB})}{2\S_5}n_{aB}\nonumber\\\nonumber
	&+\frac {fn_{aB}n_{AA}}{2\S_6}+\frac {fn_{aA}\left(\sfrac12n_{aB}+\sfrac12n_{AB}+n_{BB}\right)}{2\S_5}+\frac {fn_{aA}\left(\sfrac12n_{aB}+\sfrac12n_{AB}+n_{BB}\right)}{2\S_6}\\
	=&b_{AB}^{up}+n_{aA}\left(\sfrac12n_{aB}+\sfrac12n_{AB}+n_{BB}\right)\left(\sfrac f{2\S_5}+\sfrac f{2\S_6}\right)+n_{aB}n_{AA}\left(\sfrac f{2\S_5}+\sfrac f{2\S_6}\right)+\frac {fn_{aB}n_{AB}}{2\S_5}\nonumber\\
	&-\frac {fn_{aa}n_{AA}\left(\sfrac12n_{aB}+\sfrac12n_{AB}+n_{BB}\right)}{\S_5\S_6}-\frac{2f\left(\sfrac12n_{AB}+n_{AA}\right)\left(\sfrac12n_{AB}+n_{BB}\right)(n_{aA}+n_{aB})}{\S_5(n_{AA}+n_{AB}+n_{BB})},\\
	d_{AB}=&d_{AB}^{up}+cn_{AB}(n_{aB}+n_{aA}),\\
	\dot{n}_{AB}\leq&\dot{n}_{AB}^{up}+\sfrac f{\S_5}n_{aA}\left(\sfrac12n_{aB}+\sfrac12n_{AB}+n_{BB}\right)+\sfrac f{\S_5}n_{aB}n_{AA}+\sfrac f{2\S_5}n_{aB}n_{AB},\\
	\dot{n}_{AB}\geq&\dot{n}_{AB}^{up}-\sfrac {f\left(\sfrac12n_{aB}+\sfrac12n_{AB}+n_{BB}\right)}{\S_5\S_6}n_{aa}n_{AA}-\sfrac{2f\left(\sfrac12n_{AB}+n_{AA}\right)\left(\sfrac12n_{AB}+n_{BB}\right)}{\S_5(n_{AA}+n_{AB}+n_{BB})}(n_{aA}+n_{aB})-cn_{AB}(n_{aB}+n_{aA}).
	\end{align}
	And finally for the $AA$ population:
	\begin{align}\nonumber
	b_{AA}=&\frac {f\left(\sfrac12n_{AB}+n_{AA}\right)^2}{\S_5}+\frac {fn_{aA}n_{AB}}{4\S_5}-\frac {fn_{aa}n_{AA}\left(\sfrac12 n_{aA}\!+\!n_{AA}\!+\!\sfrac12n_{AB}\right)}{\S_5\S_6}+\frac {fn_{aA}\left(\sfrac12n_{aA}\!+\!n_{AA}\!+\!\sfrac 12n_{AB}\right)}{2\Sigma_6}\\
	=&b_{AA}^{up}-\frac {f\left(\sfrac12n_{AB}+n_{AA}\right)^2(n_{aA}+n_{aB})}{\S_5(n_{AA}+n_{AB}+n_{BB})}+\frac {fn_{aA}n_{AB}}{4\S_5}-\frac {fn_{aa}n_{AA}\left(\sfrac12 n_{aA}+n_{AA}+\sfrac12n_{AB}\right)}{\S_5\S_6}\nonumber\\
	&+\frac {fn_{aA}\left(\sfrac12n_{aA}+n_{AA}+\sfrac 12n_{AB}\right)}{2\Sigma_6},\\
	d_{AA}=&d_{AA}^{up}+cn_{AA}(n_{aa}+n_{aA}+n_{aB}),\\
	\dot{n}_{AA}\leq&\dot{n}_{AA}^{up}+\frac {fn_{aA}n_{AB}}{4\S_5}+\frac {fn_{aA}\left(\sfrac12n_{aA}+n_{AA}+\sfrac 12n_{AB}\right)}{2\Sigma_6},\\
	\dot{n}_{AA}\geq&\dot{n}_{AA}^{up}-\frac {f\left(\sfrac12n_{aA}+\sfrac12n_{AB}+n_{AA}\right)^2(n_{aa}+n_{aA}+n_{aB})}{\S_5(n_{AA}+n_{AB}+n_{BB})}-cn_{AA}(n_{aa}+n_{aA}+n_{aB}).
	\end{align}
\end{proof}

As solutions of a dynamical system are continuous with respect to its parameters (in particular with respect to $\delta$), the latter lemma shows that until $T_2$, the 3-system $(AA,AB,BB)$ is at most perturbed  by $\OO(\delta)$. We will show that $T_2$ diverges with $\e$. Thus, for small enough $\d$, $AB$ will have time to reach the small fixed value $\sqrt{\e _0}>0$ in this phase, and we can use the asymptotic decay $1/t$ of the $AB$ and $1/t^2$ of the  $AA$ populations,
 which is proved in \cite{BovNeu}. 
We now start to analyse the growth of the small $aa$-, $aA$- and $aB$ populations. The sum-process $\S_5$ plays a crucial role for the behaviour of the system in this phase   and we need finer bounds on it.
\begin{remark}
In the sequel we often use that $\S_5\leq\S_6\leq\S_5+\OO(\d)$, which follows from \eqref{delta_bound} and Lemma \ref{3system}.
\end{remark}

\begin{proposition}\label{prop-S5}
	The sum-process $\S_5=n_{aA}+n_{AA}+n_{aB}+n_{AB}+n_{BB}$ satisfies for all $t\in[T_1,T_2]$: 
	\begin{align}
	\label{prop-S5equ}
	\bar n_B-\frac\D{c\bar n_B}n_{AA}- \frac{\D^2}{c\bar n_B}n_{AA}\leq\S_5\leq\bar n_B-\frac\D{c\bar n_B}n_{AA}+ \frac{\D^2}{c\bar n_B}n_{AA}.
	\end{align}
\end{proposition}
\begin{proof}
We estimate a minorising process and a majorising process on $\S_5$:
	\begin{align}
	b_{\S_5}\leq&f(n_{AA}+n_{AB}+n_{BB})\nonumber\\
	&+f\frac{(n_{aA}+n_{aB})(\sfrac34n_{aA}+n_{AA}+\sfrac34n_{aB}+n_{AB}+n_{BB})}{\S_5}+\OO(\d)
	\leq f\S_5+\OO(\d),\\
		b_{\S_5}\geq&f(n_{AA}+n_{AB}+n_{BB})\nonumber\\
	&+f\frac{(n_{aA}+n_{aB})(\sfrac34n_{aA}+n_{AA}+\sfrac34n_{aB}+n_{AB}+n_{BB})}{\S_5}-\OO(\d)
	\geq f\S_5-\OO(\d), \\
	d_{\S_5}\leq&\S_5(D-\D+c\S_5)+\D(n_{AA}+n_{aA})-2\eta n_{aA}n_{BB}+\OO(\d),\\
	d_{\S_5}\geq&\S_5(D-\D+c\S_5)+\D(n_{AA}+n_{aA})-2\eta n_{aA}n_{BB}.
	\end{align}
	We get  
	\begin{align}
	\dot{\S}_5&\leq-c\S_5^2+\S_5(f-D+\D)-\D n_{AA}+\OO(\delta)\label{dotS5up},\\
	\dot{\S}_5&\geq-c\S_5^2+\S_5(f-D+\D)-\D n_{AA}-\OO(\delta)\label{dotS5down}.
	\end{align}
	We start with the proof of the upper bound. 
	We use Lemma \ref{lem-level} and show that when $\Sigma_5$ reaches the upper-bound, it decays faster than the latter.
	 Using \eqref{dotS5up} we compute $\dot\Sigma_5$ at the bound. Note that if  $\S_5=\bar n_B-\sfrac\D{c\bar n_B}n_{AA}+\sfrac{\D^2}{c\bar n_B}n_{AA}$, then $\S_5^2=\bar n_B^2-\sfrac{2\D}{c}n_{AA}+\sfrac{\D^2}{c^2\bar n_B^2}n_{AA}^2+\sfrac{2\D^2}cn_{AA}+\OO(\D^4)n_{AA}^2$, thus
	\begin{align}
	\dot\S_5\leq-\D^2n_{AA}-\sfrac{\D^2}{c\bar n_B^2}n_{AA}^2+\OO(\d)<0.
	\end{align}
	It is left to show that $\dot\S_5\leq-\sfrac\D{c\bar n_B}\dot n_{AA}+\sfrac{\D^2}{c\bar n_B}\dot n_{AA}$. Since we already know (cf.\ Lemma \ref{3system}) that $(AA,AB,BB)$ behaves like a 3-system with $\OO(\delta)$ perturbations, then  $\dot n_{AA}\leq0$, this finishes the proof of the upper bound.
	
	Now we check the lower bound. If $\S_5=\bar n_B-\frac\D{c\bar n_B}n_{AA}- \sfrac{\D^2}{c\bar n_B}n_{AA}$ then $\S_5^2=\bar n_B^2-\sfrac{2\D}cn_{AA}-\sfrac{\D^2}{c^2\bar n_B^2}n_{AA}-\sfrac{2\D^2}cn_{AA}$.
	Using  \eqref{dotS5down}, the derivative of $\S_5$ at the lower bound is thus bounded by
	\begin{align}
	\dot\S_5&\geq\D^2n_{AA}-\sfrac{\D^2}{c\bar n_B^2}n_{AA}-\OO\left(\D^3\right)n_{AA}^2-\OO(\d)\geq\D^2n_{AA}\left(1-\sfrac{1+\OO(\D)}{c\bar n_B}\right)-\OO(\d)>0.
	\end{align}
	For the second inequality we use that $n_{AA}\leq\bar n_A$ (Proposition \ref{aB<AB}).
	By \eqref{prop-S5equ} and Lemma \ref{lem-level}, it is enough to show that at the lower bound $\dot\S_5\geq-\sfrac\D{c\bar n_B}\dot n_{AA}$. Since $AA$ is decreasing we have to calculate a majorising process on $AA$:
	\begin{align}
	b_{AA}&\leq\sfrac f{\S_5}n_{AA}(n_{AA}+n_{AB})+\sfrac f{4\S_5}n_{AB}^2+\OO(\d).
	\intertext{For the death rate we use that $\S_5\geq\bar n_A-\OO(\d)$:}
	d_{AA}&\geq (f-\OO(\d))n_{AA},\\
	\dot n_{AA}&\leq-\sfrac{f}{\S_5}n_{AA}n_{BB}+\sfrac f{4\S_5}n_{AB}^2+\OO(\d).
	\end{align}
	Hence we have to show that the slope of $\S_5$ at the lower bound is bigger than the one of the lower bound, namely we show that $\D^2n_{AA}\left(1-\sfrac{1+\OO(\D)}{c\bar n_B}\right)-\OO(\d)\geq\sfrac{\D f}{c\bar n_B\bar n_A}\left(n_{AA}n_{BB}-\sfrac14n_{AB}^2\right)-\OO(\d\D)$, in the case $n_{AA}n_{BB}>\sfrac14n_{AB}^2$. This is equivalent to show that $\chi:= n_{AA}n_{BB}-\sfrac14n_{AB}^2\leq\sfrac{\D\bar n_A}f\left(c\bar n_B-1\right)n_{AA}$.	
	 For this we use once again Lemma \ref{lem-level} and estimate the derivative of $\chi$  from above with the help of minorising processes on $AA$ and $BB$ and a majorising process on $AB$. For bounding the death rates we use Proposition \ref{aB<AB}: 
	 \begin{align}
	b_{AA}&\geq\sfrac f{\S_5}n_{AA}(n_{AA}+n_{AB})+\sfrac f{4\S_5}n_{AB}^2-\OO(\d),\\
	d_{AA}&\leq (f+\D)n_{AA}+\OO(\d),\\
	\dot n_{AA}&\geq-\sfrac{f}{\S_5}n_{AA}n_{BB}-\D n_{AA}+\sfrac f{4\S_5}n_{AB}^2-\OO(\d).\label{LB_dotnAA}\\
	b_{BB}&\geq\sfrac f{\S_5}n_{BB}(n_{AB}+n_{BB})+\sfrac f{4\S_5}n_{AB}^2-\OO(\d),\\
	d_{BB}&\leq fn_{BB},\\
	\dot n_{BB}&\geq-\sfrac{f}{\S_5}n_{AA}n_{BB}+\sfrac f{4\S_5}n_{AB}^2+\OO(\d).\\
	b_{AB}&\leq\sfrac f{\S_5}n_{AB}\left(n_{AA}+\sfrac12n_{AB}+n_{BB}\right)+\sfrac {2f}{\S_5}n_{AA}n_{BB}+\OO(\d),\\
	d_{AB}&\geq (f-\D)n_{AB},\\
	\dot n_{AB}&\leq\sfrac{2f}{\S_5}n_{AA}n_{BB}-\sfrac f{2\S_5}n_{AB}^2+\D n_{AB}+\OO(\d).
	\end{align}
	 The derivative is given by:
		 \begin{align}\nonumber\label{chi}
	 \dot\chi&=\dot n_{AA}n_{BB}+n_{AA}\dot n_{BB}-\sfrac12\dot n_{AB}n_{AB}\\
	 &\leq-f\chi+\OO(\d).
	 \end{align}	
	 At the upper bound we get:
	 \begin{align}
	 \dot\chi\leq-\D\bar n_A(c\bar n_B-1)n_{AA}+\OO(\d)<0.
	 \end{align}
	 It is left to show that at the upper bound $\dot\chi\leq\sfrac{\D\bar n_A}f(c\bar n_B-1)\dot n_{AA}$. Using the minorising process $\dot n_{AA}\geq-\D n_{AA}-\sfrac{f}{\bar n_A}\chi-\OO(\d)$ (see \eqref{LB_dotnAA}), \eqref{chi} and Proposition \ref{aB<AB} we show that
	 \begin{align}
	 \sfrac{\D\bar n_A}f(c\bar n_B-1)\dot n_{AA}+f\chi-\OO(\d)\geq\chi\left(f-\D(c\bar n_B-1)\right)-\sfrac{\D^2\bar n_A}f(c\bar n_B-1)n_{AA}-\OO(\d)>0.
	 \end{align}
	 This finishes the proof of the lower bound.
\end{proof}

%
%
\begin{lemma}\label{lem-expS}
	 For $t\in[T_1,T_2]$ and  for $\D$ sufficiently small,
 		\begin{equation}
		\dot\Sigma_{aA,aB}\geq-\OO(\D)\Sigma_{aA,aB}.
		\end{equation}
\end{lemma}
\begin{proof}

	Using Proposition \ref{prop-S5}, we have the following bound on the process:
\begin{align}
b_{\Sigma_{aA,aB}}&\geq f\frac{n_{aA}(\sfrac12n_{aA}+n_{AA}+n_{aB}+n_{AB}+n_{BB})+n_{aB}(n_{AA}+\sfrac12n_{aB}+n_{AB}+n_{BB})}{n_{aA}+n_{AA}+n_{aB}+n_{AB}+n_{BB}}-\OO(\d n_{aA})\nonumber\\
&\geq f\Sigma_{aA,aB}-\OO(\delta n_{aA}),\\
d_{\Sigma_{aA,aB}}&=\Sigma_{aA,aB}(D-\D+c\S_5)-\eta n_{aA}n_{BB}+\D n_{aA}+cn_{aA}n_{aa}\nonumber\\
&\leq f\Sigma_{aA,aB}-n_{aA}(\eta n_{BB}-\D)+\OO(\D^2n_{AA})\S_{aA,aB}^2,\\
\dot{\Sigma}_{aA,aB}&\geq n_{aA}(\eta n_{BB}-\D-\OO(\delta))\OO(\D^2n_{AA})\S_{aA,aB}^2\geq n_{aA}(-\D-\OO(\delta))-\OO(\d\D^2n_{AA})\S_{aA,aB}\nonumber\\
&\geq \Sigma_{aA,aB} (-\D-\OO(\delta)).
\end{align}
\end{proof}

\begin{lemma}
\label{aa-S2}
For all $t\in[T_1,T_2]$ the $aa$ population is bounded by
\begin{align}\label{ub-aa}
\frac{f}{16\bar n_B(f+\D+\OO(\d))}\S_{aA,aB}^2\leq n_{aa}\leq\frac f{\bar n_A(D+\D)}\S_{aA,aB}^2.
\end{align}
\end{lemma}
Observe that this implies $T_2=T^{aA=\delta AA}\wedge T^{aB=\delta AB}$.

\begin{proof}
First observe that the inequality is satisfied at $t=T_1$.
We start with the upper bound and show that $n_{aa}$ would decrease at this bound. For this we estimate a majorising process on $aa$:
\begin{align}
b_{aa}&\leq\sfrac f{\S_3}n_{aa}\left(\sfrac12n_{aA}+n_{aa}\right)+\sfrac f{4\S_5}\S_{aA,aB}^2+\sfrac f{2\S_5}n_{aA}n_{aa},\\
d_{aa}&\geq n_{aa}(D+\D),\\
\dot{n}_{aa}&\leq\sfrac f{\S_3}n_{aa}^2+\sfrac {f}{\S_3}n_{aa}n_{aA}+\sfrac f{4\S_5}\S_{aA,aB}^2-n_{aa}(D+\D).
\end{align}
We calculate the slope of this process at the upper bound using Proposition \ref{prop-S5} and \eqref{delta_bound}:
\begin{align}
\dot{n}_{aa}&\leq\sfrac f{4\S_5}\S_{aA,aB}^2-\sfrac f{\bar n_A}\S_{aA,aB}^2+\OO(\S_{aA,aB}^2n_{aA})
\leq-\sfrac{3f-\OO(\d)}{4\bar n_A}\S_{aA,aB}^2<0.
\end{align}
By Lemma \ref{lem-level}, to ensure that \eqref{ub-aa} stays an upper bound it is enough to show that the right-hand side satisfies
\begin{align}
-\sfrac{3f-\OO(\d)}{4\bar n_A}\S_{aA,aB}^2\leq\sfrac{2f}{\bar n_A(D+\D)}\dot{\S}_{aA,aB}\S_{aA,aB}.
\end{align}
This is a consequence of Lemma \ref{lem-expS}.\\
For the lower bound we proceed similarly. Since $\sfrac1\S_6=\sfrac1\S_5-\sfrac{n_{aa}}{\S_5\S_6}$ and with the knowledge of the upper bound, we estimate a minorising process on $aa$:
\begin{align}
b_{aa}&\geq\sfrac f{4\S_5}\S_{aA,aB}^2+\sfrac f{4\S_6}n_{aa}n_{aA}\left(2-\OO(\d)\right)\geq\sfrac f{4\S_5}\S_{aA,aB}^2.\\
\intertext{The estimation of the death rate follows from Proposition \ref{aB<AB} and \eqref{delta_bound}}
d_{aa}&\leq n_{aa}(f+\D+\OO(\d)),\\
\dot{n}_{aa}&\geq\sfrac f{4\S_5}\S_{aA,aB}^2-n_{aa}(f+\D+\OO(\d)).
\end{align}
With Proposition \ref{prop-S5} we see that at the lower bound the process  increases:
\begin{align}
\dot{n}_{aa}&\geq\left(\sfrac f{4\S_5}-\sfrac f{16\bar n_B}\right)\S_{aA,aB}^2>0.
\end{align}
By Lemma \ref{lem-level}, it is left to show that  $\dot n_{aa}\geq\sfrac f{8\bar n_B(f+\D+\OO(\d))}\dot\S_{aA,aB}\S_{aA,aB}$ at the lower bound.
Thus we have to calculate a majorising process on $\S_{aA,aB}$:
\begin{align}
b_{\S_{aA,aB}}&\leq f\S_{aA,aB}+\OO\left(\S_{aA,aB}^2\right),\label{braa}
\intertext{Using Proposition \ref{prop-S5}, Proposition \ref{aB<AB} and that $\eta\leq c$ we get for the death rate}\nonumber
d_{\S_{aA,aB}}&\geq (D-\D+c\S_5)\S_{aA,aB}+n_{aA}(\D-\eta n_{BB})\\\nonumber
&\geq (f-\D(1+\OO(\D)))\S_{aA,aB}-(f-D)\S_{aA,aB}\\
&=(D-\D(1+\OO(\D)))\S_{aA,aB},\\
\dot\S_{aA,aB}&\leq(f-D+\D(1+\OO(\D)))\S_{aA,aB}+\OO\left(\S_{aA,aB}^2\right).
\end{align}
Besides we get at the lower bound using Proposition \ref{prop-S5}
\begin{align}\nonumber
&\sfrac{f(f-D+\D)}{8\bar n_B(f+\D+\OO(\d))}\S_{aA,aB}^2-\left(\sfrac{f}{4\S_5}-\sfrac f{16\bar n_B}\right)\S_{aA,aB}^2+\OO(\S_{aA,aB}^3)&\\
&=-\sfrac{f}{16\bar n_B}\sfrac{f+2D+\D(1-\OO(\D))}{f+\D+\OO(\d)}\S_{aA,aB}^2+\OO\left(\S_{aA,aB}^3\right)<0.
\end{align}
This finishes the proof of the lower bound.
\end{proof}

Let 
\begin{equation}
T_= :=\inf\{t>T_1: n_{aA}(t)=n_{aB}(t)\}.
\end{equation} 

\begin{proposition}\label{prop-T1T=}
	For all $t\in[T_1,T_=]$, 
	\begin{align}
	n_{aB}\leq n_{aA}=\OO(\e).
	\end{align}
\end{proposition}
\begin{proof}
	In this time interval the newborns of genotype $aA$ are in majority produced by reproductions of a population of order one, namely $AB$ or $AA$, with the population $aA$. Since $n_{aA}$ feels  competition from a macroscopic population ($AA$, $AB$ or $BB$) the $aA$ population stays of order $\OO(\e)$.
We make this more rigorous.
To show this we consider a majorising process on $aA$ and use Proposition \ref{prop-S5}, and Lemma \ref{aa-S2}:
\begin{align}
b_{aA}&\leq fn_{aA}-\sfrac f{\S_5}n_{aA}(n_{BB}+\sfrac12n_{AB})+\sfrac f{2\S_5}n_{aB}(2n_{AA}+n_{AB})+\OO\left(\S_{aA,aB}^2\right),\\
d_{aA}&\geq n_{aA}(f+\D-\sfrac{\D}{\bar n_B}n_{AA}-\eta n_{BB}-\OO\left(\D^2n_{AA})\right),\\\nonumber
\dot{n}_{aA}&\leq-n_{aA}\left(n_{BB}\sfrac{f-\eta\S_5}{\S_5}+\sfrac f{2\S_5}n_{AB}+\D\left(1-\sfrac{n_{AA}}{\bar n_B}\right)-\OO\left(\D^2n_{AA}\right)\right)+\sfrac f{\S_5}n_{aB}(\sfrac12n_{AB}+n_{AA}+\OO(\d))\\\nonumber
&\leq-n_{aA}\left(n_{BB}\sfrac{D+\D}{\S_5}+\sfrac f{2\S_5}n_{AB}+\D\left(1-\sfrac{n_{AA}}{\bar n_B}\right)-\OO\left(\D^2n_{AA}\right)\right)+\sfrac f{\S_5}n_{aB}(\sfrac12n_{AB}+n_{AA}+\OO(\d))\\
&\leq-n_{aA}\left(\sfrac f{\S_5}\left(\sfrac{D+\D}{f}n_{BB}+\sfrac12n_{AB}\right)+\D\left(1-\sfrac{n_{AA}}{\bar n_B}
\right)-\OO\left(\D^2n_{AA}\right)\right)+\sfrac f{\S_5}n_{aB}(\sfrac12n_{AB}+n_{AA}+\OO(\d)).
\end{align}
By Proposition \ref{3system} and Proposition 3.4 in \cite{CMM13} there exists a time $t_0=\OO(1)$ such that the expression in the first bracket becomes bigger than the expression in the second bracket. Thus $n_{aA}$ decreases after $t_0$ and since $aA$ does not exceed $\OO(\e)$ until $t_0$ it will stay smaller or equal to $\OO(\e)$ until $T_=$.
\end{proof}

We show that as soon as $aB$ crosses $aA$ the $BB$ population is already bigger than or equal to the $AA$ population. First we construct a process that provide an upper bound on $aB$:

\begin{lemma}\label{ub-aB-aA}
For all $t\in[T_1,T_2]$ the $aB$ population is upper bounded by
\begin{align}
n_{aB}\leq\frac{n_{AB}+2n_{BB}+\sfrac{2\Delta}c}{n_{AB}+2n_{AA}}n_{aA}\equiv C(t)n_{aA}.
\end{align}
\end{lemma}
\begin{proof} First observe that the bound is fulfilled at $t=T_1$.
Similarly to the proof of Lemma \ref{aa-S2} we estimate a majorising process on $aB$ given by:
\begin{align}
\dot{n}_{aB}&\leq-n_{aB}\left(\sfrac f{2\Sigma_5}(n_{AB}+2n_{AA})-\sfrac\Delta{\bar n_B}n_{AA}-\OO(\D^2n_{AA})\right)+n_{aA}\sfrac f{2\Sigma_5}(n_{AB}+2n_{BB}+\OO(\d)).
\end{align}
By Lemma \ref{lem-level}, we have to show that as soon as $aB$ reaches the upper bound it decreases faster than the bound, thus we calculate the slope of the majorising process at this value: 
\begin{align}\nonumber
\dot{n}_{aB}&\leq-\sfrac f{2\S_5}\left(n_{AB}+2n_{BB}+\sfrac{2\D}c-\OO(\D^2n_{AA})\right)n_{aA}+\sfrac{\D\left(n_{AB}+2n_{BB}+2\D/c\right)}{\bar n_B(n_{AB}+2n_{AA})}n_{AA}n_{aA}+\sfrac {f(n_{AB}+2n_{BB}+\OO(\d))}{2\S_5}n_{aA}\nonumber
\intertext{Using Proposition \ref{prop-S5} and that $\sfrac{n_{AA}}{n_{AB}+2n_{AA}}\leq\sfrac12$ we get}\nonumber
\dot{n}_{aB}&\leq-\sfrac{\D f-\OO(\D^2n_{AA})}{c\S_5}n_{aA}+\sfrac{\D}{\S_5}\left(\sfrac12n_{AB}+n_{BB}+\sfrac{\D}{c}\right)n_{aA}\\\nonumber
&\leq\sfrac{\D+\OO(\D^2n_{AA})+\OO(\d)}{\S_5}n_{aA}\left(\bar n_B+\sfrac\D c-\sfrac fc\right)\\
&=-\sfrac{\D+\OO(\D^2n_{AA})+\OO(\d)}{c\S_5}\left(D-2\D\right)n_{aA}\leq0.
\end{align}
We have to show that, at the upper bound, $\dot{n}_{aB}<C(t)\dot{n}_{aA}+\dot{C}(t)n_{aA}$. Since the 3-system $(AA,AB,BB)$ converges towards $(0,0,\bar n_B)$,  we know from Proposition 3.4 in \cite {CMM13} that $C(t)$ is a monotone increasing function and hence $\dot{C}(t)\geq0$. Thus if we can show that $\dot{n}_{aB}< C(t)\dot{n}_{aA}$ we are done.
For this we have to calculate the slope of the minorising process on $aA$ when $aB$ reaches the upper bound. This process is given by:
\begin{align}
\dot{n}_{aA}&\geq-n_{aA}\left(\sfrac f2+\D-\eta n_{BB}+\sfrac f{2\Sigma_5}(n_{BB}-n_{AA})+\OO(\d)\right)+n_{aB}\sfrac f{2\Sigma_5}(n_{AB}+2n_{AA}).
\end{align}
The slope at the upper bound is:
\begin{align}\nonumber
\dot{n}_{aA}&\geq-n_{aA}\left(\sfrac f2+\D-\eta n_{BB}+\sfrac f{2\Sigma_5}(n_{BB}-n_{AA})-\sfrac f{2\S_5}\left(n_{AB}+2n_{BB}+\sfrac{2\D}c+\OO(\d)\right)\right)\\\nonumber
&\geq-n_{aA}\left(\D-\eta n_{BB}-\sfrac {\D f}{c\Sigma_5}+\OO(\d)\right)\\
&\geq n_{aA}\left(\D\sfrac{D-\D}{c\S_5}+\eta n_{BB}-\OO(\d)\right)\geq0.
\end{align}
Since $C(t)>0$ this finishes the proof.
\end{proof}

\begin{lemma}
	\label{AA=BB}
	We have $T_=< T_2$. Moreover,
	\begin{align}
	n_{AA}(T_=)\leq n_{BB}(T_=)+\OO(\D).
	\end{align}
\end{lemma}

\begin{proof}
	We first show that $T_=< T_2$. Using Proposition \ref{prop-S5} we construct two processes that provide an upper bound and a lower bound on $n_{aB}$. For the birth rates of these processes we use that until time $T_=$, $n_{aB}\leq n_{aA}=\OO(\e)$ (Proposition \ref{prop-T1T=}) and Lemma \ref{aa-S2}:
	\begin{align}
	b_{aB}&\geq f n_{aB}-\sfrac f{\S_5}n_{aB}(\sfrac12n_{AB}+n_{AA}+\OO(\e))+\sfrac f{\S_5}n_{aA}(\sfrac12n_{AB}+n_{BB}-\OO(\e^2)),\\
	b_{aB}&\leq f n_{aB}-\sfrac f{\S_5}n_{aB}(\sfrac12n_{AB}+n_{AA})+\sfrac f{\S_5}n_{aA}(\sfrac12n_{AB}+n_{BB}+\OO(\e^2)).
	\intertext{For the death rates we use Proposition \ref{prop-S5}:}
	d_{aB}&\leq n_{aB}f,\\
	d_{aB}&\geq n_{aB}(f-\sfrac{\D}{\bar n_B}n_{AA}-\OO(\D^2n_{AA})),\\
	\dot n_{aB}&\leq-n_{aB}\left(\frac {f(\sfrac12n_{AB}+n_{AA})}{\S_5}-\sfrac{\D}{\bar n_B}n_{AA}-\OO(\D^2n_{AA})\right)+n_{aA}\frac{f(\sfrac12n_{AB}+n_{BB}+\OO(\e^2))}{\S_5},\\
	\dot n_{aB}&\geq-n_{aB}\frac {f(\sfrac12n_{AB}+n_{AA}+\OO(\e))}{\S_5}+n_{aA}\frac {f(\sfrac12n_{AB}+n_{BB}-\OO(\e^2))}{\S_5}.\label{majo-aB}
	\end{align}
We first show that $T_=<\infty$.  We know that the 3-system $(AA,AB,BB)$ converges to $(0,0,\bar n_B)$ and that $n_{aB}\leq n_{aA}=\OO(\e)$ (Proposition \ref{prop-T1T=}), for $t\leq T_=$. If we assume that $n_{aB}<n_{aA}$, then  \eqref{majo-aB} implies that at some time $t_0$, where $n_{AB}+2n_{BB}$ is already macroscopic, we have
\begin{align}
\dot{n}_{aB}\geq\OO(\e), \quad n_{aB}\geq\OO(\e)t.
\end{align}
Thus the time  $n_{aB}$ needs to reach $n_{aA}=\OO(\e)$ is of order $\OO(1)$. This time is shorter than $T_{aA=\delta AA}$. Indeed, suppose the contrary, then by Proposition \ref{prop-T1T=} $n_{aA}$ does not exceed $\OO(\e)$ before $T_2$, and thus $T^{aA=\delta AA}\geq T^{AA}_{\OO(\e/\delta)}=\OO\left((\delta/\e)^2\right)$ which diverges with $\e$. A similar reasoning shows that $T_=<T^{aB=\delta AB}$.  Hence $T_=<T_2$.\\
It is left to show that $n_{AA}(T_=)\leq n_{BB}(T_=)+\OO(\D)$.
From Lemma \ref{ub-aB-aA} we deduce that at $T_=$, 
	\begin{align}		
	\sfrac12n_{AB}+n_{AA}&\leq\sfrac12n_{AB}+n_{BB}+\sfrac{\D}c\nonumber\\
	n_{AA}&\leq n_{BB}+\OO(\D).
	\end{align}
\end{proof}

\begin{lemma}
\label{bound-AB}
For all $t\in[T_1,T_2]$ the $AB$ population is  bounded by
\begin{enumerate}
\item
$n_{AB}\geq2\sqrt{\bar n_Bn_{AA}}-2n_{AA}\left(1+\sfrac\D{c\bar n_B}\right)$,
\item
$n_{AB}\leq2\sqrt{\bar n_Bn_{AA}\left(1+\sfrac\D f\right)}-2n_{AA}$.
\end{enumerate}
\end{lemma}
\begin{proof}
\item[(1)] The proof works like the one of Lemma \ref{aa-S2}. First observe that the bound holds at $t=T_1$ which can be proven by using Corollary \ref{phase1} and Proposition \ref{prop-S5}.
Then we calculate a minorising process on $AB$:
\begin{align}
	b_{AB}&\geq f(2n_{AA}+n_{AB})-\sfrac f{\S_5}(2n_{AA}+n_{AB})(n_{AA}+\sfrac12n_{AB}+\OO(\d)),\\
	d_{AB}&\leq fn_{AB},\\
	\dot n_{AB}&\geq-n_{AB}\sfrac f{\S_5}\left(\sfrac12n_{AB}+n_{AA}+\OO(\d)\right)+2fn_{AA}-\sfrac{2f}{\S_5}n_{AA}\left(\sfrac12n_{AB}+n_{AA}+\OO(\d)\right). 
	\end{align}
We use Proposition \ref{prop-S5} and show that this minorising process would increase quicker than the lower bound if $AB$  reaches it:
\begin{align}
\dot n_{AB}\geq&-\sfrac {2f}{\S_5}\left(\sqrt{\bar n_Bn_{AA}}-n_{AA}\left(1+\sfrac\D{c\bar n_B}\right)\right)\left(\sqrt{\bar n_Bn_{AA}}-\sfrac\D{c\bar n_B}n_{AA}+\OO(\d)\right)\nonumber\\\nonumber
&+2fn_{AA}-\sfrac{2f}{\S_5}n_{AA}\left(\sqrt{\bar n_Bn_{AA}}-\sfrac\D{c\bar n_B}n_{AA}+\OO(\d)\right)\\
\geq&\sfrac{2f}{\S_5}\sfrac\D{c\bar n_B}n_{AA}\left(2 \sqrt{\bar n_Bn_{AA}}-\sfrac{\D}{c\bar n_B}n_{AA}+\OO(\d)\right)>0.
\end{align}
From Lemma \ref{lem-level} it is left to show that at the lower bound,
\begin{align}\label{lb-AB}
\dot{n}_{AB}\geq\frac{\bar n_B\dot{n}_{AA}}{\sqrt{\bar n_Bn_{AA}}}-2\dot{n}_{AA}\left(1+\sfrac\D{c\bar n_B}\right).
\end{align}
For this we calculate a majorising process on $AA$ using Proposition \ref{prop-S5} and \eqref{delta_bound}:
\begin{align}
b_{AA}&\leq\sfrac f{\S_5}n_{AA}(n_{AA}+n_{AB})+\sfrac f{4\S_5}n_{AB}^2+\OO(\d),\\
d_{AA}&\geq (f-\OO(\D^2))n_{AA},\\
\dot{n}_{AA}&\leq-n_{AA}\left(f-\sfrac f{\S_5}(n_{AA}+n_{AB})-\OO(\D^2)\right)+\sfrac f{4\S_5}n_{AB}^2+\OO(\d).
\end{align}
If we now insert the lower bound and use Proposition \ref{prop-S5} we get, 
\begin{align}
\dot{n}_{AA}\leq-\sfrac{2f}{\S_5}\sfrac\D{c\bar n_B}n_{AA} \sqrt{\bar n_Bn_{AA}}+\OO(\D^2)<0.
\end{align}
Thus \eqref{lb-AB} is fulfilled. 
\item[(2)] First, observe that the upper bound is fulfilled at $t=T_1$. We then have to estimate a majorising process on $AB$. For the birth rate we use \eqref{delta_bound} and for the death rate we use Proposition \ref{prop-S5}.
\begin{align}
	b_{AB}&\leq f(2n_{AA}+n_{AB})-\sfrac f{\S_5}(2n_{AA}+n_{AB})(n_{AA}+\sfrac12n_{AB})+\OO(\d),\\
	d_{AB}&\geq n_{AB}(D-\D+c\bar n_B-\sfrac{\D}{\bar n_B}n_{AA}-\OO(\D^2n_{AA}))\\
	&= n_{AB}(f-\sfrac{\D}{\bar n_B}n_{AA}-\OO(\D^2n_{AA})),\\
	\dot n_{AB}&\leq-\sfrac f{2\bar n_B}n_{AB}^2-n_{AB}\sfrac{2f-\D}{\bar n_B}n_{AA}+2fn_{AA}-\sfrac{2f}{\bar n_B}n_{AA}^2+\OO(\D^2n_{AA}). 
	\end{align}
	As before we calculate the slope of this majorising process if it would reach the upper bound:
	\begin{align}
	\dot{n}_{AB}\leq-\sfrac{2\D}{\bar n_B}n_{AA}^2+\OO(\D^2n_{AA})<0.
	\end{align}
	By Lemma \ref{lem-level} we have to show that 
	\begin{align}
	\dot{n}_{AB}<\dot{n}_{AA}\left(\sfrac{\bar n_B\left(1+\D/f\right)}{\sqrt{\bar n_Bn_{AA}\left(1+\D/f\right)}}-2\right).
	\end{align}
	For this we calculate the slope of a minorising process on $AA$:
	\begin{align}
	b_{AA}&\geq\sfrac f{\S_5}n_{AA}(n_{AA}+n_{AB})+\sfrac f{4\S_5}n_{AB}^2-\OO(\d^2).
	\intertext{Using Proposition \ref{prop-S5} the death rate is bounded by:}
	d_{AA}&\leq n_{AA}(f+\D+\OO(\d^2)),\\
	\dot{n}_{AA}&\geq-n_{AA}\left(f-\sfrac f{\S_5}(n_{AA}+n_{AB})+\D+\OO(\d^2)\right)+\sfrac f{4\S_5}n_{AB}^2.
	\end{align}
	At the upper bound $AA$ would start to increases:
	\begin{align}
	\dot{n}_{AA}\geq\sfrac\D{\bar n_B}n_{AA}^2-\OO(\d^2)>0.
	\end{align}
	Thus we get
	\begin{align}
	\dot{n}_{AA}\left(\sfrac{\bar n_B\left(1+\D/f\right)}{\sqrt{\bar n_Bn_{AA}\left(1+\D/f\right)}}-2\right)-\dot{n}_{AB}\geq\sfrac{\D\left(1+\D/f\right)}{\sqrt{\bar n_Bn_{AA}\left(1+\D/f\right)}}n_{AA}^2-\OO(\D^2n_{AA})>0.
	\end{align}
This finishes the proof of (2).
\end{proof}

The following Proposition is a statement for the 3-system $(AA,AB,BB)$ but it remains valid until
time  $T_2$ in the 6-system $(aa,aA,AA,aB,AB,BB)$, for $\d<\D$.
\begin{proposition}
	\label{maxAB}
	The maximal value $n_{AB}^{max}$ of $n_{AB}$ in $[T_1,T_2]$ is bounded by
	\begin{align}
	\frac{\bar n_B}2-\OO(\D)\leq n_{AB}^{max}\leq\frac{\bar n_B}2+\OO(\D).
	\end{align}
	Moreover, let $T_{AB}^{max}$ be the time when $n_{AB}$ takes its maximum, then $n_{AA}$ and $n_{BB}$ are bounded by
	\begin{align}
	\frac{\bar n_B}4-\OO(\D)\leq n_{AA}(T_{AB}^{max})\leq \frac{\bar n_B}4+\OO(\D),\\
	\frac{\bar n_B}4-\OO(\D)\leq n_{BB}(T_{AB}^{max})\leq \frac{\bar n_B}4+\OO(\D).
	\end{align}
\end{proposition}

\begin{proof}
From Lemma \ref{bound-AB} (1) we get that
\begin{align}
n_{AB}\geq2\sqrt{\bar n_Bn_{AA}}-2n_{AA}\left(1+\sfrac\D{c\bar n_B}\right).
\end{align}
We look for the value of $AA$ where the expression on the right hand side takes on its minimum, thus we have to derivate $n_{AA}$ and set it to zero:
	\begin{align}
	\frac{\bar n_B}{\sqrt{\bar n_Bn_{AA}}}-\left(2+\sfrac{2\D}{c\bar n_B}\right)&=0\\
	\bar n_B^2&=\left(4+8\sfrac\D{c\bar n_B}+\OO(\D^2)\right)\bar n_Bn_{AA}\\
	\frac{\bar n_B}4-\OO(\D)&=n_{AA}.
	\end{align}

	If we insert this in $n_{AB}$  we get the lower bound:
	\begin{align}
	n_{AB}\geq\sfrac{\bar n_B}2+\OO(\D).
	\end{align}
	For the upper bound on $n_{AB}$ we proceed similarly.
	From Lemma \ref{bound-AB} (2) we get 
	\begin{align}
n_{AB}\leq2\sqrt{\bar n_Bn_{AA}\left(1+\sfrac\D f\right)}-2n_{AA}.
\end{align}
Setting the derivation of the rhs to zero gives:
\begin{align}
0&=\sfrac{\bar n_B\left(1+\sfrac\D f\right)}{\sqrt{\bar n_Bn_{AA}\left(1+\sfrac\D f\right)}}-2\\
n_{AA}&=\frac{\bar n_B}4+\OO(\D).
\end{align}
Finally we get
	\begin{align}
	\label{davor}
	n_{AB}\leq\sfrac{\bar n_B}2-\OO(\D)\quad \text{ and } \quad
	n_{AA}=\sfrac{\bar n_B}4-\OO(\D).
	\end{align}
\end{proof}
\begin{remark}
	Note that by Proposition \ref{prop-S5}, \eqref{davor} and \eqref{delta_bound} $n_{AA}=n_{BB}\pm\OO(\D)=\sfrac{\bar n_B}4\pm\OO(\D)$ as soon as $n_{AB}$ reaches its maximal value.
\end{remark}

\begin{proposition}\label{phase2} 
\label{aA<e} For all $t\in[T_1,T_2]$, 
		\begin{align}
		n_{aA}\leq\OO(\e)\vee n_{aB}.
		\end{align}
\end{proposition}

\begin{proof}
		For $t\leq T_=$ this follows from Proposition \ref{prop-T1T=}.
For $t>T_=$ we show this by constructing  a majorising process of $n_{aA}(t)$. For the birth rate we use \eqref{delta_bound} and Lemma \ref{aa-S2}:
		\begin{align}\nonumber
		b_{aA}\leq &f\frac {(n_{aA}+n_{aB})(2n_{AA}+n_{AB}+\OO(\d))}{2\S_5}\\
		\leq&\sfrac {f+\OO(\d)}2(n_{aA}+n_{aB})+\frac {f(n_{AA}-n_{BB})}{2\bar n_A}(n_{aA}+n_{aB}).
		\intertext{Using Proposition \ref{prop-S5} and that $\sfrac{n_{AA}}{\bar n_B}<1$ we get the death rate:}
		d_{aA}\geq&n_{aA}\left(D+c\bar n_B-\sfrac\D{\bar n_B}n_{AA}-\eta n_{BB}-\OO(\D^2n_{AA})\right)\nonumber\\
		\geq&n_{aA}(f-\eta n_{BB}),\\
		\dot n_{aA}\leq&-n_{aA}\left(\sfrac f2-\sfrac {f(n_{AA}-n_{BB})}{2\bar n_A}-\eta n_{BB}-\OO(\d)\right)+n_{aB}\left(\sfrac f2+\sfrac {f(n_{AA}-n_{BB}+\OO(\d))}{2\bar n_A}\right).\label{n_aA(T_=)}
		\end{align}
	By Lemma \ref{lem-level}, it is left to show that $\dot{n}_{aA}\leq\dot{n}_{aB}$ whenever $n_{aA}=n_{aB}$.  At this upper bound we have 
	$\dot{n}_{aA}\leq {n_{aB}} (\frac f{\bar n_A}(n_{AA}-n_{BB})+\eta n_{BB}+\OO(\d))$.
	We now calculate a minorising process on $n_{aB}$. The birth rate can be estimate by using Proposition \ref{prop-S5}, \eqref{delta_bound} and that $\sfrac{1/2n_{AB}+n_{BB}}{\S_5}\geq\sfrac12-\OO(\D)$, for $t>T_=$:
	\begin{align}
	b_{aB}&\geq\sfrac f{2\S_5}(n_{aA}+n_{aB})(n_{aB}+n_{AB}+2n_{BB}).\\
	\intertext{For the death rate Proposition \ref{prop-S5} is used:}
	d_{aB}&\leq n_{aB}(D-\Delta+c\bar n_B)=fn_{aB},\\
	\dot n_{aB}&\geq\sfrac f{2\S_5}n_{aA}(n_{aB}+n_{AB}+2n_{BB})-\sfrac f{2\S_5}n_{aB}(2n_{AA}+2n_{aA}-n_{aB}-n_{AB}).
	\end{align}
	Thus $\dot{n}_{aB}\geq\sfrac f{\S_5}n_{aB}(n_{BB}-n_{AA}+n_{AB})$ whenever $n_{aA}=n_{aB}$, and hence $\dot{n}_{aB}-\dot{n}_{aA}\geq\sfrac f{\bar n_A}n_{aB}(2n_{BB}-2n_{AA}+\eta n_{BB}-\OO(\Delta))>0$ by Proposition \ref{maxAB}. This finishes the proof.
\end{proof}

Now we show that the time $T^{aA=\delta AA}$ is finite and prove that it is smaller than or equal to $T^{aB=\delta BB}$.
To estimate the order of magnitude of the time $T_2$ we need bounds on $n_{aA}$ which depends on $\S_{aA,aB}$.

\begin{lemma}
	\label{aA-S2}
	For all $t\in[T_1,T_2]$ the $aA$ population is bounded by 
		\begin{equation}
\frac{f(n_{AB}+2n_{AA})}{4\bar n_B(f+\Delta)}\S_{aA,aB}\leq n_{aA}\leq\frac{f(n_{AB}+2n_{AA})}{\bar n_A(D-2\Delta)}\S_{aA,aB}.
\end{equation}
	\end{lemma}

\begin{proof}
\item[(1)]
We start with the upper bound. First observe that it holds at $t=T_1$. By Lemma \ref{lem-level} it is enough to show that if $n_{aA}$ would reach the upper bound it would decrease faster than the bound.
Using Lemma \ref{aa-S2} and \eqref{delta_bound} the birth rate of a majorising process on $aA$ is given by
\begin{align}
b_{aA}&\leq\sfrac f{2\Sigma_5}\Sigma_{aA,aB}(n_{AB}+2n_{AA}+\OO(\d)).
\intertext{For the death rate we use that $\eta<c, \sfrac{n_{AA}}{\bar n_B}<1$ and Proposition \ref{prop-S5}:}
d_{aA}&\geq n_{aA}\left(D+c\bar n_B-\sfrac\Delta{\bar n_B}n_{AA}-\eta n_{BB}-\OO(\D^2n_{AA})\right)\geq n_{aA}(D-2\D),\\
\dot{n}_{aA}&\leq\frac {f(2n_{AA}+n_{AB}+\OO(\d))}{2\S_5}\Sigma_{aA,aB}-n_{aA}(D-2\Delta).
\end{align}
We calculate the slope of the majorising process at the upper bound using Proposition \ref{prop-S5}:
\begin{align}
\dot{n}_{aA}&\leq f(2n_{AA}+n_{AB})\Sigma_{aA,aB}\left(\sfrac1{2\S_5}-\sfrac1{\bar n_A}+\OO(\d)\right)\leq-\sfrac f{2\bar n_A}(2n_{AA}+n_{AB}+\OO(\d))\Sigma_{aA,aB}.
\end{align}
We have to show that at the upper bound,
\begin{align}
\label{dot aA}
\dot{n}_{aA}\leq\frac{f(\dot{n}_{AB}+2\dot{n}_{AA})}{\bar n_A(D-2\D)}\S_{aA,aB}+\frac{f(n_{AB}+2n_{AA})}{\bar n_A(D-2\D)}\dot{\S}_{aA,aB}.
\end{align}
To do this we calculate minorising processes on $n_{AB}$ and $n_{AA}$. For the birth rates we use Lemma \ref{aa-S2} and \eqref{delta_bound}:
\begin{align}
b_{AB}&\geq\sfrac f{\S_5}n_{AB}\left(\sfrac12n_{aA}+n_{AA}+\sfrac12n_{aB}+\sfrac12n_{AB}+n_{BB}-\OO(\d^2)\right)+\sfrac{2f}{\S_5}n_{AA}\left(\sfrac12n_{aB}+n_{BB}-\OO\left(\d^2\right)\right),\\
b_{AA}&\geq\sfrac f{\S_5}n_{AA}\left(n_{aA}+n_{AA}+n_{AB}-\OO\left(\d^2\right)\right)+\sfrac{f}{2\S_5}n_{AB}\left(n_{aA}+\sfrac12n_{AB}-\OO\left(\d^2\right)\right).
\intertext{Applying Proposition \ref{prop-S5}, Lemma \ref{aa-S2} and \eqref{delta_bound} we get the following death rates:}
d_{AB}&\leq n_{AB}f,\\
d_{AA}&\leq n_{AA}(f+\D+\OO(\d^2)).
\intertext{Hence, the minorising processes are given by:}
\dot{n}_{AB}&\geq-\sfrac f{\S_5}n_{AB}\left(\sfrac12n_{aA}+\sfrac12n_{aB}+\sfrac12n_{AB}+\OO\left(\d^2\right)\right)+\sfrac{2f}{\S_5}n_{AA}\left(\sfrac12n_{aB}+n_{BB}-\OO\left(\d^2\right)\right),\label{minorAB}\\
\dot{n}_{AA}&\geq-\sfrac f{\S_5}n_{AA}\left(n_{aB}+n_{BB}+\OO\left(\d^2\right)\right)-\D n_{AA}+\sfrac f{2\S_5}n_{AB}\left(n_{aA}+\sfrac12n_{AB}-\OO\left(\d^2\right)\right).\label{minorAA}
\end{align}
From \eqref{minorAB} and \eqref{minorAA} we get that
\begin{align}
\dot{n}_{AB}+2\dot{n}_{AA}&\geq \sfrac f{\S_5}n_{AB}\left(\sfrac12n_{aA}-\sfrac12n_{aB}-\OO\left(\d^2\right)\right)+\sfrac{2f}{\S_5}n_{AA}\left(-\sfrac12n_{aB}-\OO\left(\d^2\right)\right)-2\D n_{AA}\nonumber\\
&\geq-2\D n_{AA}-\OO\left(\d\right)(n_{AA}+n_{AB})\label{beide}.
\end{align}
By Lemma \ref{lem-expS}, we know that $\dot{\S}_{aA,aB}\geq-\D\S_{aA,aB}$.
Thus, using \eqref{beide} the right-hand side minus the left-hand side of \eqref{dot aA} is bounded from below by
\begin{align}
\label{show slope}\nonumber
&-\frac{2\D n_{AA}f\S_{aA,aB}}{\bar n_A(D-2\D)}-\frac{\D(n_{AB}+2n_{AA})\S_{aA,aB}}{\bar n_A(D-2\D)}+\frac {f\left(2n_{AA}+n_{AB}+\OO\left(\d\right)\right)\S_{aA,aB}}{2\bar n_A}\nonumber\\
&\hspace{10,8cm}-\OO\left(\d\right)\S_{aA,aB}(n_{AA}+n_{AB})\nonumber\\
&\geq\frac{fn_{AA}\S_{aA,aB}}{\bar n_A}\left(1-\frac{2\D+\OO(\d)}{D-2\D}\right)+\frac {fn_{AB}\S_{aA,aB}}{2\bar n_A}\left(1-\frac{2\D+\OO(\d)}{D-2\D}\right)>0.
\end{align}
This finishes the proof of (1).
\item[(2)]
For the lower bound we proceed similarly (using Lemma \ref{lem-level}). This time we show that if $n_{aA}$ would reach the lower bound it would start to increase faster than the bound. Thus we need a minorising process on $n_{aA}$. Using Lemma \ref{aa-S2} and \eqref{delta_bound} the birth rate of such a process is given by
\begin{align}
b_{aA}&\geq\sfrac f{2\Sigma_5}\Sigma_{aA,aB}\left(2n_{AA}+n_{AB}-\OO\left(\d^2\right)\right).
\intertext{With Proposition \ref{prop-S5} we get:}
d_{aA}&\leq n_{aA}\left(f+\Delta+\OO\left(\d^2\right)\right),\\
\dot{n}_{aA}&\geq\frac {f\left(2n_{AA}+n_{AB}-\OO\left(\d^2\right)\right)}{2\bar n_B}\Sigma_{aA,aB}-n_{aA}(f+\Delta).
\end{align}
We calculate the slope of the minorising process at the lower bound:
\begin{align}\nonumber
\dot{n}_{aA}&\geq\frac {f\left(2n_{AA}+n_{AB}-\OO\left(\d^2\right)\right)}{2\bar n_B}\Sigma_{aA,aB}-\frac {f(2n_{AA}+n_{AB})}{4\bar n_B}\Sigma_{aA,aB}\\
&=\frac {f\left(2n_{AA}+n_{AB}-\OO\left(\d^2\right)\right)}{4\bar n_B}\Sigma_{aA,aB}>0.\label{II}
\end{align}
Thus the minorising process on $n_{aA}$ would increase when the $aA$ population would reach the lower bound. To ensure this lower bound we have to show
\begin{align}
\label{schranke}
\dot{n}_{aA}\geq\frac{f(\dot{n}_{AB}+2\dot{n}_{AA})}{4\bar n_B(f+\D)}\S_{aA,aB}+\frac{f(n_{AB}+2n_{AA})}{4\bar n_B(f+\D)}\dot{\S}_{aA,aB}.
\end{align}
For this we consider a majorising process on $\Sigma_{aA,aB}$. The birth rate is the same as in \eqref{braa} and the death rate can be lower bounded by using Proposition \ref{prop-S5} and Lemma \ref{aa-S2}:
\begin{align}
d_{\S_{aA,aB}}&\geq\S_{aA,aB}\left(f-\sfrac\D{\bar n_B}n_{AA}-\sfrac{\D^2}{\bar n_B}n_{AA}\right)+n_{aA}(\D-\eta n_{BB})+\OO\left(\S_{aA,aB}^3\right).
\intertext{Hence, we get:}
\dot{\Sigma}_{aA,aB}&\leq\sfrac \Delta{\bar n_B}n_{AA}\Sigma_{aA,aB}-n_{aA}(\Delta-\eta n_{BB})+\OO(\D^2n_{AA})\S_{aA,aB}.
\end{align}
Using that $\eta<c$, the slope of this process if $n_{aA}$ reaches the lower bound is estimated by
\begin{align}\nonumber
\dot{\Sigma}_{aA,aB}&\leq\sfrac\Delta{\bar n_B}n_{AA}\Sigma_{aA,aB}-\frac{f(2n_{AA}+n_{AB})}{4\bar n_B(f+\Delta)}(\Delta-\eta n_{BB})\Sigma_{aA,aB}+\OO(\D^2n_{AA})\S_{aA,aB}\\
&\leq\frac{f(2n_{AA}+n_{AB})}{4\bar n_B}\frac{f-D}{f+\Delta}\Sigma_{aA,aB}+\sfrac\Delta{\bar n_B}n_{AA}\Sigma_{aA,aB}+\OO(\D^2n_{AA})\S_{aA,aB}.\label{I}
\end{align}
Moreover we need majorising processes on $AA$ and $AB$. Using \eqref{delta_bound} the birth rates of these processes can be bounded by:
\begin{align}
b_{AB}&\leq\sfrac f{\S_5}n_{AB}\left(\sfrac12n_{AB}+n_{AA}+n_{BB}\right)+\sfrac{2f}{\S_5}n_{AA}n_{BB}+\OO(\d),\\
b_{AA}&\leq\sfrac f{\S_5}n_{AA}\left(n_{AB}+n_{AA}\right)+\sfrac{f}{4\S_5}n_{AB}^2+\OO(\d).
\intertext{For bounding the death rates we apply Proposition \ref{prop-S5}:}
d_{AB}&\geq n_{AB}\left(f-\sfrac{\D(1+\D)}{\bar n_B}n_{AA}\right),\\
d_{AA}&\geq n_{AA}\left(f+\D-\sfrac{\D(1+\D)}{\bar n_B}n_{AA}\right).
\intertext{Hence we get:}
\dot{n}_{AB}&\leq-\sfrac f{2\S_5}n_{AB}^2+\sfrac{2f}{\S_5}n_{AA}n_{BB}+\sfrac{\D(1+\D)}{\bar n_B}n_{AA}n_{AB}+\OO(\d),\label{majorAB}\\
\dot{n}_{AA}&\leq-n_{AA}\left(\sfrac f{\S_5}n_{BB}+\D-\sfrac{\D(1+\D)}{\bar n_B}n_{AA}\right)+\sfrac f{4\S_5}n_{AB}^2+\OO(\d).\label{majorAA}
\end{align}
From \eqref{majorAB} and \eqref{majorAA} we get that
\begin{align}
\dot{n}_{AB}+2\dot{n}_{AA}\leq-\D n_{AA}\left(2-\sfrac{2n_{AA}+n_{AB}}{\bar n_B}\right)+\OO\left(\D^2n_{AA}\right)+\OO(\d)<\OO\left(\D^2n_{AA}+\d\right).
\end{align}
Thus for \eqref{schranke} it is enough to show that
\begin{align}
\dot{n}_{aA}&\geq\frac{f(2n_{AA}+n_{AB})}{4\bar n_B(f+\Delta)}\dot{\Sigma}_{aA,aB}+\left(\OO\left(\D^2n_{AA}\right)+\OO(\d)\right)\S_{aA,aB},\label{III}
\end{align}
using \eqref{I} and \eqref{II} the lhs minus the rhs of \eqref{III} is given by
\begin{align}
&\frac {f\left(2n_{AA}+n_{AB}-\OO\left(\d^2\right)\right)}{4\bar n_B}\Sigma_{aA,aB}-\frac{f^2(n_{AB}+2n_{AA})^2}{16\bar n_B^2(f+\D)}\frac{f-D}{f+\Delta}\Sigma_{aA,aB}\nonumber\\\nonumber
&\hspace{4.7cm}-\frac{f(n_{AB}+2n_{AA})}{4\bar n_B(f+\D)}\frac{\Delta}{\bar n_B}n_{AA}\Sigma_{aA,aB}-\OO\left(\D^2n_{AA}+\d\right)\S_{aA,aB}\\\nonumber
&\geq\frac {f\left(2n_{AA}+n_{AB}-\OO\left(\d^2\right)\right)}{4\bar n_B}\Sigma_{aA,aB}-\frac {f(2n_{AA}+n_{AB})}{8\bar n_B}\Sigma_{aA,aB}\left(1+\frac{2\D(1+\D)n_{AA}}{\bar n_B(f+\D)}\right)-\nonumber\\
&\hspace{4.7cm}-\OO\left(\D^2n_{AA}+\d\right)\S_{aA,aB}>0.
\end{align}
For the first inequality we use the rough estimations $\sfrac{f-D}{f+\D},\sfrac f{f+\D}<1$ and $\sfrac{2n_{AA}+n_{AB}}{\bar n_B}<2$.
This concludes the proof.
\end{proof}

\begin{proposition}\label{S2 bound}
For all $t\in[T_1,T_2]$ the process $\S_{aA,aB}$ is bounded by
\begin{itemize}
\item[(1)]$\dot{\S}_{aA,aB}\leq n_{aA}\left(\eta n_{BB}-\D\frac{n_{AB}+\OO(\D n_{AA})}{n_{AB}+2n_{AA}}\right).$
\item[(2)]$\dot{\Sigma}_{aA,aB}\geq n_{aA}(\eta n_{BB}-\D-\OO(\d)).$
\end{itemize}
\end{proposition}

\begin{proof}
	\item[(1)]
	 We construct a majorising process on $\Sigma_{aA,aB}$. For the birth rate we use Lemma \ref{aa-S2}:
		\begin{align}\nonumber
		b_{\Sigma_{aA,aB}}&\leq n_{aA}\frac{f(\sfrac12n_{aA}+n_{AA}+n_{aB}+n_{AB}+n_{BB})}{\S_5}+n_{aB}\frac{f(n_{AA}+\sfrac12n_{aB}+n_{AB}+n_{BB})}{\S_5}+\OO(\S_{aA,aB}^2)\nonumber\\
		&\leq f\Sigma_{aA,aB}+\OO(\S_{aA,aB}^2).
		\intertext{By Proposition \ref{prop-S5} the death rate is given by:}
		\nonumber
		d_{\Sigma_{aA,aB}}&\geq \Sigma_{aA,aB}(D-\D+c\S_5)+\D n_{aA}-\eta n_{aA}n_{BB}\\
		&\geq \Sigma_{aA,aB}(f-\sfrac{\D(1+\D)}{\bar n_B}n_{AA})+\D n_{aA}-\eta n_{aA}n_{BB},\\
		\dot{\Sigma}_{aA,aB}&\leq \sfrac{\D(1+\D)}{\bar n_B}n_{AA}n_{aB}-n_{aA}(\D-
		\sfrac{\D(1+\D)}{\bar n_B}n_{AA}-\eta n_{BB})+\OO(\S_{aA,aB}^2).
		\end{align}
		To bound ${n}_{aB}$ ,we use Lemma \ref{ub-aB-aA}:
		\begin{align}
		\dot{\S}_{aA,aB}&\leq n_{aA}\left(\frac{\D(1+\D) n_{AA}}{n_{AB}+2n_{AA}}\frac{n_{AB}+2n_{BB}+\sfrac{2\D}c}{\bar n_B}-\D+\sfrac{\D(1+\D)}{\bar n_B}n_{AA}+\eta n_{BB}+\OO(\d)\right)\\\nonumber
		&\leq n_{aA}\left(\eta n_{BB}+\frac{\D(n_{AA}(n_{AB}+2n_{BB})+n_{AA}(n_{AB}+2n_{AA})-\bar n_B(n_{AB}+2n_{AA}))+\OO\left(\D^2n_{AA}\right)}{\bar n_B(n_{AB}+2n_{AA})}\right)
		\intertext{Using that $n_{AA}+n_{AB}+n_{BB}\leq\bar n_B$ (Proposition \ref{prop-S5}) we get:}
		&\leq n_{aA}\left(\eta n_{BB}-\D\frac{n_{AB}+\OO(\D n_{AA})}{n_{AB}+2n_{AA}}\right).
		\end{align}
	\item[(2)] 
		This time we construct a minorising process on $\Sigma_{aA,aB}$ by using Proposition \ref{prop-S5} and Lemma \ref{aa-S2}: 
	\begin{align}\nonumber
		b_{\Sigma_{aA,aB}}&\geq f\frac{n_{aA}\left(\sfrac12n_{aA}+n_{AA}+n_{aB}+n_{AB}+n_{BB}\right)+n_{aB}\left(n_{AA}+\sfrac12n_{aB}+n_{AB}+n_{BB}\right)}{n_{aA}+n_{AA}+n_{aB}+n_{AB}+n_{BB}}-\OO\left(\d^2\right)\\
		&\geq f\Sigma_{aA,aB}-\OO\left(\d^2\right),\\
		d_{\Sigma_{aA,aB}}&\leq\Sigma_{aA,aB}(D-\D+c\S_5)-\eta n_{aA}n_{BB}+\left(\D+\OO\left(\d^2\right)\right) n_{aA}\nonumber\\
		&\leq f\Sigma_{aA,aB}-n_{aA}\left(\eta n_{BB}-\D-\OO\left(\d^2\right)\right),\\
		\dot{\Sigma}_{aA,aB}&\geq n_{aA}(\eta n_{BB}-\D-\OO(\d)).
		\end{align}
\end{proof}
From this Proposition we can deduce
\begin{corollary}\label{LBS}
There exists a $t^*\in[T_1,T_2]$, such that for all $t\in[ t^*,T_2]$ and $\eta>\frac{4\D}{\bar n_B}=: \eta^\star$,  
	\begin{align}
	\dot{\Sigma}_{aA,aB}(t)>0.
	\end{align}
\end{corollary}

\begin{proof}
	A fine calculation will show that the competition $c-\eta$ felt by an $aA$ individual from a $BB$ individual allows the sum  $\Sigma_{aA,aB}$ to grow when $\eta$ is large enough, whereas it decreases when $\eta=0$. Note that we consider here the sum ${\Sigma}_{aA,aB}$ because the influence of $\eta$ cannot be seen in the rates of the $aB$ population alone. Heuristically, the growth of the $aB$ population happens due to the indirect influence (source of $a$ allele) of the less decaying $aA$ population.
%
		We prove that the minorising process of $\Sigma_{aA,aB}$ estimated in Proposition \ref{S2 bound}  starts to increase:
		\begin{align}
		\dot{\Sigma}_{aA,aB}&\geq n_{aA}(\eta n_{BB}-\D-\OO(\delta)).
		\end{align}
	As soon as $n_{BB}>\D/\eta$, the sum-process $\Sigma_{aA,aB}$ increases. 
	With the knowledge of the behaviour of the 3-system $(AA,AB,BB)$, (see \cite{CMM13}), Lemma \ref{AA=BB} and Proposition \ref{maxAB} we know that, for $t\geq T_=$, we have $n_{BB}\geq\sfrac{\bar n_B}4-\OO(\D)$. Hence, if we choose $\eta>\sfrac{4\D}{\bar n_B}$ the sum-process $\Sigma_{aA,aB}$ increases. 		
\end{proof}

Now we  calculate the time $T^{aA=\delta AA}\land T^{aB=\delta AB}$ and we will see that $T^{aA=\delta AA}\land T^{aB=\delta AB}=T^{aA=\delta AA}$.
\begin{theorem} \label{thm-T2}
	The time $T_2=\OO(\e^{-1/(1+\eta\bar n_{B}-\D)})$.
\end{theorem}
\begin{proof}

	From Proposition \ref{S2 bound} (2) we have a lower bound on $\dot{\S}_{aA,aB}$, and with Lemma \ref{aA-S2} we can further bound $n_{aA}$ (either from above or from below depending on the sign of the prefactor):
	\begin{align}\nonumber
	\dot{\S}_{aA,aB}&\geq (\eta n_{BB}-\D-\OO(\delta)) n_{aA}\\\nonumber
	&\geq (\eta n_{BB}-\D-\OO(\delta))\OO{(n_{AB}+2n_{AA})}\S_{aA,aB}\\
	&\geq\frac{\OO(\eta\bar n_{B}/4-\D)}{\OO(1)+\OO(1)t}\S_{aA,aB}.
	\end{align}
	where in the last line, the estimation on $n_{BB}$ comes from Proposition \ref{maxAB}, and the estimation on $n_{AB}$ comes from \cite{BovNeu} where it is proved that $n_{AB}$ starts to decrease like $1/t$ after a time of order $\OO(1)$.	As ${\Sigma}_{aA,aB}(T_1)=\OO(\e)$, the solution of the ODE that gives a lower bound is:
	\begin{align}\label{LBSigma}
	{\Sigma}_{aA,aB}(t)\geq{\OO(\e)}(\OO(1)+\OO(1)t)^{\OO(\eta\bar n_{B}/4-\D)}.
	\end{align}
	By using Proposition \ref{S2 bound} (1),  we get the same kind of solution as an upper bound on ${\Sigma}_{aA,aB}$ (note on the last step we can upper bound $n_{BB}$ by $\bar n_B$): 
	\begin{align}\label{UBSigma}
	{\Sigma}_{aA,aB}(t)\leq{\OO(\e)}(\OO(1)+\OO(1)t)^{\OO(\eta\bar n_{B}-\D)}
	\end{align}
	Using \eqref{LBSigma} and the lower bound in Lemma \ref{aA-S2} (together with the trivial estimation $n_{AA}\geq0$) we get a minorising process on $aA$:
	\begin{align}
	n_{aA}(t)\geq \OO(n_{AB}	{\Sigma}_{aA,aB})\geq{{\OO(\e)}(\OO(1)+\OO(1)t)^{\OO(\eta\bar n_{B}/4-\D)}}/{(\OO(1)+\OO(1)t)}.
	\end{align}
	The corresponding majorising process has an $\bar n_B$ instead of $\bar n_B/4$. By solving $n_{aA}=\delta n_{AA}=\OO(n_{AB}^2)$ we get the order of magnitude of $T_{aA=\delta AA}$:
	\begin{align}
	\OO\left(\e^{-1/(1+\eta\bar n_{B}-\D)}\right) \leq T_{aA=\delta AA}\leq \OO\left(\e^{-1/(1+\eta\bar n_{B}/4-\D)}\right).
	\end{align}
	Note that $1+\eta\bar n_B-\D>0$ for $\D$ small enough, and thus $T_{aA=\delta AA}$ diverges with $\e$ and the order calculations above are justified.
	
		It is left to ensure that $aB$ does not exceed $\delta n_{AB}$ in this time. It follows from Lemma \ref{aA-S2} that during the time interval $[T_1,T_2]$, we have ${\Sigma}_{aA,aB}=\OO(n_{aB})$. Thus, solving $n_{aB}=\delta n_{AB}$ amounts to solving $\OO({\Sigma}_{aA,aB})=\OO(1)/(\OO(1)+\OO(1)t)$ which gives the same order of magnitude as for $T_{aA=\delta AA}$. Thus the two times are of the same order. 
	
		Note that for $\eta=0$,  ${\Sigma}_{aA,aB}(T_2)=\OO\left(\e ^{1+{\D}/{(1-\D)}}\right)$. 
	{ This proves point \ref{eta0decay} of Theorem \ref{main-thm}.}
	
\end{proof}

\begin{proposition}\label{prop-T2}
	$T_2=T^{aA=\delta AA}$.
\end{proposition}
\begin{proof}
	This follows from Theorem \ref{thm-T2} and Lemma \ref{aa-S2}.
\end{proof}

\begin{proposition} 
 At time $t= T_2$ and if $f$ is taken sufficiently large (Assumption C2), $n_{aa}$ starts to grow out of itself: there exists some positive constant $c_{T_2}>0$ such that
		\begin{align}
		\dot n_{aa}\geq c_{T_2}\cdot n_{aa}.
		\end{align}
\end{proposition}

\begin{proof}
	We have $n_{AA}(T_2)=\OO\left(\e^{2/(1+\eta\bar n_B-\Delta)}\right)$. Thus, at the end of the second phase,
	\begin{align}
	b_{aa}&\geq fn_{aa}\frac{\sfrac12\d n_{AA}}{n_{AA}(1+\OO(\d))}=\frac{\delta f n_{aa}}{2(1+\OO(\delta))},\\
	d_{aa}&\leq n_{aa}(D+\D+n_{AA}(1+\OO(\d)))=n_{aa}(D+\Delta+\OO\left(\e^{2/(1+\eta\bar n_B-\Delta)}\right),\\
	\dot{n}_{aa}&\geq n_{aa}\left(\sfrac {\d f}2-D-\D-\OO\left(\e^{2/(1+\eta\bar n_B-\Delta)}\right)\right),
	\end{align} 
	the right-hand side is positive for $f$ large enough.
\end{proof}

\subsection{Phase 3:  Exponential growth of $aa$ until co-equilibrium with $BB$  }\label{subsec-phase3}\emph{}\\

\begin{figure}[b]
	\begin{center}
		\includegraphics[width=0.45\textwidth]{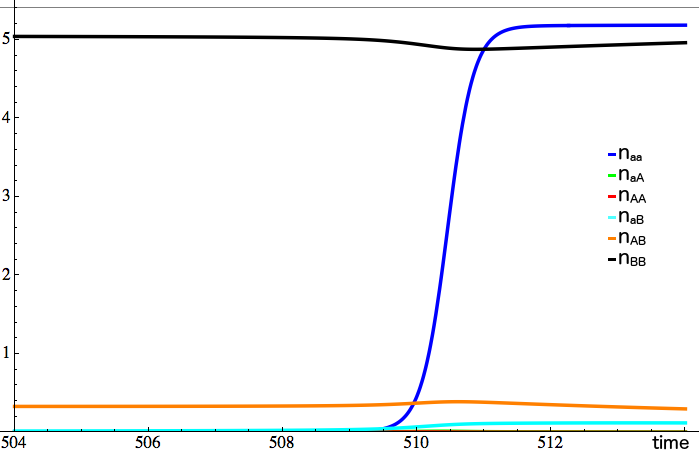}
		\includegraphics[width=0.45\textwidth]{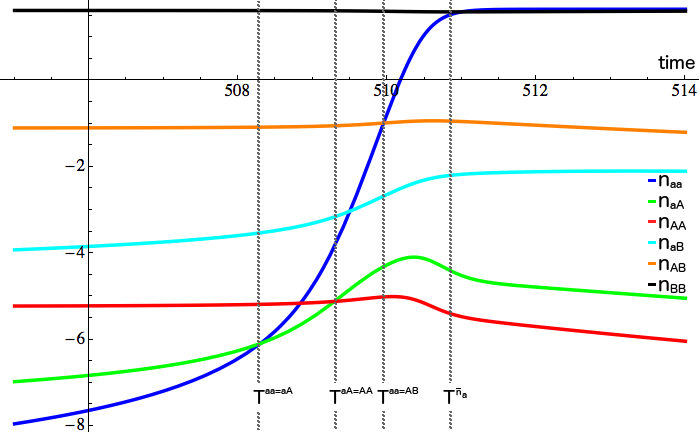}
		\caption{zoom-in when $aa$ recovers, general qualitative behaviour of $\{n_{i}\left(t\right),i\in\mathcal{G}\}$ (lhs) and log-plot (rhs).}
		\label{zoom-in}
	\end{center}
\end{figure}

Since $aa$ is growing now also out of itself it will influence the sum-process $\S_5=n_{aA}+n_{AA}+n_{aB}+n_{AB}+n_{BB}$ and we need new lower bounds on $\S_5$ in the following steps, the proof of this works similar to the one of Proposition \ref{prop-S5} by taking into account all contributing populations. Let us compute the ODE to which $\Sigma_5$ is the solution:

\begin{proposition}
	\label{S5-3Phase} The sum-process $\S_5$ is the solution to
	\begin{align}
	\dot{\S}_5=&\S_5\left(f-D-\D-c\S_5\right)-\D\left(n_{aA}+n_{AA}\right)-cn_{aa}\left(n_{aA}+n_{AA}\right)+2\eta n_{aA}n_{BB}\nonumber\\
	&+\sfrac{f}{\S_3}n_{aa}\left(\sfrac12n_{aA}+n_{AA}\right)-\sfrac{f}{4\S_5}n_{aB}(n_{aA}+n_{aB})-\sfrac{f}{4\S_6}n_{aA}\left(2n_{aa}+n_{aA}+n_{aB}\right).
	\end{align}
\end{proposition}
\begin{proof}
	We calculate the birth- and the death-rate of $\S_5$ under consideration of the $aa$ population:
	\begin{align}
	b_{\S_5}=&\sfrac f{\S_3}n_{aa}\left(\sfrac12n_{aA}+n_{AA}\right)+\sfrac f{\S_5}\left(\left(n_{aB}+n_{AB}+n_{BB}\right)\S_5-\sfrac14n_{aB}\left(n_{aA}+n_{aB}\right)\right)\nonumber\\\nonumber&
	+\frac f{\S_6}\left(\left(n_{aA}+n_{AA}\right)\S_6-n_{aA}\left(\sfrac12n_{aa}+\sfrac14n_{aA}+\sfrac14n_{aB}\right)\right)\\
	=&f\S_5+\sfrac f{\S_3}n_{aa}\left(\sfrac12n_{aA}+n_{AA}\right)-\sfrac f{4\S_5}n_{aB}\left(n_{aA}+n_{aB}\right)-\sfrac f{4\S_6}n_{aA}\left(2n_{aa}+n_{aA}+n_{aB}\right),\\
	d_{\S_5}=&\S_5\left(D-\D+c\S_5\right)+(cn_{aa}+\D)\left(n_{aA}+n_{AA}\right)-2\eta n_{aA}n_{BB},
	\end{align}
	which gives the result.
\end{proof}

We introduce some notation for the order of magnitude of $n_{AA}(T_2)$. We write $n_{AA}(T_2)=\OO(\e^\gamma)$ with
\begin{equation}
\gamma:=2/(1+\eta\bar n_B-\Delta).
\end{equation}
Thus the initial condition of the third Phase can be written as:
\begin{align}
n_{aa}(T_2)&\leq\OO(\e^2),\\
n_{aA}(T_2)&=\OO(\d n_{AA})=\OO(\d\e^\g),\\
n_{AA(T_2)}&=\OO(\e^\g),\\
n_{aB}(T_2)&=\OO(\d n_{AB})=\OO(\d\e^{\g/2}),\\
n_{AB}(T_2)&=\OO(\e^{\g/2}),\\
n_{BB}(T_2)&=\bar n_B-\OO(\e^{\g/2}),
\end{align}

Let 
\begin{equation}
T_3:=T^{aa}_{\bar n_a-\e_0}=\inf\left\{t>T_2:n_{aa}(t)=\bar n_a-\e_0\right\}.
\end{equation}

{We define the stopping time 
\begin{align}
T_+:=\min\left\{T_{\e ^{\gamma/10}}^{AB},T^{aA=AB},T^{AA=AB}\right\},
\end{align}
where:
\begin{align}
T_{\e ^{\gamma/10}}^{AB}&:=\inf\left\{t>T_2: n_{AB}(t)\geq\e ^{\gamma/10}\right\},\\
T^{aA=AB}&:=\inf\left\{t>T_2: n_{aA}(t)=n_{AB}(t)\right\},\\
T^{AA=AB}&:=\inf\left\{t>T_2: n_{AA}(t)=n_{AB}(t)\right\}.
\end{align}
Note that for the following arguments, ${\gamma/10}$ can be replaced by any positive power smaller than $\gamma/2$.}

We have to ensure that the $aa$ population grows until a neighbourhood of its equilibrium. 
From Theorem \ref{thm-T2} and Proposition \ref{prop-T2}, at the end of the second Phase we have 
 that $n_{aA},n_{AA},n_{aB}\leq n_{AB}<\D$. We start to bound $\S_5$:

\begin{lemma}
\label{S5-3P}
{For all $t\in[T_2,T_+]$,}
the sum-process $\S_5$ is bounded from above and below by
\begin{align}
\bar n_B-\sfrac{f+3\D}{c\bar n_B}n_{AB}\leq\S_5\leq\bar n_B+\sfrac{4(f+\D)}{c\bar n_B}n_{AB}.
\end{align}
\end{lemma}
\begin{proof}
We use Lemma \ref{lem-level} and construct minorising and majorising processes on $\S_5$ and $n_{AB}$. First observe that the bounds are satisfied at $t=T_2$ by Proposition \ref{prop-S5}. From Proposition \ref{S5-3Phase} we can deduce that
\begin{align}
 f\S_5-fn_{AB}~&\leq b_{\S_5}\leq f\S_5+2fn_{AB},\\
\S_5(D-\D+c\S_5)-(2f+\D)n_{AB}&\leq
d_{\S_5}\leq\S_5(D-\D+c\S_5)+2(f+\D)n_{AB},\\
\S_5(f-D+\D-c\S_5)-(f+2\D)n_{AB}&\leq
\dot\S_5\leq\S_5(f-D+\D-c\S_5)+(4f+\D)n_{AB}.
\end{align}
At the lower bound we get that the sum-process would increase:
\begin{align}
\dot\S_5\geq\D n_{AB}-\OO\left(n_{AB}^2\right)>0,
\end{align}
and at the upper bound it would decrease:
\begin{align}
\dot\S_5\leq-3\D n_{AB}-\OO\left(n_{AB}^2\right)<0.
\end{align}
It is left to show that
\begin{align}
\label{S5-dotbounds}
-\sfrac{(f+3\D)}{c\bar n_B}\dot n_{AB}\leq\dot\S_5\leq\sfrac{4(f+\D)}{c\bar n_B}\dot n_{AB}.
\end{align}
For this we construct a minorising process on $AB$ using that $n_{aA},n_{AA},n_{aB}\leq n_{AB}$, for the birth rate:
\begin{align}
b_{AB}&\geq\sfrac{f}{\S_5}n_{AB}\left(\sfrac14n_{aA}+\sfrac12n_{AA}+\sfrac12n_{aB}+\sfrac12n_{AB}+n_{BB}\right)\geq fn_{AB}-\sfrac {9f}{4\S_5}n_{AB}^2,\label{ABbirth}\\
d_{AB}&\leq n_{AB}(D-\D+c\S_5),\\
\dot n_{AB}&\geq n_{AB}\left(f-D+\D-c\S_5-\sfrac {9f}{4\S_5}n_{AB}\right).
\end{align}
At the lower bound we have $\dot n_{AB}\geq-\sfrac{5f+12\D}{4\bar n_B}n_{AB}^2+\OO\left(n_{AB}^3\right)$.\\ The lhs of \eqref{S5-dotbounds} $\leq\frac{(5f+12\D)(f+3\D)}{4c\bar n_B^2}n_{AB}^2+\OO\left(n_{AB}^3\right)<\D n_{AB}-\OO\left(n_{AB}^2\right)$. 
At the upper bound we get  $\dot n_{AB}\geq-\sfrac{25f+16\D}{4\bar n_B}n_{AB}^2-\OO\left(n_{AB}^3\right)$ and that the rhs of \eqref{S5-dotbounds} is larger or equal to $-\sfrac{4(f+\D)(21f+4\D)}{4c\bar n_B^2}n_{AB}^2-\OO\left(n_{AB}^3\right)>-3\D n_{AB}-\OO\left(n_{AB}^2\right)$.
\end{proof}


The following lemma ensures that the $aA$- and the $AA$ populations stay smaller than $\OO\left(n_{AB}^2\right)$ for all $t\in[T_2,T_+]$:
\begin{lemma}
\label{AB^2bounds}
\begin{enumerate}
\item For all $t\in[T_2,T_+]$, the $AA$ population is bounded from above by 
\begin{align}
n_{AA}\leq\sfrac 2{\bar n_B}n_{AB}^2,
\end{align}
\item For all $t\in[T_2,T_+]$ and $\eta<\sfrac c2\left(1-\sfrac{\OO(\D)}{f-D+\D}\right)$, the $aA$ population is bounded from above by 
\begin{align}
n_{aA}\leq\sfrac{10f}{\bar n_B(D-\D)}n_{AB}^2.
\end{align}
\end{enumerate}
In particular this implies $T_+=T^{AB}_{\e ^{\gamma/10}}$.
\end{lemma}
\begin{proof}
The proof uses again Lemma \ref{lem-level}. First observe that by Lemma \ref{bound-AB} (1) and since $n_{aA}(T_2)=\d n_{AA}(T_2)$ the bounds are satisfied at $t=T_2$.
\begin{enumerate}
\item We have to show that 
\begin{align}
	\label{AAdotboundup}
	\dot n_{AA}<\sfrac4{\bar n_B}n_{AB}\dot n_{AB}. 
\end{align}
We construct a majorising process on $n_{AA}$:
\begin{align}
b_{AA}&\leq\sfrac f{\S_5}n_{AA}(n_{aA}+n_{AA}+n_{AB})+\sfrac f{4\S_5}(n_{aA}+n_{AB})^2.
\intertext{For the death rate we use Lemma \ref{S5-3P}:}
d_{AA}&\geq n_{AA}(f-\OO(\D)),\\
\dot n_{AA}&\leq-n_{AA}\left(\sfrac f{\S_5}(n_{aB}+n_{BB})-\OO(\D)\right)+\sfrac f{\S_5}n_{AB}^2.
\end{align}
At the upper bound we thus have: $\dot n_{AA}\leq-\sfrac f{\bar n_B}n_{AB}^2+\OO\left(\D n_{AB}^2\right)$.\\
We now construct a minorising process on $n_{AB}$. As in \eqref{ABbirth} until $T_+$ we have a minoration of the birth rate:
\begin{align}
b_{AB}&\geq\sfrac{f}{\S_5}n_{AB}\left(\sfrac14n_{aA}+\sfrac12n_{AA}+\sfrac12n_{aB}+\sfrac12n_{AB}+n_{BB}\right)\geq fn_{AB}-\sfrac{9f}{4\S_5}n_{AB}^2.
\intertext{Using Lemma \ref{S5-3P} the death rate is given by:}
d_{AB}&\leq n_{AB}(D-\D+c\S_5)\leq fn_{AB}+\sfrac{4(f+\D)}{\bar n_B}n_{AB}^2,\\
\dot n_{AB}&\geq-\sfrac{25f+\OO(\D)}{4\bar n_B}n_{AB}^2. \label{minorABb}
\end{align}  
Thus the rhs of \eqref{AAdotboundup} is larger or equal to $-\sfrac{25f+\OO(\D)}{\bar n_B^2}n_{AB}^3$. Since $n_{AB}<\e ^{\gamma/10}$ until time $T_+$, we indeed have, for $\e$ small enough:
\begin{align}
	\dot n_{AA}\leq -\sfrac f{\bar n_B}n_{AB}^2+\OO\left(\D n_{AB}^2\right) < -\sfrac{25f+\OO(\D)}{\bar n_B^2}n_{AB}^3 \leq \sfrac4{\bar n_B}n_{AB}\dot n_{AB}. 
\end{align}

\item Similarly to (1) we construct a majorising process on $aA$. For the birth rate we use the result of (1):
\begin{align}
b_{aA}&\leq\sfrac f{2\S_3}n_{aa}n_{aA}+\sfrac {9f}{2\bar n_B}n_{AB}^2+\sfrac f{2\S_6}n_{aa}n_{aA}+\OO(\D n_{AB}^2).
\intertext{Applying Lemma \ref{S5-3P} we get for the death rate:}
d_{aA}&\geq n_{aA}\left(f+\D-\eta\bar n_B+cn_{aa}-\OO(n_{AB})\right),\\
\dot n_{aA}&\leq-n_{aA}\left(\sfrac {f}{2}-\eta\bar n_B+\sfrac{f-2(D-\D)}{2\S_5}n_{aa}\right)+\sfrac {9f}{2\bar n_B}n_{AB}^2+\OO(\D n_{AB}^2),\label{grow-aAb}\\
\dot n_{aA}&\leq-n_{aA}\sfrac {f-2\eta\bar n_B}2+\sfrac {9f+\OO(\D)}{2\bar n_B}n_{AB}^2.
\end{align}
At the upper bound we get: $\dot n_{aA}\leq n_{AB}^2\left(\sfrac {9f+\OO(\D)}{2\bar n_B}-\sfrac {10f(f-2\eta\bar n_B)}{2\bar n_B(D-\D)}\right)$.
We will show that 
\begin{align}
\label{aAdotboundupb}
\dot n_{aA}\leq\sfrac{20f}{\bar n_B(D-\D)}n_{AB}\dot n_{AB}. 
\end{align}
For this we use the minorising process on $AB$ constructed in \eqref{minorABb},  given by $\dot n_{AB}\geq-\sfrac{25f+\OO(\D)}{4\bar n_B}n_{AB}^2$. Thus the rhs of \eqref{aAdotboundupb} is larger or equal to $-\sfrac{105f+\OO(\D)}{\bar n_B(D-\D)}n_{AB}^3.$ Thus we have to ensure that $\left(\sfrac {9f+\OO(\D)}{2\bar n_B}-\sfrac {10f(f-2\eta\bar n_B)}{2\bar n_B(D-\D)}\right)<-\sfrac{105f+\OO(\D)}{\bar n_B(D-\D)}\D$. This yields the condition on $\eta$:
\begin{align}
\eta\leq\sfrac {f-D-\OO(\D)}{2\bar n_B}=\sfrac c2\left(1-\sfrac{\OO(\D)}{f-D+\D}\right).
\end{align}
\end{enumerate}
\end{proof}

\begin{remark}
Observe that the condition $\eta<\sfrac c2\left(1-\sfrac{\OO(\D)}{f-D+\D}\right)$ in Lemma \ref{AB^2bounds} prevents $aA$ to grow exponentially fast. Inequality \eqref{grow-aAb} shows that for larger values for $\eta$, $n_{aA}$ would start to grow out of itself and thus the system would converge towards the 6-point equilibrium as we checked numerically.  Hence, the assumption $\eta<\sfrac c2\left(1-\sfrac{\OO(\D)}{f-D+\D}\right)$ is essential in this phase and propagates to the following lemmata since we need therein the  $n_{AB}^2$-dependent bound on $aA$. 
\end{remark}

Using Lemma \ref{AB^2bounds} we can also compute a lower bound for $AA$:
\begin{lemma}
\label{AB^2lb}
For $\eta\leq\sfrac c2-\OO(\D)$ and $t\in[T_2,T_+]$, the $AA$ population is bounded from below by
\begin{align}
\sfrac1{8\bar n_B}n_{AB}^2\leq n_{AA}.
\end{align}
\end{lemma}
\begin{proof}
By Lemma \ref{lem-level} we construct a minorising process on $AA$ using Lemma \ref{S5-3P} for the death rate. The bound is satisfied at $t=T_2$ by Lemma \ref{bound-AB} (2).
\begin{align}
b_{AA}&\geq\sfrac f{4\S_5}n_{AB}^2,\\
d_{AA}&\leq n_{AA}(D+c\S_5+cn_{aa}),\\
\dot n_{AA}&\geq-n_{AA}\left(f+\D+\sfrac{4(f+\D)}{\bar n_B}n_{AB}+cn_{aa}\right)+\sfrac f{4\S_5}n_{AB}^2.
\end{align}
At the lower bound we have $\dot n_{AA}\geq\sfrac{D-\OO(\D)}{8\bar n_B}n_{AB}^2=\OO(n_{AB}^2)$.
It is left to show that
\begin{align}
\label{AAdotboundlowb}
\dot n_{AA}\geq \sfrac1{4\bar n_B}n_{AB}\dot n_{AB}.
\end{align}
We construct a majorising process on $AB$:
\begin{align}
b_{AB}&\leq\sfrac f{\S_5}(n_{aA}+2n_{AA}+n_{AB})\left(\sfrac 12n_{aB}+\sfrac12n_{AB}+n_{BB}\right)\nonumber\\
	&\leq\sfrac f{\S_5}n_{AB}\left(\sfrac12n_{aA}+n_{AA}+\sfrac12n_{aB}+\sfrac12n_{AB}+n_{BB}\right)+f(n_{aA}+2n_{AA})\nonumber\\
&\leq fn_{AB}+f(n_{aA}+2n_{AA})-\sfrac f{4\S_5}n_{AB}^2.\label{MajorAB}
\intertext{For the death rate we use Lemma \ref{S5-3P}:}
d_{AB}&\geq n_{AB}\left(f-\sfrac{(f+3\D)}{\bar n_B}n_{AB}\right),\\
\dot n_{AB}&\leq \sfrac{3f+12\D}{\bar n_B}n_{AB}^2+f(n_{aA}+2n_{AA}).
\end{align}
 Using Lemma \ref{AB^2bounds} we get that the rhs of \eqref{AAdotboundlowb} is smaller or equal to $\sfrac f{4\bar n_B}n_{AB}n_{aA}+\OO\left(n_{AB}^3\right)\leq\OO\left(n_{AB}^3\right)<\OO(n_{AB}^2)$ for $\e$ small enough since $n_{AB}\leq\e ^{\gamma/10}$ for $t\in[T_2,T_+]$.
\end{proof}

With all these lemmata we are now able to show that $n_{AB}$ stays below $\e ^{\gamma/10}$ until $T_3$ when $n_{aa}$ reaches the neighbourhood of its equilibrium.
\begin{lemma}
\label{ABstay}
For $t\in[T_2,T_+]$, it holds
\begin{align}
T_+=T_{\e ^{\gamma/10}}^{AB}\geq\OO\left(\e^{-\g/2}\right).
\end{align}
\end{lemma}
\begin{proof}
As in \eqref{MajorAB} we construct a majorising process on $AB$ {using Lemma \ref{AB^2bounds}:
\begin{align}
b_{AB}&\leq fn_{AB}+f(n_{aA}+2n_{AA})-\sfrac f{4\S_5}n_{AB}^2,\nonumber\\
&\leq fn_{AB}+\OO\left(n_{AB}^2\right).
\intertext{By Lemma \ref{S5-3P} the death rate can be bounded by:}
d_{AB}&\geq fn_{AB}-\OO\left(n_{AB}^2\right),\\
\dot n_{AB}&\leq\OO\left(n_{AB}^2\right).
\end{align}}
With the initial condition $n_{AB}(T_2)=\OO\left(\e^{\gamma/2}\right)$, we get the upper bound 
\begin{align}\label{UB-AB}
n_{AB}(t)\leq\frac{\OO(1)}{\OO\left(\e^{-\g/2}\right)-\OO(1)t}.
\end{align}
Thus $T_{\e ^{\gamma/10}}^{AB}\geq\OO\left(\e^{-\g/2}\right)$. Observe that it follows that $T_{\e ^{\gamma/5}}^{AA}=\OO(\e^{-\g/2})$.
\end{proof}

To ensure the exponential growth of $aa$ we need that the $aA$ population does not decay under the order $\OO(\e^{\gamma})$.

\begin{lemma}
\label{aA-AB^2}
For $\eta\leq\sfrac c2-\OO(\D)$ and for all $t\in[T_2,T_2+\OO(\e ^{-\gamma/2})]$, if $n_{aa}\leq n_{AB}$, then
\begin{align}\label{bound-aA}
n_{aA}\geq \OO\left(\delta n_{AB}^2\right).
\end{align}
\end{lemma}
\begin{proof}
We construct a minorising process on $aA$ in a very rough way assuming that there is no birth. Since $n_{aa}\leq n_{AB}$ the death rate can be bounded by using Lemma \ref{S5-3P}:
\begin{align}
d_{aA}&\leq n_{aA}(f+\Delta+cn_{aa}+\OO(n_{AB}))\nonumber\\
&\leq n_{aA}(f+\D+\OO(n_{AB})).
\end{align}
By Lemma \ref{AB^2bounds} we get:
\begin{align}
\dot n_{aA}\geq-\OO\left(n_{AB}^2\right).
\end{align}
At time $T_2$ we know that $n_{aA}=\delta n_{AA}=\OO\left(\delta n_{AB}^2\right)=\OO(\delta\e ^\gamma)$. Thus, using \eqref{UB-AB} to bound $n_{AB}$, we conclude that $n_{aA}$ satisfies until $T_+$:
\begin{align}\label{min-proc-aA}
n_{aA}(t)\geq{\frac{\OO(1)}{\OO(1)t-\OO(\e^{-\gamma/2})}+\OO(\delta \e ^{\gamma})+\OO(\e ^{\gamma/2}).}
\end{align}

We have to ensure that the slope of a minorising process on $AB$ is of the same order.
Thus we construct a minorising process on $AB$ using Lemma \ref{AB^2bounds} and that $n_{aa}\leq n_{AB}$:
\begin{align}
b_{AB}&\geq\sfrac f{\S_5}n_{AB}\left(\sfrac12n_{aA}\!+\!n_{AA}\!+\!\sfrac12n_{aB}\!+\!\sfrac12n_{AB}\!+\!n_{BB}\right)\!+\!\sfrac f{\S_6}(n_{aA}\!+\!2n_{AA})\left(\sfrac12n_{aB}\!+\!\sfrac12n_{AB}\!+\!n_{BB}\right)\!-\!\OO\left(n_{AB}^3\right)\nonumber\\
&\geq fn_{AB}-\sfrac f{\S_5}n_{AB}^2-\OO\left(n_{AB}^3\right).
\end{align}
where for the last inequality we used that $n_{aB}\leq n_{AB}$. Using Lemma \ref{S5-3P} we get:
\begin{align}
d_{AB}&\leq n_{AB}(f+\OO(n_{AB})),\\
\dot n_{AB}&\geq-\OO(n_{AB}^2).
\end{align}
Thus the minorising process on $AB$ is given by:
\begin{align}\label{min-proc-AB}
n_{AB}(t)\geq{\frac{\OO(1)}{\OO(1)t+\OO(\e^{-\gamma/2})}}
\end{align}
	The minorising process \eqref{min-proc-AB} stays of order $\OO(\e ^{\gamma/2})$ during a time of order $\OO(\e ^{-\gamma/2})$.\\
	The minorising process \eqref{min-proc-aA} needs time of order $\OO\left(\e^{-\gamma/2}\right)$, to reach the order $\OO\left(\delta \e^{\gamma+\alpha}\right)$ for any $\alpha>0$.
	The bound \eqref{bound-aA} is thus ensured for a time of order $\OO(\e ^{-\gamma/2})$.
\end{proof}

Now we show that $n_{aa}$ increases to a neighbourhood of its equilibrium before time $T_+$.

\begin{lemma}
	\label{lem-aa-exp-LB}
For $\eta\leq\sfrac c2-\OO(\D)$ and all $t\in[T_2,T_+]$ the $aa$ population increases to a $\e_0$-neighbourhood of its equilibrium $\bar n_a$ exponentially fast, and $T_3<T_+$.
\end{lemma}
\begin{proof}
We construct a minorising process on $aa$ and distinguish some cases. We remember that $n_{aa}(T_2)=\OO(\e ^2)$ if we do not have a better starting value.
First observe that by Lemma \ref{AB^2bounds} it holds that $d_{aa}\leq n_{aa}\left(D+\D+cn_{aa}+\OO\left(n_{AB}^2\right)\right)$.

\begin{enumerate}
\item If $n_{aa}\leq n_{aA}\leq n_{AA}$ or $n_{aA}\leq n_{aa}\leq n_{AA}$\\
In that case the birth-rate is given by $b_{aa}\geq fn_{aa}\sfrac{n_{aA}}{6n_{AA}}$. 
 With Lemma \ref{AB^2bounds} and \ref{aA-AB^2} we get $b_{aa}\geq \OO(\d)fn_{aa}$ and $\dot n_{aa}\geq n_{aa}\left(\OO(\d)f-D-\D-\OO\left(n_{AB}^2\right)\right)$. \\
Hence the time until $aa$ reaches $\bar n_a-\e _0$ satisfies $T_3\leq\OO\left(\ln(\e^{-2})\right)=\OO(\ln(1/\e ))$.

\item If $n_{aa},n_{AA}\leq n_{aA}$\\
In that case, by Lemma \ref{AB^2bounds} we get
$b_{aa}\geq \sfrac f6n_{aa}$ and $\dot n_{aa}\geq n_{aa}\left(\sfrac f6\!-\!D\!-\!\D\!-\!\OO\left(n_{AB}^2\right)\!\right)$. \\
Hence $T_3\leq \OO\left(\ln\left(\e^{-2}\right)\right)=\OO(\ln(1/\e ))$.

\item If $n_{aA}\leq n_{AA}\leq n_{aa}\leq n_{AB}$\\
In that case, by Lemma \ref{AB^2bounds}, $b_{aa}\geq \sfrac f3n_{aa}$ and $\dot n_{aa}\geq n_{aa}\!\left(\sfrac f3\!-\!D\!-\!\D\!-\!cn_{aa}\!-\!\OO\!\left(n_{AB}^2\right)\!\right)$. 
Hence  $T_3\leq \OO\left(\ln\left(\e^{-\g/5}\right)\right)=\OO(\ln(1/\e ))$.

\item If $n_{AA}\leq n_{aA}\leq n_{aa}\leq n_{AB}$\\
In that case, by Lemma \ref{AB^2bounds}, $b_{aa}\geq \sfrac f3n_{aa}$ and $\dot n_{aa}\!\geq n_{aa}\!\left(\sfrac f3\!-\!D\!-\!\D\!-\!cn_{aa}\!-\!\OO\left(n_{AB}^2\right)\!\right)$. 
Hence $T_3\leq \OO\left(\ln\left(\e^{-\g/5}\right)\right)=\OO(\ln(1/\e ))$.

\item If $n_{aa}> n_{AB}$\\
In that case, by Lemma \ref{AB^2bounds}, $b_{aa}\geq n_{aa}(f-\OO({n_{AB}^2}/{n_{aa}}))$ and $\dot n_{aa}\geq n_{aa}(f-D-\D-cn_{aa}-\OO({n_{AB}^2}/{n_{aa}}))$. 
Hence $T_3\leq \OO(\e^{\g/10}\ln(\e^{-\g/10}))=\OO(\e ^{\gamma/10}\ln(1/\e ))$.
\end{enumerate}
Thus, remembering Lemma \ref{ABstay}, we proved that $T_3\leq\OO(\ln(1/\e ))\leq \OO(\e ^{-\gamma/2})\leq T_+.$
\end{proof}

\subsection{Phase 4:   Convergence to $p_{aB}=(\bar n_a,0,0,0,0,\bar n_B)$ }\label{subsec-phase5}\emph{}

The Jacobian matrix of the field \eqref{dyn-syst} at the fixed point $p_{aB}$ has the 6 eigenvalues:   $0$ (double), and $ -(2f-D), -(f-D+\Delta), 
-(f-D-\Delta), -((f-D)(5f-4D)+f\Delta)/(4(f-D)+\eta\bar n_B)$ which are strictly negative under Assumptions (C). 
Because of the zero eigenvalues, $p_{aB}$ is a non-hyperbolic equilibrium point of the system and linearisation fails to determine its stability properties. Instead, we use the result of the center manifold theory \cite{H77, P01} that asserts that the qualitative behaviour of the dynamical system in a neighbourhood of the non-hyperbolic critical point $p_{aB}$ is determined by its behaviour on the center manifold near $p_{aB}$.

\begin{theorem}[The Local Center Manifold Theorem 2.12.1 in \cite{P01}]\label{thm-LCM} Let $f\in C^r(E)$, where $E$ is an open subset of $\R^n$ containing the origin and $r\geq1$. Suppose that $f(0)=0$ and $Df(0)$ has $c$ eigenvalues with zero real parts and $s$ eigenvalues with negative real parts, where $c+s=n$. Then the system $\dot z=f(z)$ can be written in diagonal form
	\begin{align}
		\dot x&=Cx+F(x,y),\\
		\dot y&=Py+G(x,y),
	\end{align}
	where $z=(x,y)\in\R^c\times\R^s$, $C$ is a $c\times c$-matrix with $c$ eigenvalues having zero real parts, $P$ is a $s\times s$-matrix with $s$ eigenvalues with negative real parts, and $F(0)=G(0)=0, DF(0)=DG(0)=0.$ Furthermore, there exists $\delta>0$ and a function, $h\in C^r(N_\d(0))$, where $N_\d(0)$ is the $\d$-neighbourhood of $0$, that defines the local center manifold and satisfies:
	\begin{align}
		\label{cme}
		Dh(x)[Cx+F(x,h(x))]-Ph(x)-G(x,h(x))=0,
	\end{align}
	for $|x|<\d$. The flow on the center manifold $W^c(0)$ is defined by the system of differential equations
	\begin{align}\label{flow}
		\dot x=Cx+F(x,h(x)),
	\end{align}
	for all $x\in\R^c$ with $|x|<\d$.
\end{theorem}

\begin{theorem}\label{phase4}
	 The non-hyperbolic critical point $p_{aB}$ is a stable fixed point and the flow on the center manifold near the critical point approaches $p_{aB}$ with speed $\frac1t$.
\end{theorem}
\begin{proof}
	We apply the Local Center Manifold Theorem \ref{thm-LCM}. All the calculations below were done with the program Mathematica 11.0.0.0 Student Edition. We do not write the results of all the intermediary calculations as they would take a few pages and bring no more information. Instead, we describe as precisely as possible the calculations we did so that they can be checked by the reader (using a similar computer program).

By the affine transformation $(n_{aa},n_{BB})\mapsto (n_{aa}-\bar n_a,n_{BB}-\bar n_B)$ we get a translated system $\tilde F(n)$
which has a critical point at the origin. We compute the two eigenvectors corresponding to 0 eigenvalues of the Jacobian matrix of $\tilde F$ at the fixed point $(0,0,0,0,0,0)$, which are
\begin{align}
	EV_1=(0,0,0,1,0,-1)\quad and \quad
	EV_2=(0,0,0,-1,1,0).
\end{align}
We perform a new change of variable to work in the basis of eigenvectors of $\tilde F(n)$. Let us call the new coordinates $x_1,\ldots, x_6$. 
Let $h(x_1,x_2)$ be the local center manifold (still unknown). We shall look at its local shape near $(0,0)$ and expand it up to second order. Let us write  
\begin{align}
h(x_1,x_2)=\begin{pmatrix}
{\lambda_3} x_1^2+{\nu_3} {x_1} {x_2}+{\mu_3} x_2^2\\
{\lambda_4} x_1^2+{\nu_4} {x_1} {x_2}+{\mu_4} x_2^2\\
{\lambda_5} x_1^2+{\nu_5} {x_1} {x_2}+{\mu_5} x_2^2\\
{\lambda_6} x_1^2+{\nu_6} {x_1} {x_2}+{\mu_6} x_2^2
\end{pmatrix}+O\left( x^3\right).
\end{align} 
We then substitute the series expansions into the center manifold equation \eqref{cme} which gives us 12 equations for the 12 unknowns $\lambda_3,\ldots,\mu_6$.
Substitution of the explicit second order approximation of the center manifold equation into \eqref{flow} yields 
the (order 2 approximation of the) flow on the local center manifold:
{\begin{align}
\dot x_1&=	\frac{A_1}{B_1}{x_1} {x_2} +\frac{ C_1}{D_1}{x_2}^2+\frac{E_1}{F_1}{x_1}^2+O\left(x^3\right),\\
\dot x_2&=	\frac{A_2}{B_2}{x_1} {x_2} +\frac{ C_2}{D_2}{x_2}^2+\frac{E_2}{F_2}{x_1}^2+O\left(x^3\right),
\end{align}}
where we obtain
{\footnotesize\begin{align}
A_1&=3 c^2 D f^2-c^2 {\Delta} f^2-3 c^2 f^3,\\
B_1&=(D-{\Delta}-f) \left(4 c D^2-9 c D f+c {\Delta} f+5 c f^2-4 D^2 {\eta}+4 D {\Delta} {\eta}+8 D {\eta} f-4 {\Delta} {\eta} f-4 {\eta} f^2\right),\\
C_1&=12 c^2 D^3 f^2-4 c^2 D^2 {\Delta} f^2-39 c^2 D^2 f^3+12 c^2 D {\Delta} f^3+42 c^2 D f^4-c^2 {\Delta}^2 f^3-8 c^2 {\Delta} f^4\nonumber\\
&\quad-15 c^2 f^5+12 c D^3 {\eta} f^2-16 c D^2 {\Delta} {\eta} f^2-36 c D^2 {\eta} f^3+4 c D {\Delta}^2 {\eta} f^2+32 c D {\Delta} {\eta} f^3\nonumber\\
&\quad+36 c D {\eta} f^4-4 c {\Delta}^2 {\eta} f^3-16 c {\Delta} {\eta} f^4-12 c {\eta} f^5,\\
D_1&=8 (D-2 f) (D-f) (D-{\Delta}-f) \times\nonumber\\
&\quad\times\left(4 c D^2-9 c D f+c {\Delta} f+5 c f^2-4 D^2 {\eta}+4 D {\Delta} {\eta}+8 D {\eta} f-4 {\Delta} {\eta} f-4 {\eta} f^2\right),\\
E_1&=c f,\quad
F_1=2 (-D+{\Delta}+f),
\end{align}}
and
{\footnotesize\begin{align}
A_2&=2 c^2 D^2 f-3 c^2 D f^2+c^2 f^3-2 c D^2 {\eta} f+2 c D {\Delta} {\eta} f+4 c D {\eta} f^2-2 c {\Delta} {\eta} f^2-2 c {\eta} f^3,\\
B_2&=(D-{\Delta}-f) \left(4 c D^2-9 c D f+c {\Delta} f+5 c f^2-4 D^2 {\eta}+4 D {\Delta} {\eta}+8 D {\eta} f-4 {\Delta} {\eta} f-4 {\eta} f^2\right),\\
C_2&=-3 c D {\eta} f^2+c {\Delta} {\eta} f^2+3 c {\eta} f^3,\\
D_2&=2 (D-2 f) \left(4 c D^2-9 c D f+c {\Delta} f+5 c f^2-4 D^2 {\eta}+4 D {\Delta} {\eta}+8 D {\eta} f-4 {\Delta} {\eta} f-4 {\eta} f^2\right),\\
E_2&=0,\quad F_2=1.
\end{align}}
It is left to show that the above system flows toward the origin, at least for $\eta$ smaller than a certain constant. To do that, we perform another change of variables which allows us to work in the positive quadrant. We call the new coordinates (on the center manifold) $y_1$ and $y_2$, and the new field $\hat F$. Observe that it is sufficient to prove that the scalar product of the field with the position is negative. We thus consider the function
\begin{equation}
s(y_1,y_2)=\left(\hat F(y_1,y_2), (y_1,y_2)\right),
\end{equation}
which is a quadratic form in $y_1$ and $y_2$.
As the field $\hat F$ is homogeneous of degree 2 in its variables, it is enough to consider any direction given by $y_2=\lambda y_1$, and prove that $s(y_1,\lambda y_1)<0$ for all $\lambda>0$. As the expressions are so ugly, we work perturbatively in $f$ and consider it as large as needed. Observe that the numerator and the denominator of $s(y_1,\lambda y_1)$ are polynomials of degree 5 in $f$. We thus compute the coefficient in front of $f^5$, and obtain by a series expansion:
\begin{align}
s(y_1,\lambda y_1)=\frac{c {y_1}^3 \left(c \left(16 \lambda ^3+7 \lambda ^2+16 \lambda +40\right)-4 {\eta} \left(5 \lambda ^3+8 \lambda ^2+8 \lambda +8\right)\right)}{64 {\eta}-80 c}f^5+\OO\left(f^4\right).
\end{align}
Observe that the denominator is always negative (because by the  assumption that $\eta\leq c$).
We finally compute the minimal value of the ratio 
\begin{equation}\label{r}
{r}(\lambda )\text{:=}\frac{16 \lambda ^3+7 \lambda ^2+16 \lambda +40}{4 \left(5 \lambda ^3+8 \lambda ^2+8 \lambda +8\right)},
\end{equation}
and obtain $r_{max}\simeq0.593644$. Thus, asymptotically as $f\to\infty$, the field is attractive for $\eta<c\cdot r_{max}$.
Thus we see that $p_{aB}$ is a stable fixed point which is approached with speed $\frac 1{t}$ as long as $\eta<c\cdot r_{max}$. \\
Figure \ref{pic-flow} shows the flow in the center manifold of the fixed point $p_{aB}$, for two values of $\eta$, one below the threshold $cr_{max}$ and one above. We see that the flow is attractive in the first case and repulsive in the second one.
\end{proof}

\begin{figure}[t]
	\begin{center}
		\includegraphics[width=.45\textwidth]{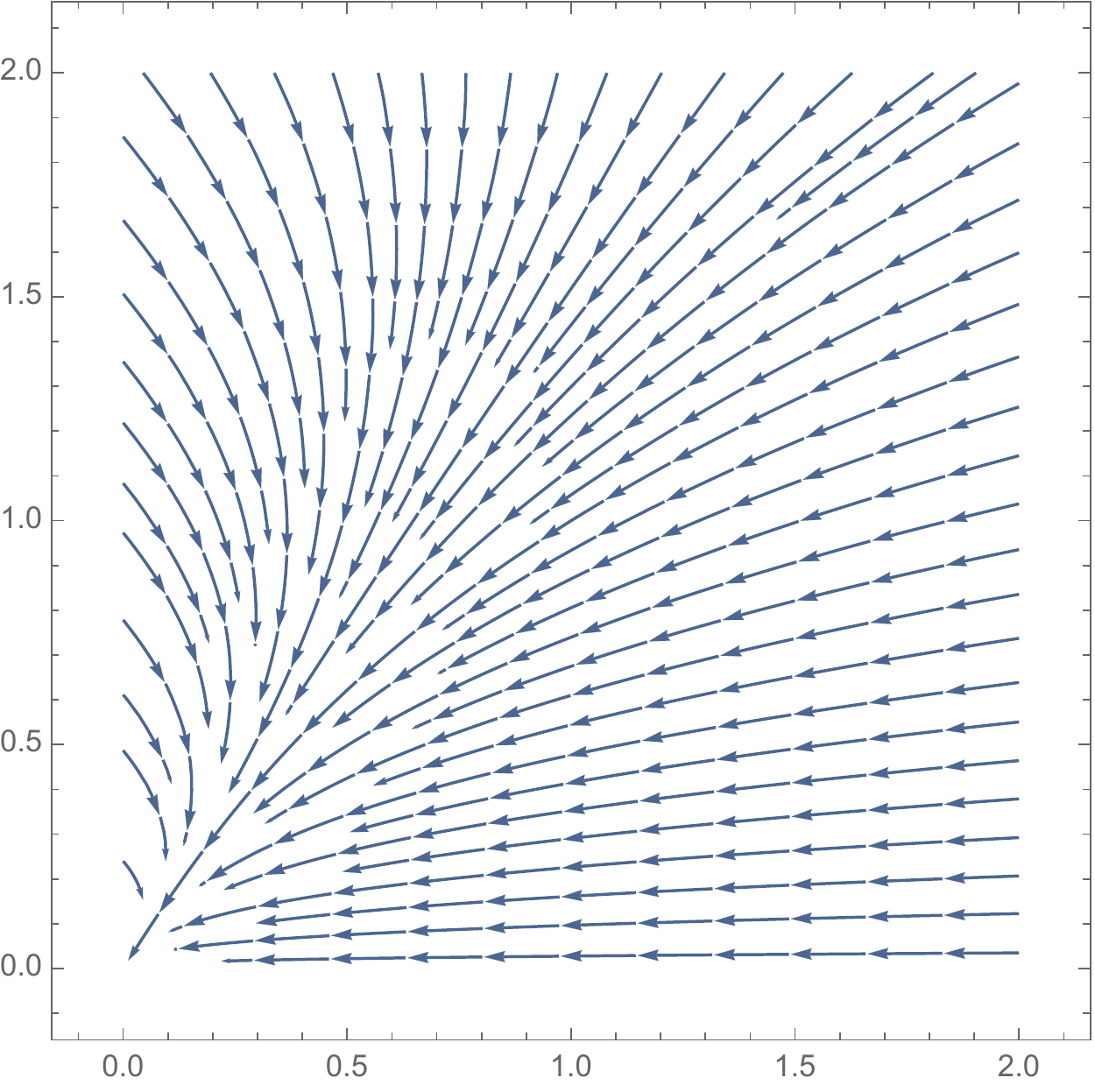}
		\includegraphics[width=.45\textwidth]{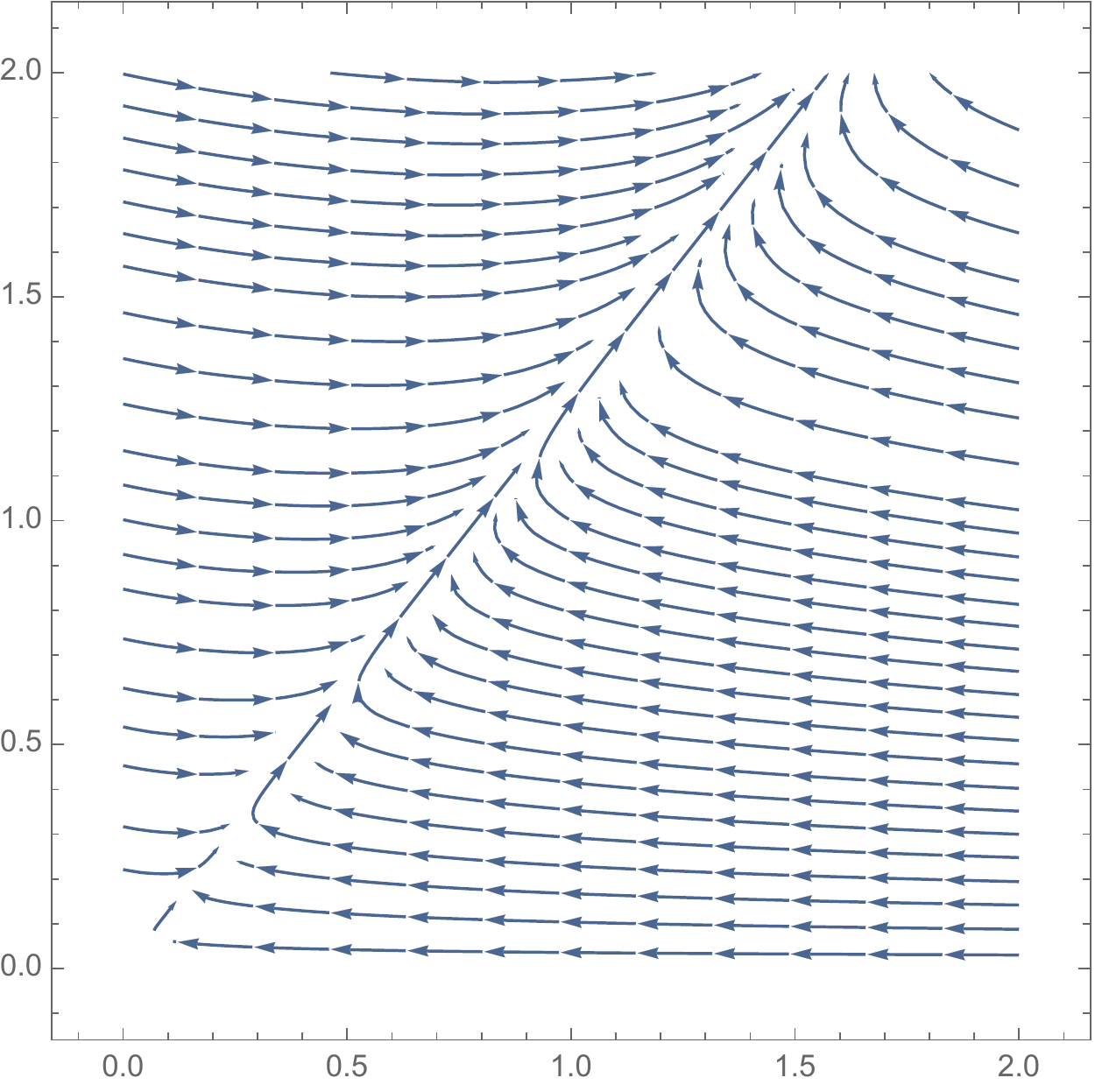}
		\caption{Flow of the dynamical system in the center manifold of the fixed point $p_{aB}$, for $\eta=0.02$ (left) and $\eta=0.6$ (right). }
		\label{pic-flow}
	\end{center}
\end{figure}

\section{Discussion}
\label{discussion}
In the rigorous results we have presented in this paper, we made some particular assumptions on the parameters of our model in order to simplify the analysis of the (already difficult) dynamical system.
In this section we discuss which of these assumptions can be relaxed, based on heuristic considerations and numerical simulations.



\medskip
\paragraph{The no-reproduction-small-competition model.}
In the model considered so far, we  assume  that the mutation to the $B$ allele produces a new species different to the one of phenotype $a$. 
This is done by  the \emph{no reproduction} assumption between individuals of phenotype $a$ and of phenotype $B$. 

These requirements are not needed to observe the recovery of the $aa$ population. In fact, what we require is that the invasion fitness of the $aa$ population into a resident $BB$ population is positive. 
Therefore, we can relax the no-competition assumption  and  add a small competition, $c_{aB}$, 
between $aa$ individuals and $BB$ individuals. 
This additional competition increases the time until $aa$ can reinvade and also modify the two-
population fixed point $p_{aB}$  
(see Figure \ref{no-nono-model}). 

\begin{figure}[h!]
	\begin{center}
		\includegraphics[width=0.47\textwidth]{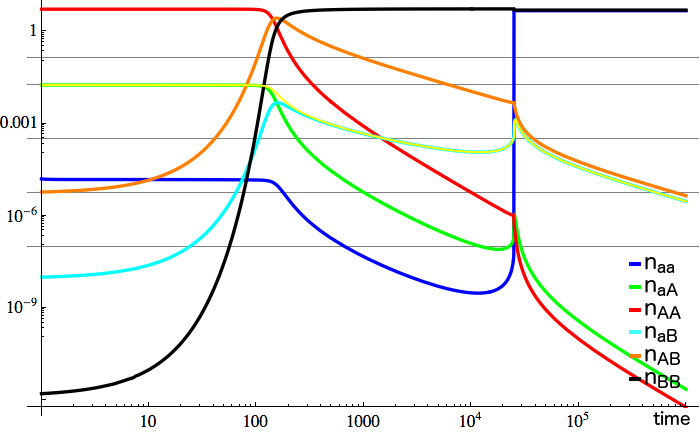}
		\caption{Numerical solution of the deterministic process, loglogplot for $\eta=0$
		 and $c_{aB}=0.1$.}
		\label{no-nono-model}
	\end{center}
\end{figure}

%

Adding the factor $\eta$  accelerates the process of recovery, and, consequently, allows to increase the competition rate $c_{aB}$ (see Figure \ref{eta-big}).

\begin{figure}[h!]
	\begin{center}
		\includegraphics[width=0.46\textwidth]
		{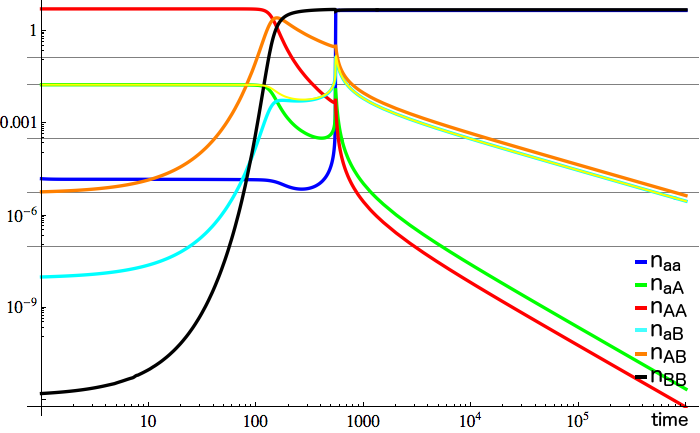}
		\includegraphics[width=0.46\textwidth]{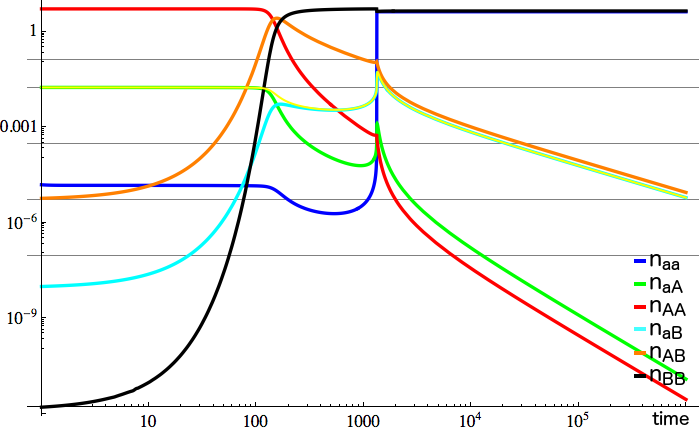}
		\caption{Numerical solution of the deterministic process, loglog-plot (left) for $\eta=0.01$ 
		and $c_{aB}=0.1$, (right) for $\eta=0.01$ and $c_{aB}=0.2$.}
		\label{eta-big}
	\end{center}
\end{figure}



For small $\eta$ we end up in a $aa$-$BB$ equilibrium, but by accelerating (increasing $\eta$ or decreasing $c_{aB}$) the process even more,  we can also end up in a 6-point equilibrium (all six population coexist) (see Figure \ref{eta6}). 

\begin{figure}[h!]
	\begin{center}
		\includegraphics[width=0.46\textwidth]
		{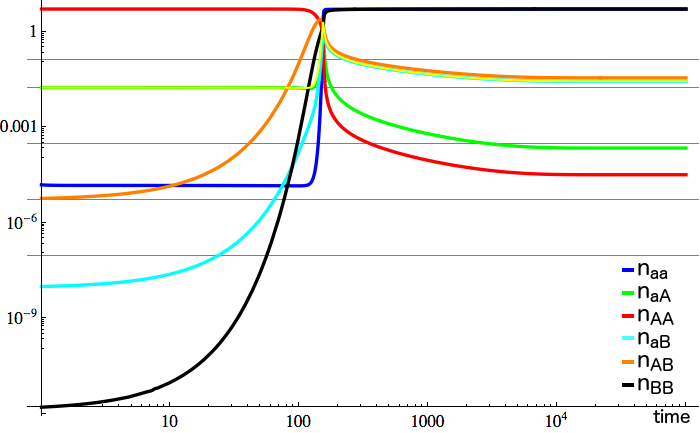}
		\includegraphics[width=0.46\textwidth]{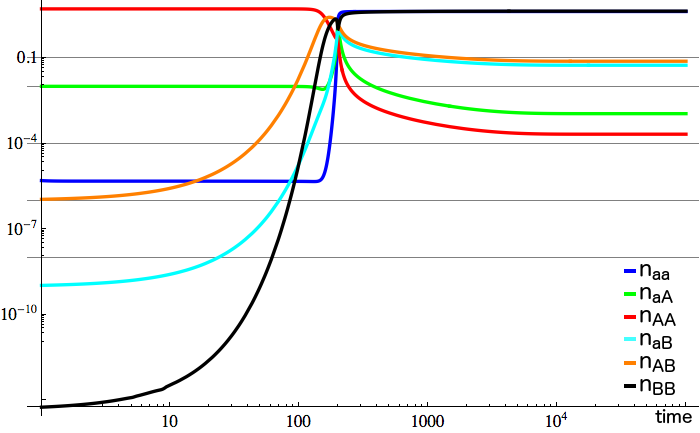}
		\caption{Numerical solution of the deterministic process, loglog-plot (left) for $\eta=0.56$ 
		and $c_{aB}=0$, (right) for $\eta=0.17$ and $c_{aB}=0.2$.}
		\label{eta6}
	\end{center}
\end{figure}

If there is competition between individuals of phenotype $a$ and of phenotype $B$ ($c_{aB}>0$) the $aa$ and $BB$ populations have smaller equilibria as the no-competition equilibria $\bar n_{a}$ and $\bar n_B$, obtained when $c_{aB} =0$. Thus the competition felt from $aA$ by $aa$ and $BB$ is lower and a smaller $\eta$ is enough to observe the 6-point equilibrium.


\medskip
\paragraph{The all-with-all model.}
The assumption of no reproduction between individuals of phenotype $a$ and of phenotype $B$ is 
also not really necessary for the recovery of the $aa$ population. 
Let us discuss the \emph{all-with-all model} where all phenotypes can reproduce among themselves, 
that is, where the reproductive compatibility is $R_i(j)=1$, for all $i,j\in\mathcal G$.

If we analyse the invasion fitness of the $aa$ population in the macroscopic $BB$ population, it is positive (and we can observe the recovery of $aa$)  if the fecundity $f$ scales with $\e$ in such a way that $f\cdot\e^2>D+\Delta$. Indeed, in this model the whole population acts as potential partner for each individual, and the birth rate of $aa$ scales with $\S_6$ instead of $\S_3$.

If this requirement on $f$ is fulfilled, then numerical simulations show that most of the results carry over to this 
model, but with the main difference that the 2-points equilibrium is replaced by a 3-points equilibrium.
The reason for this is that the reproduction between $a$ and $B$ individuals will always give 
birth to $aB$ individuals and thus the $aB$ population also survives (see Figure \ref{eta} (up)). 

We can also add a small competition between individuals of phenotype $a$ and  individuals of phenotype $B$ and still get the 3-point equilibrium (see Figure \ref{eta}(up)).

As in the no-reproduction model, adding the factor $\eta$ results in accelerating the process (see Figure \ref{eta} (middle-left)).

\begin{figure}[h!]
	\begin{center}
		\includegraphics[width=0.46\textwidth]
		{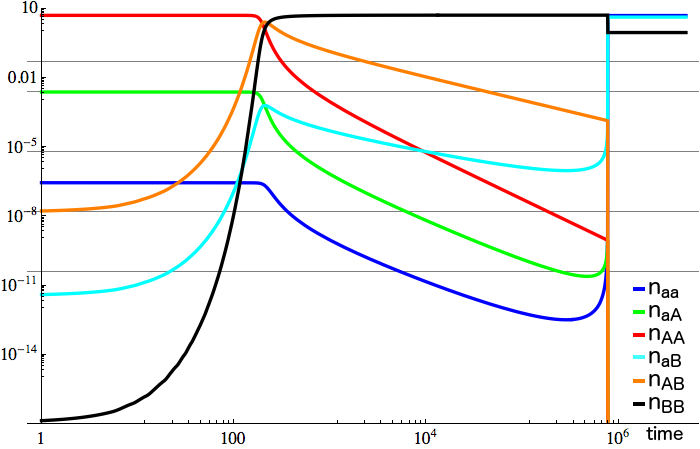}
		\includegraphics[width=0.46\textwidth]
		{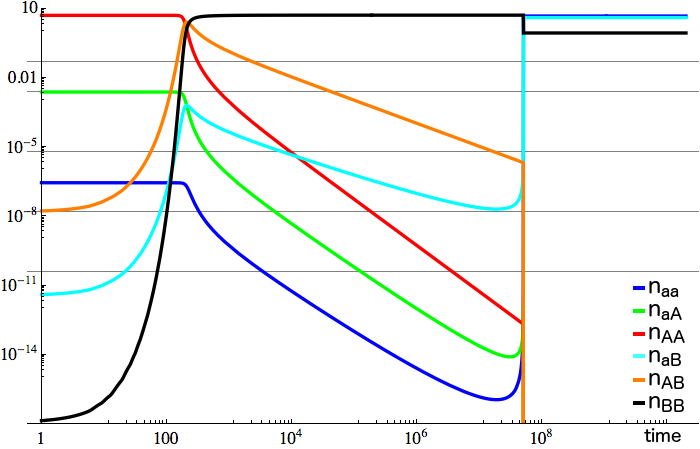}
		\includegraphics[width=0.46\textwidth]
		{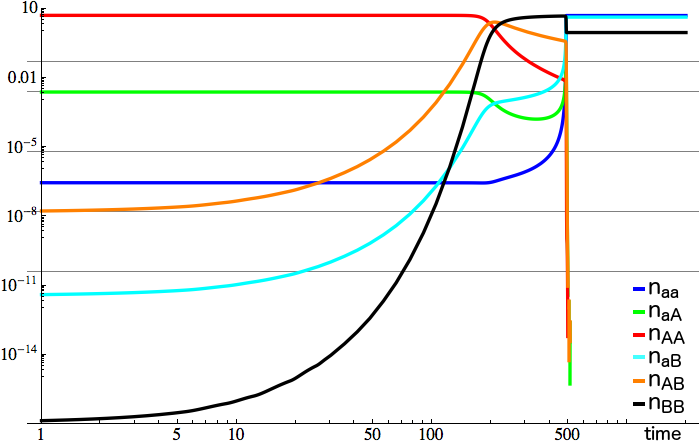}
		\includegraphics[width=0.46\textwidth]
		{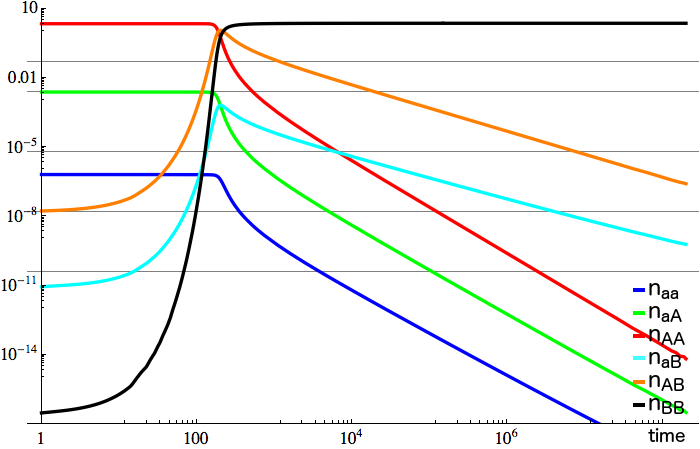}
		\includegraphics[width=0.46\textwidth]
		{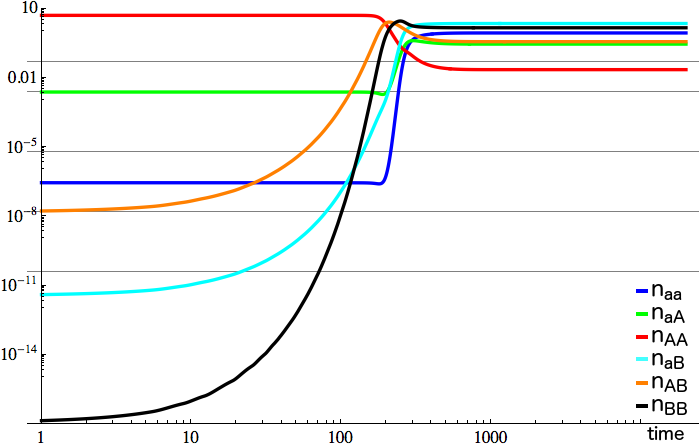}
		\includegraphics[width=0.46\textwidth]
		{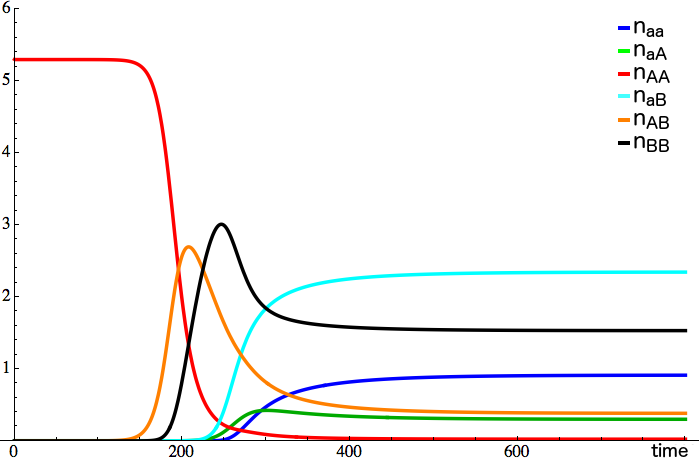}
		\caption{Numerical solution of the all-with-all deterministic process, loglog-plot:
			(up-left) for $f=6$, $\eta=0$ and $c_{aB}=0$, (up-right) for $f=6$, $\eta=0$ and $c_{aB}=0.05$;
			 (middle-left)  for $\eta=0.02$ and $c_{aB}=0$, (middle-right) for $\eta=0$, $c_{aB}=0$ and $f=3$;
			(down-left) for $\eta=0.17$ and $c_{aB}=0.925$, (down-right) rescaled individual plot for $\eta=0.17$ and $c_{aB}=0.925$.}
		\label{eta}
	\end{center}
\end{figure}

With  reasonable choices for $\eta$ and $c_{aB}$,  we end up in a 6-point-equilibrium where all 
populations coexist (see Figure \ref{eta} (down)). Observe, the $aB$ population can be bigger than the $AB$ population, because it gets an additional birth factor from the reproduction of individuals of genotype $aa$ with individuals of phenotype $B$ which outcompetes the birth of $AB$ individuals by reproduction of individuals of genotype $AA$ and of phenotype $B$.

\renewcommand\thesection{\Alph{section}}
\setcounter{section}{0}
\renewcommand{\theequation}{\Alph{section}.\arabic{equation}}

\section{Appendix}\label{appendix}

We collect in this appendix all the important definitions of times separating phases or subphases of the process, and provide a proof summary with all the implications. Recall Definition \ref{def-stop-times}, and Figures \ref{detsys} and \ref{zoom-in}. We write $N_\delta(x)$ for the $\delta$-neighbourhood of $x\in\R^6$.
We consider 
\begin{equation}
\D>\delta>\e_0>\e>0.
\end{equation}
and
	\begin{longtable}{ll}
		$T_1$ 	&\hspace{-0.2cm}$:=T_{\e_0}^{aa+aA+aB+AB+BB}$,\\
		$T_=$  &\hspace{-0.2cm}$ :=T^{aA=aB}$,\\
		$T_2$  &\hspace{-0.2cm}$:=T^{aA=\delta AA}\land T^{aB=\delta AB}\land T^{aa= aA\wedge aB}$,\\
		$T_3$  &\hspace{-0.2cm}$:=T^{aa}_{\bar n_a-\e^{\g/2}}$.
	\end{longtable}
where $\gamma:=2/(1+\eta\bar n_B-\Delta)$ is the order of magnitude of $n_{AA}$ at $T_2$ :  $n_{AA}(T_2)=\OO(\e^\gamma)$.\\
We summarise below the detailed structure of the proof (we abbreviate Lemma, Proposition and Theorem by L,P, and T respectively):
\begin{enumerate}
     		\item[Phase 1:]		$t\in[0,T_1]$ \\
     		Initial conditions :
     		$(n_{aa},n_{aA},n_{AA},n_{aB},n_{AB},n_{BB})=(\OO(\e ^2),\e ,\bar n_A\pm\OO(\e ), 0,\e ^3, 0)$.\\
     		With those initial conditions the following bounds hold:
     		\begin{align*}
     		n_{AA}\leq \bar n_A,\quad n_{BB}\leq\bar n_B, \quad n_{aB}\leq n_{AB}\quad\text{(Proposition \ref{aB<AB}) }
     		\end{align*}
     			
     	Here are the main steps of the proof of Propostion \ref{prop-bounds} and Corollary \ref{phase1}. Until the time $T_1$,
     	\begin{align*}
     	\text{(P\ref{aB<AB})}\Rightarrow\quad	&n_{BB}\leq n_{AB}^2 \quad\text{(1)}\\
     	\text{(1)}\Rightarrow\quad	&n_{aB}\leq n_{aA}n_{AB} \quad\text{(2)}\\
     	\text{(1)(2)}\Rightarrow\quad	&n_{aa}\leq n_{aA}^2 \quad\text{(3)}\\
     	\text{(1)(2)(3)}\Rightarrow\quad	&T_1= T^{AB}_{\e _0}=\OO(\log(\e _0/\e ^3)^{1/(\Delta-\OO(\e _0))})\quad\text{(Proposition \ref{prop-bounds})}\\
     	\\
     	\text{(P\ref{prop-bounds})}\Rightarrow\quad&\dot n_{AB}=\OO(\Delta) n_{AB} \quad\text{(Corollary \ref{phase1}(1))}\\
     	\text{(P\ref{prop-bounds})}\Rightarrow\quad&\bar n_A-\OO(\e) \leq \S_5\leq \bar n_A+2\Delta\e_0  \quad\text{(4)}\\
     		\text{(4)(P\ref{prop-bounds})}\Rightarrow\quad&  n_{aB}\leq\OO(\e^{1-\OO(\e_0)}\e_0),  
     																					n_{aA}\leq\OO(\e^{1-\OO(\e_0)}),\\
     																			&		n_{aa}\leq \OO(\e^{2-\OO(\e_0)}),
     																					n_{AA}=\bar n_A\pm \OO(\e_0) \quad\text{(Corollary \ref{phase1}(3)))}\\
     	\end{align*}

     		\item[Phase 2:] $t\in[T_1,T_2]$ \\
     		Initial conditions : \\
     		$(n_{aa},n_{aA},n_{AA},n_{aB},n_{AB},n_{BB})=(\OO(\e ^{2-\OO(\e_0)}),\OO(\e ^{1-\OO(\e_0)}),\bar n_A\pm\OO(\e_0 ), \OO(\e ^{1-\OO(\e_0)}\e_0),\e_0, \OO(\e_0^2)$.\\
     		As the process stays uniformly bounded in time, until $T_2$ we have 
     		\begin{equation*}
     		n_{aa},n_{aA},n_{aB}\leq\OO(\delta)\qquad(*)
     		\end{equation*}
     		\begin{align*}
     		(*)&\Rightarrow\quad	(n_{AA},n_{AB},n_{BB})=(n^{up}_{AA},n^{up}_{AB},n^{up}_{BB})+\OO(\delta) \quad\text{(Lemma \ref{3system})}\\
     		\text{L\ref{3system},}(*)&\Rightarrow\quad	\S_5=\bar n_B-\D n_{AA}/{c\bar n_B}+ {\OO(\D^2n_{AA})} \quad\text{(Proposition \ref{prop-S5})}\\
     		\text{P\ref{prop-S5},}(*)&\Rightarrow\quad	\dot\Sigma_{aA,aB}\geq-\OO(\D)\Sigma_{aA,aB} \quad\text{(Lemma \ref{lem-expS})}\\
     		\text{P\ref{prop-S5},}(*)&\Rightarrow\quad n_{aa}=\OO(\Sigma_{aA,aB}^2)	\Rightarrow T_2=T^{aA=\delta AA}\wedge T^{aB=\delta AB} \quad\text{(Lemma \ref{aa-S2})}\\
     		\text{P\ref{prop-S5},L\ref{aa-S2}}&\Rightarrow\quad \text{Until }T_=,\quad n_{aB}\leq n_{aA}=\OO(\e ) \quad\text{(Proposition \ref{prop-T1T=})}\\
     		\text{P\ref{prop-S5},}(*)&\Rightarrow\quad n_{aB}\leq\frac{n_{AB}+2n_{BB}+\sfrac{2\Delta}c}{n_{AB}+2n_{AA}}n_{aA} \quad\text{(Lemma \ref{ub-aB-aA})}\\
     		\text{P\ref{prop-S5},L\ref{aa-S2},P\ref{prop-T1T=},}&\Rightarrow\quad T_=<T_2 \quad\text{(Lemma \ref{AA=BB})}\\
     		\text{L\ref{ub-aB-aA}}&\Rightarrow\quad n_{AA}(T_=)\leq n_{BB}(T_=)+\OO(\Delta) \quad\text{(Lemma \ref{AA=BB})}\\
     		\text{P\ref{prop-S5},}(*)&\Rightarrow\quad \max_{t\in[T_1,T_2]}n_{AB}=:n_{AB}(T^{max}_{AB})=\bar n_B/2+\OO(\Delta) \quad\text{(Proposition \ref{maxAB})}\\
     		\text{P\ref{prop-S5},}(*)&\Rightarrow\quad n_{AA}(T^{max}_{AB}),n_{BB}(T^{max}_{AB})=\bar n_B/4+\OO(\Delta) \quad\text{(Proposition \ref{maxAB})}\\
     	    \text{L\ref{aa-S2},P\ref{prop-S5},}(*)&\Rightarrow\quad n_{aA}\leq n_{aB}\vee\OO(\e ) \quad\text{(Proposition \ref{phase2})}\\
     	    \text{L\ref{lem-expS},\ref{aa-S2},P\ref{prop-S5},}(*)&\Rightarrow\quad n_{aA}= \Sigma_{aA,aB} \OO(n_{AB}+2n_{AA}) \quad\text{(Lemma \ref{aA-S2})}\\
     	    \text{L\ref{aa-S2},P\ref{prop-S5},}(*)&\Rightarrow\quad  \dot{\Sigma}_{aA,aB}=n_{aA}(\eta n_{BB}-\OO(\Delta))\quad\text{(Proposition \ref{S2 bound})}\\
     	    \text{L\ref{aA-S2},P\ref{maxAB},\ref{S2 bound},}&\Rightarrow\quad  T_2=\OO\left(\e ^{1/(1+\eta\bar n_B-\Delta)}\right)\quad\text{(Theorem \ref{thm-T2})}\\
     	    \text{L\ref{aa-S2},T\ref{thm-T2}}&\Rightarrow\quad  T_2=T^{aA=\delta AA}\quad\text{(Propostion \ref{prop-T2})}\\
     		\end{align*}

     		\item[Phase 3:]$t\in[T_2,T_3]$\\
     			Initial conditions :\\
     		$(n_{aa},n_{aA},n_{AA},n_{aB},n_{AB},n_{BB})=(\OO(\e^2),\OO(\d\e^\g),\OO(\e ^\gamma),\OO(\delta \e ^{\gamma/2} ),\OO(\e ^{\gamma/2}),\bar n_B-\OO(\e ^{\gamma/2}))$.
     				\begin{align*}
     			\text{P\ref{S5-3Phase}}&\Rightarrow\quad	\S_5=\bar n_B\pm\OO(n_{AB}) \quad\text{(Lemma \ref{S5-3P})}\\
     			\text{L\ref{S5-3P}}&\Rightarrow\quad	n_{AA},n_{aA}\leq\OO(n_{AB}^2) \quad\text{(Lemma \ref{AB^2bounds})}\\
     			\text{L\ref{S5-3P},\ref{AB^2bounds},\ref{ABstay}}&\Rightarrow\quad	n_{AA},n_{aA}\geq\OO(n_{AB}^2) \quad\text{(Lemmas \ref{AB^2lb} and \ref{aA-AB^2})}\\
     			\text{L\ref{S5-3P},\ref{AB^2bounds},\ref{AB^2lb},\ref{aA-AB^2}}&\Rightarrow\quad	n_{AB}\leq\e ^{\gamma/10} \quad\text{(Lemmas \ref{ABstay} and \ref{lem-aa-exp-LB})}\\
     			\text{L\ref{AB^2bounds},\ref{aA-AB^2}}&\Rightarrow\quad	\dot n_{aa}\geq cn_{aa}\text{ with } c>0 \quad\text{(Lemma \ref{lem-aa-exp-LB})}
     				\end{align*}
 					
 			\item[Phase 4:]$t\in[T_3,\infty]$\\
 					Initial conditions : \\
 					$(n_{aa},n_{aA},n_{AA},n_{aB},n_{AB},n_{BB})=(\bar n_a-\e_0,\OO(\e ^{\gamma/5}), \OO(\e ^{\gamma/5}), \OO(\e ^{\gamma/10}), \OO(\e ^{\gamma/10}),\bar n_B\pm\OO(\e ^{\gamma/10}) ) $.\\
 					T\ref{thm-LCM}$\Rightarrow$ For $\eta<c\cdot 0.593644$, the fixed point $p_{aB}=(\bar n_a,0,0,0,0,\bar n_B)$ is stable and for $\e ,\e_0$ small enough, the system converges to it at speed $1/t$ (Theorem \ref{phase4}).
 					
\end{enumerate}

\bibliography{Library-Biomaths}
\bibliographystyle{abbrv}

\end{document}